\documentclass{amsart}

\usepackage{amsmath,amsthm,amsfonts,amssymb,bm,graphicx, mathrsfs}

\usepackage{comment}
\usepackage
{hyperref}

\hypersetup{colorlinks=true,citecolor=blue,linkcolor=blue,urlcolor=blue,
pdfstartview=FitH }

\theoremstyle{plain}
\newtheorem{lemma}{Lemma}
\newtheorem{theorem}[lemma]{Theorem}

\newtheorem{cor}[lemma]{Corollary}
\newtheorem{prop}[lemma]{Proposition}
\numberwithin{equation}{section}

\theoremstyle{definition}

\theoremstyle{remark}

\renewcommand{\geq}{\geqslant}
\renewcommand{\leq}{\leqslant}

\newcommand{\changed}[1]{{\color{black} #1}}

\usepackage{calc}
\newsavebox\CBox
\newcommand\hcancel[2][0.5pt]{%
  \changed{\ifmmode\sbox\CBox{$#2$}\else\sbox\CBox{#2}\fi%
  \makebox[0pt][l]{\usebox\CBox}%
  \rule[0.5\ht\CBox-#1/2]{\wd\CBox}{#1}}}

\makeatletter
\DeclareRobustCommand\widecheck[1]{{\mathpalette\@widecheck{#1}}}
\def\@widecheck#1#2{%
    \setbox\z@\hbox{\m@th$#1#2$}%
    \setbox\tw@\hbox{\m@th$#1%
       \widehat{%
          \vrule\@width\z@\@height\ht\z@
          \vrule\@height\z@\@width\wd\z@}$}%
    \dp\tw@-\ht\z@
    \@tempdima\ht\z@ \advance\@tempdima2\ht\tw@ \divide\@tempdima\thr@@
    \setbox\tw@\hbox{%
       \raise\@tempdima\hbox{\scalebox{1}[-1]{\lower\@tempdima\box
\tw@}}}%
    {\ooalign{\box\tw@ \cr \box\z@}}}
\makeatother

\newcommand{\C}{\mathbb{C}}

\renewcommand{\Bbb}{\mathbb}

\begin{document}

\author{Valentin Blomer}
\author{Andrew Corbett}
 
\address{Mathematisches Institut, Endenicher Allee 60, 53115 Bonn}
\email{blomer@math.uni-bonn.de}

\address{The Innovation Centre, University Of Exeter, Exeter, UK, EX4 4RN}

\email{A.J.Corbett@exeter.ac.uk}
  
\title{A symplectic restriction problem}

\thanks{The first author was supported in part by the DFG-SNF lead agency program grant BL 915/2-2}

\begin{abstract}

We investigate the norm of a degree 2 Siegel modular form of asymptotically large weight whose argument is restricted to the 3-dimensional subspace of its imaginary part. On average over  Saito--Kurokawa lifts an asymptotic formula is established that is consistent with the mass equidistribution conjecture on the Siegel upper half space as well as the Lindel\"of hypothesis for the corresponding Koecher--Maa{\ss}  series. The ingredients include 
a new relative trace formula for pairs of Heegner periods. 

\end{abstract}


\setcounter{tocdepth}{2}  \maketitle 

\maketitle

\section{Introduction} 

\subsection{Restriction norm of eigenfunctions} The question, `to what extent can the mass of a Laplace eigenfunction $\phi$ on a Riemannian manifold $X$ localize?', is a classical problem in analysis and is often quantified by upper (or lower) bounds for $L^p$-bounds for the restriction of $\phi$ to suitable submanifolds $Y \subseteq X$. The prototypical example is the case where $X$ is a surface and $Y$ is a curve, often a geodesic;
see e.g.\ \cite{B, BGT, CS, GRS, Ma2, R1, R2} and references therein.

 If $X$ is the quotient of a symmetric space by an arithmetic lattice (often called an arithmetic manifold), an additional layer of number theoretic structure enters. Not only can this be used to obtain stronger bounds \cite{Ma1}, but sometimes the period integrals can be expressed in terms of special values of $L$-functions. A typical such case is the $L^2$-restriction of a Maa{\ss} form for the group ${\rm SL}_{n+1}(\Bbb{Z})$ to its upper left $n$-by-$n$ block, which can be expressed as an average of central values of ${\rm GL}(n) \times {\rm GL}(n+1)$ Rankin--Selberg $L$-functions \cite{LY, LLY}. Other potential cases arise from the Gross-Prasad conjecture \cite{II}. 
 Often in this context, optimal restriction norm bounds are equivalent to the Lindel\"of conjecture on average over the spectral family of $L$-functions in question.
 

In this paper we consider certain Siegel modular forms $F$ for the symplectic group ${\rm Sp}_4(\Bbb{Z})$: we investigate the $L^2$-restriction of a Saito--Kurokawa lift $F(Z)$ on the 6-dimensional Siegel upper half space $\Bbb{H}^{(2)}$ to the 3-dimensional subspace where the argument $Z = X+iY$ is restricted to its imaginary part. This is  a very natural set-up, it is a direct higher-dimensional analogue of the classical problem of bounding a cusp form $f$ for  ${\rm SL}_2(\Bbb{Z})$ on the vertical geodesic, mentioned at   the beginning; cf.\  \cite[Section 7]{BKY}. While the latter leads, via Hecke's integral representation, directly to the corresponding $L$-function $L( f, s)$, things become much more involved for Siegel modular forms. 

We start by stating the corresponding period formula. For an even positive integer $k$ let $S_k^{(2)}$ denote the space of Siegel modular forms of degree 2 of weight $k$ for the group ${\rm Sp}_4(\Bbb{Z})$, equipped with the standard Petersson inner product; see Section \ref{sec2}. We think of $k$ as tending to infinity and are interested in asymptotic results with respect to $k$. We  restrict the argument of a cusp form $F\in S_k^{(2)}$  to its imaginary part $iY$ with $Y \in \mathcal{P}(\Bbb{R})$ where $\mathcal{P}(\Bbb{R})$, equipped with the measure $dY/(\det Y)^{3/2}$, is the set of positive definite symmetric $2$-by-$2$ matrices. Consider the \emph{restriction norm} 
\begin{equation}\label{NF}
\mathcal{N}(F) := \frac{\pi^2}{90} \cdot \frac{1}{\| F \|^2_2} \int_{{\rm SL}_2(\Bbb{Z}) \backslash \mathcal{P}(\Bbb{R})} |F(i Y)|^2 (\det Y)^{k} \frac{dY}{(\det Y)^{3/2}},
\end{equation}
where ${\rm SL}_2(\Bbb{Z})$ acts on $\mathcal{P}(\Bbb{R})$ by $\gamma \mapsto \gamma^{\top} Y \gamma$. 
Letting $\Bbb{H}$ denote the usual upper half plane,
we observe that
$${\rm SL}_2(\Bbb{Z}) \backslash \mathcal{P}(\Bbb{R}) \cong {\rm SL}_2(\Bbb{Z}) \backslash\Bbb{H} \times \Bbb{R}_{>0}$$
has infinite measure; see \eqref{product} below. The factor
$$\frac{\pi^2}{90} = \frac{3}{\pi} \cdot \frac{\pi^3}{270} = \frac{\text{vol}({\rm Sp}_4(\Bbb{Z}) \backslash \Bbb{H}^{(2)}) }{ \text{vol} ({\rm SL}_2(\Bbb{Z}) \backslash \Bbb{H})}$$
accounts for the fact that, in accordance with the literature, we choose the standard measures on ${\rm Sp}_4(\Bbb{Z}) \backslash \Bbb{H}^{(2)}$ and ${\rm SL}_2(\Bbb{Z}) \backslash \Bbb{H}$ which are not probability measures. 

Let $\Lambda$ denote a set of  spectral components of $L^2({\rm SL}_2(\Bbb{Z})\backslash \Bbb{H})$ consisting of the constant function $\sqrt{3/\pi}$, an orthonormal basis of  Hecke--Maa{\ss} cusp forms  and the Eisenstein series $E(., 1/2 + it)$ for $t \in \Bbb{R}$. The set $\Lambda$ is equipped with the counting measure on its discrete part and with the measure $dt/4\pi$ on its continuous part. We denote by $\int_{\Lambda}$ the corresponding combined sum/integral. For $F \in S_k^{(2)}$ and ${\tt u} \in \Lambda$ let $L(F \times {\tt u}, s)$ denote the Koecher--Maa{\ss}   series defined in \eqref{KM}.   This series has a functional equation featuring the gamma factors $G(F \times {\tt u}, s)$ as defined in \eqref{G-KM}, but has \emph{no} Euler product. 
The following proposition is proved in Section \ref{sec6}. 

\begin{prop}\label{prop1} For $F \in S_k^{(2)}$ with even $k$  we have  
 \begin{displaymath}
 \begin{split}
& \mathcal{N}(F)  = \frac{\pi^2}{90 }\cdot   \frac{1}{32} \cdot \frac{1}{ \| F \|^2_2}  \int_{-\infty}^{\infty}  \int_{\Lambda_{\text{\rm ev}}}  |G(F \times \overline{{\tt u}}, 1/2 + it)  L(F \times \overline{{\tt u}}, 1/2 + it) |^2  d{\tt u} \, dt , 
  \end{split}
  \end{displaymath}
  where $\Lambda_{\text{\rm ev}}$ denotes the set of all even ${\tt u} \in \Lambda$.  
  \end{prop}

An interesting subfamily of Siegel modular forms are the Saito--Kurokawa lifts $F_h$ (sometimes called the \emph{Maa{\ss} Spezialschar}) of half-integral weight modular forms $h \in S^+_{k-1/2}(4)$ in Kohnen's plus-space or equivalently their Shimura lifts $f_h \in S_{2k-2}$ (see Section \ref{sec2} for details). In this case, the Koecher--Maa{\ss}  series $L(F_h \times{\tt u}, s)$ roughly becomes  a Rankin--Selberg $L$-function of two half-integral weight cusp forms, namely of $h$ and the weight $1/2$ automorphic form whose Shimura lift equals ${\tt u}$; see Proposition \ref{prop2}  below and cf.\ \cite{DI}.  Of course, this series also has no Euler product. The convexity bound for $L$-functions along with trivial bounds implies
$$\mathcal{N}(F_h) \ll k^{2+\varepsilon},$$
whilst the statement 
\begin{equation}\label{1}
\mathcal{N}(F_h) \ll k^{\varepsilon}
\end{equation} 
would follow from the Lindel\"of hypothesis   for these $L$-functions. It should be noted, however, that in absence of an Euler product it is not expected that these $L$-functions satisfy the Riemann hypothesis, but one may still hope that the Lindel\"of hypothesis is true; see \cite{Ki} for some support of this conjecture. However, even if it is, then proving \eqref{1} appears to be far out of reach by current technology -- it corresponds to an average of size $k^{3/2}$ of a family of $L$-functions of conductor $k^8$. This is analogous to the genus 1 situation in which the $L^2$-restriction norm of a holomorphic cusp form of weight $k$ leads to an average of size $k^{1/2}$ of a family of $L$-functions of conductor $k^4$; see \cite[(1.12)]{BKY}. These problems belong to the hard cases where sharp bounds for the $L^2$-restriction norm imply very strong subconvexity bounds.

A different symplectic restriction problem  was treated in \cite{LiuY} and \cite{BKY}, where the argument $Z \in \Bbb{H}^{(2)}$
  of Saito-Kurokawa lifts was  restricted to the diagonal, a four-dimensional subspace of $\Bbb{H}^{(2)}$. The corresponding analogue of Proposition \ref{prop1}, due to Ichino \cite{Ic}, leads to an average of size $k$ of $L$-functions 
  of conductor $k^4$.


\subsection{The main result and mass equidistribution in higher rank} Fix a smooth,  non-negative test function $W$ with non-empty support in $[1, 2]$. Let $\omega := \int_1^2 W(x) x\, dx$ and consider
\begin{equation}\label{Nav}
\mathcal{N}_{\text{av}}(K) := \frac{1}{\omega} \cdot \frac{12}{  K^2} \cdot  \sum_{k \in 2\Bbb{N}} W\left(\frac{k}{K}\right) \sum_{h \in B_{k-1/2}^+(4)} \mathcal{N}(F_h)
\end{equation}
for a large parameter $K$ and a Hecke eigenbasis $B_{k-1/2}^+(4)$ of   $S_{k-1/2}^+(4)$. Note that $\dim S_{k-1/2}^+(4) \sim k/6$.  The first main result of this paper  is the following asymptotic formula.
\begin{theorem}\label{thm1} We have $\mathcal{N}_{\text{{\rm av}}}(K) =  4\log K + O(1)$ as $K \rightarrow \infty$.
\end{theorem}

This may be interpreted as an average version of the Lindel\"of hypothesis for twisted Koecher--Maa{\ss} series. This restriction problem, however, is structurally quite different from all previously considered restriction problems with connections to $L$-functions, as the period formula features $L$-functions that are not in the Selberg class and the restriction norm does not remain bounded. 

 More importantly, there is  a strong connection between Theorem \ref{thm1} and the  \textit{mass equidistribution conjecture} that we now explain. Let $g$ be  a test function on ${\rm Sp}_4(\Bbb{Z}) \backslash \Bbb{H}^{(2)}$. Then the (arithmetic) mass equidistribution conjecture for the Siegel upper half space states that 
$$ \frac{1}{\| F \|^2} \int_{{\rm Sp}_4(\Bbb{Z}) \backslash \Bbb{H}^{(2)}} g(Z) |F(Z)|^2 (\det Y)^{k} \frac{dX\, dY}{(\det Y)^{3}} \longrightarrow \int_{{\rm Sp}_4(\Bbb{Z}) \backslash \Bbb{H}^{(2)}} g(Z)  \frac{dX\, dY}{(\det Y)^{3}} $$
as $F$ traverses a sequence of Hecke--Siegel cusp forms of growing weight. While the corresponding statement for classical cusp forms of degree 1 was proved by Holowinsky and Soundararajan \cite{HS}, no such statement has been obtained for Siegel modular forms of higher degree (but see \cite{SV} for certain cases of the quantum unique ergodicity conjecture in higher rank).  Nevertheless, one may even go one step further and conjecture that the above limit holds when one restricts  the full space ${\rm Sp}_4(\Bbb{Z}) \backslash \Bbb{H}^{(2)}$ to a submanifold. In particular, one might conjecture that 
$$  \frac{\text{vol}({\rm Sp}_4(\Bbb{Z}) \backslash \Bbb{H}^{(2)}) }{\| F \|^2} \int_{{\rm SL}_2(\Bbb{Z}) \backslash \mathcal{P}(\Bbb{R})} g(Y) |F(iY)|^2 (\det Y)^{k} \frac{  dY}{(\det Y)^{3/2}} \longrightarrow\int_{{\rm SL}_2(\Bbb{Z}) \backslash \mathcal{P}(\Bbb{R})}  g(Y)  \frac{  dY}{(\det Y)^{3/2}} $$
holds. As the right hand side has infinite measure, we cannot simply replace $g$ with the constant function. This is precisely the reason why $\mathcal{N}_{\text{av}}(K)$ is unbounded as $K \rightarrow \infty$. However, since $F$ is a cusp form, the $L^2$-normalized and ${\rm Sp}_4(\Bbb{Z})$-invariant function $|F(iY)|^2 (\det Y)^{k}/\| F \|^2$ decays exponentially quickly if $Y$ is (in a  precise sense) very large or very small. So effectively $g$ may be restricted to the characteristic function of a compact set depending on $k$. We quantify this in Appendix \ref{appc} and show that, for such $g$,  the right hand side equals $$\text{vol}({\rm SL}_2(\Bbb{Z})\backslash \Bbb{H})\cdot 4 \log k + O(1).$$ In this case the previous   asymptotic reads
\begin{equation}\label{vol}
  \mathcal{N}(F) \sim   4\log k
  \end{equation}
as $k \rightarrow \infty$. The asymptotic \eqref{vol} is, of course, highly conjectural, and as mentioned  above  even the ordinary mass equidistribution conjecture (without restricting to a thin subset) is currently out of reach. Theorem \ref{thm1} provides an unconditional proof of \eqref{vol} on average  over Saito--Kurokawa lifts in agreement with the mass equidistribution conjecture. In particular, the constant 4 in Theorem \ref{thm1} is very relevant, and this constant has a story of its own. It is the outcome of several archimedean integrals, numerical values in period formulae and a gigantic Euler product whose special value can be expressed in terms of zeta values (cf.\ \eqref{euler}). In deducing its value we have corrected several numerical constants in the literature. We shall come back to this point in due course; see for instance the remark after Lemma \ref{lem3}. The authors would like to thank Gergely Harcos for useful and clarifying discussions in this respect. 

Theorem \ref{thm1} opens the door for several other related problems. The reader may wonder what happens for generic, i.e.\ non-CAP Siegel modular forms. Any reasonable spectral average would include at least the space of Siegel modular forms $S_k^{(2)}$ of weight $k$ which is of dimension $\sim ck^{3}$ for some constant $c$ (in fact $c = 1/8640$). This leads to a bigger  average than the one presently considered over about $k^2$ Saito-Kurokawa lifts. The starting point for the $L^2$-restriction norm of generic Siegel modular forms  is again the period formula in Proposition \ref{prop1}. Coupled with an approximate functional equation (as in Lemma \ref{approx1}), this is amenable to the Kitaoka-Petersson formula \cite{Kit} and an analysis along the lines of \cite{Bl2}. We hope to return to this interesting problem soon.

Whilst the proof of Theorem \ref{thm1} rests on many ingredients, to which we address in detail in the coming sections, there are a few highlights which may be of stand alone interest. We describe these in the remainder of the introduction.

\subsection{A relative trace formula for pairs of Heegner periods}

Here we focus on a novel trace formula of independent interest beyond its application in proving Theorem \ref{thm1}. Let 
$D$ be a discriminant, i.e.\ a non-square integer $\equiv 0, 1$ (mod 4). For a discriminant $D < 0$ let 
 $H_D \subseteq {\rm SL}_2(\Bbb{Z}) \backslash \Bbb{H}$ denote the set of all Heegner points; that is, the set of all $z = (\sqrt{|D|}i - B)/(2A)$ where $AX^2 + BXY + CY^2$ is a $\Gamma$-equivalence class of integral quadratic forms of discriminant $D = B^2 - 4AC$. For  a function $f : {\rm SL}_2(\Bbb{Z}) \backslash \Bbb{H} \rightarrow \Bbb{C}$ define the period
\begin{equation}\label{defP}
P(D; f) = \sum_{z \in H_D} \frac{f(z)}{\epsilon(z)} 
\end{equation}
where $\epsilon(z) \in \{1, 2, 3\}$ is the order of the stabilizer of $z$ in ${\rm PSL}_2(\Bbb{Z})$. Its counterparts for positive discriminants $D$ are periods over geodesic cycles. These periods are classical objects with myriad interwoven connections to half-integral weight modular forms, base change $L$-functions, quadratic fields and quadratic forms. An interesting special case is the constant function $f = 1$ in which case $P(D; 1) = H(D)$ is, by definition, the Hurwitz class number. 

With applications to the above mentioned symplectic restriction problem in mind, we are interested in pairs of Heegner periods in the spectral average
$$
\int_{\Lambda_{\text{ev}}} P(D_1; {\tt u}) \overline{P(D_2; {\tt u})} h(t_{\tt u}) d{\tt u}
$$
for a suitable test function $h$ and two discriminants $D_1, D_2 < 0$.  While pairs of geodesics have been studied in a few situations \cite{Pi1, Pi2, MMW}, to the best of our knowledge, nothing seems to be known about spectral averages of pairs of Heegner periods. Opening the   sums in the definition of $P(D_1; {\tt u})$ and $P(D_2; {\tt u})$, this can be expressed as a double sum of an automorphic kernel
$$\sum_{z_1 \in H_{D_1}}\sum_{z_2 \in H_{D_2}}  \frac{1}{\epsilon(z_1)\epsilon(z_2)}\sum_{\gamma \in \Gamma} k(z_1, \gamma z_2)$$
in the usual notation which resembles the set-up of a relative trace formula. However, the standard methods in this situation (e.g.\ \cite{Go}) do not easily apply here as the stabilizers of $z_1$ and $z_2$ are essentially trivial. We thus take a different approach to establish the following relative trace formula for which we need some notation. For   $n >0$ and $t \in \Bbb{R}$ let
\begin{equation*}
W_t( n) := \frac{1}{2\pi i} \int_{(2)} \frac{\Gamma(\frac{1}{2}(\frac{1}{2} + s + 2it))\Gamma(\frac{1}{2}(\frac{1}{2} + s -2 it))}{\Gamma( \frac{1}{4}   + it)\Gamma( \frac{1}{4} -it)\pi^s} e^{s^2}  n^{-s} \frac{ds}{s}.
\end{equation*}
For $t\in \Bbb{R}$, $x > 0$ and $\kappa\in \Bbb{R}$ let 
\begin{equation}\label{defF}
F(x, t, \kappa) = J_{it}(x) \cos(\pi\kappa/2 - \pi i t/2) -  J_{-it}(x) \cos(\pi\kappa/2 + \pi i t/2)
\end{equation}
where $J_{it}(x)$ is the Bessel function. Finally, for $\kappa \in \Bbb{Z} + 1/2$, $n, m \in \Bbb{Z}$ and $c \in \Bbb{N}$ define the modified Kloosterman sums
\begin{equation}\label{defKlo}
K^+_{\kappa}(m, n, c) =  \sum_{\substack{d\, (\text{mod }c)\\ (d, c) = 1}} \epsilon^{2\kappa}_d \left(\frac{c}{d}\right) e\left(\frac{md + n\bar{d}}{c}\right) \cdot \begin{cases} 0, & 4 \nmid c,\\ 2, & 4 \mid c, \, 8 \nmid c,\\ 1, & 8 \mid c, \end{cases}
\end{equation}
where 
\begin{equation}\label{epsd}
   \epsilon_d = \begin{cases} 1, & d \equiv 1\, ({\rm mod }\,4), \\   i, & d \equiv 3\, ({\rm mod }\, 4).\end{cases}
   \end{equation} 
Note that $K^+_{\kappa}(n, m, c)$ is symmetric in $m$ and $n$ and $2$-periodic in $\kappa$. They satisfy the Weil-type bound
\begin{equation}\label{weil}
K_{\kappa}^+(m, n, c) \ll c^{1/2 + \varepsilon}(m, n, c)^{1/2},
\end{equation}
see e.g.\ \cite[Lemma 4]{Wa} in the case $n=m$, the general case being analogous. In order to simplify the notation we assume that $D_1, D_2$ are fundamental discriminants. In Section \ref{secproof2}  we state the general version for arbitrary negative discriminants. 

\begin{theorem}\label{thm2} Let $\Delta_1, \Delta_2$ be negative fundamental discriminants and let $h$ be an even function, holomorphic in $|\Im t | < 2/3$ with $h(t) \ll (1+|t|)^{-10}$. Then
\begin{displaymath}
\begin{split}
&\frac{1}{|\Delta_1\Delta_2|^{1/4}}\int_{\Lambda_{\text{{\rm ev}}}} P(\Delta_1; {\tt u}) \overline{P(\Delta_2; {\tt u})} h(t_{\tt u}) d{\tt u} = \frac{3}{\pi} \frac{H(\Delta_1)H(\Delta_2)}{|\Delta_1\Delta_2|^{1/4}} h(i/2)\\
& \quad +   \int_{-\infty}^{\infty} \Big|\frac{\Delta_1\Delta_2}{4}\Big|^{ it/2}   \frac{ \Gamma(-\frac{1}{4} + \frac{it}{2})   e^{(1/2- it)^2}}{ \sqrt{8\pi} \Gamma(\frac{1}{4} + \frac{i t}{2})}  \frac{L( \chi_{\Delta_1}, 1/2 + it)L(\chi_{\Delta_2}, 1/2 + it)}{ \zeta(1+ 2it)} h(t) \frac{dt}{4\pi}\\
& \quad + \delta_{\Delta_1= \Delta_2} \sum_m\frac{\chi_{\Delta_1}(m)}{m} \int_{-\infty}^{\infty}  W_{t}(m) h(t) t \tanh(\pi t)  \frac{dt}{4\pi^2} \\
&\quad + e(3/8) \sum_{n,  c, m}    \frac{  K_{3/2}^+(|\Delta_1| n^2, |\Delta_2|, c)\chi_{\Delta_1}(m)}{ n^{1/2}cm}   \int_{-\infty}^{\infty}  \frac{F(4\pi  n\sqrt{|\Delta_1\Delta_2|}/c, t, 1/2)}{\cosh(\pi t)}   h(t) W_{t}(nm) t \frac{dt}{\pi}.
\end{split}
\end{displaymath}
\end{theorem}
The experienced reader will spot the strategy of the proof from the shape of the formula: A Katok-Sarnak-type formula translates $P(\Delta; {\tt u})$ into a product of a first and a $\Delta$-th half-integral weight Fourier coefficient. In this way, a pair of two Heegner periods becomes a product of \emph{four} half-integral weight Fourier coefficients. A quadrilinear form of half-integral weight Fourier coefficients is not directly amenable to any known spectral summation formula, but we can use a Waldspurger-type formula a second time,   now in the other direction, to translate the two first coefficients into a central $L$-value. This $L$-value can be written explicitly as a sum of Hecke eigenvalues by an approximate functional equation. We can now use the correspondence between half-integral and integral weight forms  a \emph{third} time, namely by combining the Hecke eigenvalues into the half-integral weight coefficients by means of metaplectic Hecke relations. Finally, the Kuznetsov formula for the Kohnen plus space provides the desired geometric evaluation of the relative trace. This particular version of the  Kuznetsov formula is also new and will be stated and proved in Section \ref{summation}. 


\subsection{Mean values of \texorpdfstring{$L$}{L}-functions}
We highlight another ingredient of independent interest. This is a hybrid Lindel\"of-on-average bound for central values of twisted $L$-functions, its proof is deferred to Section \ref{proof}.
\begin{prop}\label{Lfunc} Let $\mathcal{D}, \mathcal{T} \geq 1$ and $\varepsilon > 0$. For a fundamental discriminant $\Delta$ let $\chi_{\Delta} = (\frac{\Delta}{.})$ be the Jacobi--Kronecker symbol.

{\rm (a)} We have 
$$\sum_{t_{ u} \leq \mathcal{T}}\alpha(u)  \sum_{\substack{|\Delta| \leq \mathcal{D}\\ \Delta \text{ {\rm  fund.  discr.}}}}  L({u} \times \chi_{\Delta}, 1/2)   \ll \Big(\sum_{t_{ u} \leq \mathcal{T}}|\alpha(u)|^2\Big)^{1/2} (\mathcal{T}\mathcal{D})^{1+\varepsilon}$$
where the sum is over an orthonormal basis of Hecke--Maa{\ss} cusp forms ${u}$ with spectral parameter $t_{u}$,  and $\alpha(u)$ is any sequence of complex numbers, indexed by Maa{\ss} forms. 

{\rm (b)} We have
$$\int_{-T}^T \alpha(t)  \sum_{\substack{|\Delta| \leq \mathcal{D}\\ \Delta \text{ {\rm  fund.  discr.}}}}  |L( \chi_{\Delta}, 1/2 + it)|^2  dt \ll \Big(\int_{-T}^T|\alpha(t)|^2 dt\Big)^{1/2} (\mathcal{T}^{1/2}\mathcal{D})^{1+\varepsilon}$$
for an arbitrary function $\alpha : [-T, T] \rightarrow \Bbb{C}$. 
\end{prop} 

The proof uses, among other things, the spectral large sieve of Deshouillers-Iwaniec \cite{DesIw} and  Heath-Brown's large sieve for quadratic characters   \cite{HB}. 
Note that   $L({ u} \times \chi_{\Delta}, 1/2) \geq 0$ is  non-negative \cite[Corollary 1]{KS}. The key point here is that there is complete uniformity in $\mathcal{T}$ and $\mathcal{D}$. We give an immediate application. Let us choose $\alpha(u) = L(u, 1/2)$ and note that 
  $L(u, 1/2) L(u \times \chi_{\Delta}, 1/2) = L({\rm BC}_{K}(u), 1/2)$ where the right hand side is the base change $L$-function to $K = \Bbb{Q}(\sqrt{\Delta})$. We can now  
use a standard mean value bound for $L(u, 1/2)$, e.g.\  \cite[Theorem 3]{Iw} to conclude  
\begin{cor}\label{Lfunc-cor}
For $\mathcal{T}, \mathcal{D} \geq 1$ and $\varepsilon > 0$ we have 
$$\sum_{t_{ u} \leq \mathcal{T}}  \sum_{\substack{\deg K/\Bbb{Q} = 2\\ |\text{{\rm disc}}(K)| \leq \mathcal{D}}}  L({\rm BC}_{K}(u), 1/2)   \ll   (\mathcal{T}^2\mathcal{D})^{1+\varepsilon}$$
where the first sum runs over a basis of  Hecke--Maa{\ss} cusp forms $u$ with spectral parameter $t_u \leq \mathcal{T}$. 
\end{cor}
Again this bound is completely uniform and best-possible in the $\mathcal{D}$ and $\mathcal{T}$ aspect. We give another interpretation of Proposition \ref{Lfunc}(a). For odd $u$, the root number of $L(u  \times \chi_{\Delta}, s)$  is $-1$ (see \cite[Lemma 2.1]{BFKMMS}), so the central value vanishes. For even $u$ the central $L$-values $L(u  \times \chi_{\Delta}, 1/2)$, as in \eqref{BarMao} below, are proportional to squares of  Fourier coefficients $b_v(\Delta)$ of weight 1/2 Maa{\ss} forms $v$ in Kohnen's subspace for $\Gamma_0(4)$, normalized as in \eqref{four-half}.  We refer to Section \ref{sec4} for the relevant definitions. In particular, for the usual choice of the Whittaker function  the normalized Fourier coefficient $\tilde{b}_v(\Delta) = e^{-\pi |t_v|/2}  |t_v|^{ \text{sgn}(\Delta)/4} |\Delta|^{1/2}   b_v(\Delta) $ is of size one on average.   In this way we conclude bounds for linear forms in half-integral weight \emph{Rankin--Selberg coefficients}:
\begin{equation}\label{interpret}
\sum_{1/4 \leq t_{ v} \leq \mathcal{T}}\alpha(v)  \sum_{\substack{|\Delta| \leq \mathcal{D}\\ \Delta \text{ {\rm  fund.  discr.}}}} |\tilde{b}_v(\Delta)|^2   \ll \Big(\sum_{t_{ v} \leq \mathcal{T}}|\alpha(v)|^2\Big)^{1/2} (\mathcal{T}\mathcal{D})^{1+\varepsilon},
\end{equation}
where the $v$-sum runs over an $L^2$-normalized  Hecke eigenbasis of non-exceptional weight 1/2 Maa{\ss} forms in Kohnen's subspace for $\Gamma_0(4)$ with spectral parameter $t_v$. We refer to the remark after the proof of Lemma \ref{lem3} for more details. 




\subsection{Organization of the paper}\label{15}

Section \ref{sec2} -- \ref{summation} prepare the stage and compile all necessary automorphic information. New results include versions of the half-integral Kuznetsov formula and a Voronoi formula for Hurwitz class numbers. Proposition \ref{prop1}, Theorem \ref{thm2}, and Proposition \ref{Lfunc}  are proved in Sections \ref{sec6}, \ref{secproof2},  \ref{proof} respectively. This is followed by an interlude on the analysis of oscillatory integrals. In the remainder we complete the proof of Theorem \ref{thm1}. In Section \ref{weakversion} we first prove an upper bound $\mathcal{N}_{\text{av}}(K) \ll K^{\varepsilon}$ by a preliminary argument. This will be useful to control certain auxiliary variables and error terms later. Due to several applications of certain spectral summation formulae, we have various diagonal and off-diagonal terms. Section \ref{diag-diag} treats the total diagonal term that extracts the leading term $4\log K$ in Theorem \ref{thm1}. Sections \ref{diag-off} and \ref{off-off} deal with the diagonal off-diagonal and the off-off-diagonal term. 

\subsection{Common notation}\label{16}

For $ c\not= 0$ we extend the Jacobi-Symbol $\chi_c(d)=  (\frac{c}{d})$ for positive odd integers $d  $ to all   integers $d\not= 0$ as the completely multiplicative function defined by $\chi_c(-1) = \text{sign}( c)$  and $\chi_c(2) = 1$ if $c \equiv 1 \, (\text{mod }8)$, $\chi_c(2) = -1$ if $c \equiv 5 \, (\text{mod }8)$, $\chi_c(2) = 0$ if $c$ is even. The value of $\chi_c(2)$ remains undefined only if $c \equiv 3 \, (\text{mod }4)$. 

We call an integer $D \in \Bbb{Z} \setminus \{0\}$ a discriminant if $D \equiv 0$ or $1$ (mod 4). Every discriminant $D$ can uniquely be written as $D = \Delta f^2$ for some $f\in \Bbb{N}$ and some fundamental discriminant $\Delta$ (possibly $\Delta = 1$). For each discriminant $D$, the map  $\chi_D$ is a quadratic character of modulus $|D|$ that is induced by the character $\chi_{\Delta}$ corresponding to the field $\Bbb{Q}(\sqrt{\Delta})$. (If $\Delta = 1$, then $\chi_{\Delta}$  is the trivial character.)
Throughout, the letters $D$ and $\Delta$ are always reserved for  discriminants resp.\   fundamental discriminants, usually negative.

The letter $\Gamma$ is used for the gamma function and also for the group $\Gamma = {\rm SL}_2(\Bbb{Z})$; confusion will not arise. We write $\overline{\Gamma} = {\rm PSL}_2(\Bbb{Z})$.

For $a, b \in \Bbb{N}$ we write $a \mid b^{\infty}$ to mean that all prime divisors of $a$ divide $b$. We also write $(a, b^{\infty}) = a/a_1$ where $a_1$ is the largest divisor of $a$ that is coprime to $b$.  
We use the usual exponential notation $e(z):=e^{2\pi i z}$ for $z\in\C$. The letter $\varepsilon$ denotes an arbitrarily small positive constant, not necessarily the same at every occurrence. 
The Kronecker symbol $\delta_{S}$ takes the value 1 if the statement $S$ is true and $0$ otherwise. The notation $\int_{(\sigma)}$ denotes a complex contour integral over the vertical line with real part $\sigma$. We use the usual Vinogradov symbols $\ll$ and $\gg$, and we use  $\asymp$ to mean both $\ll$ and $\gg$. We always assume that the number $K$ in Theorem \ref{thm1} is sufficiently large. \\

\section{Holomorphic forms of degree one and two}\label{sec2}

For a positive integer $k$ let  $S^+_{k-1/2}(4)$ denote   Kohnen's plus \cite{Ko1} space of holomorphic cusp forms 
 of weight $k-1/2$ and level 4. These have a Fourier expansion of the form
\begin{equation}\label{four}
h(z) = \sum_{(-1)^k n \equiv 0, 3\, (\text{mod 4})} c_h(n)   e(nz)
\end{equation}
and form a finite-dimensional Hilbert space with the inner product 
$$\langle h_1, h_2 \rangle =  \int_{\Gamma_0(4) \backslash \Bbb{H}} h_1(z) \overline{h_2(z)} y^{k-1/2} \frac{dx\, dy}{y^2}.$$
This space is isomorphic (as a module of the Hecke algebra)  to the space $S_{2k-2}$ of holomorphic cusp forms of weight $2k-2$ and level 1 \cite[Theorem 1]{Ko1}. 
We denote by $f_h \in S_{2k-2}$ the (unique up to scaling)  image of a newform $h \in S^+_{k-1/2}(4)$. The Hecke algebra on $S^+_{k - 1/2}$ is generated by the operators $T(p^2)$, $p$ prime, and for $p=2$ we follow Kohnen's definition \cite[p.\ 250]{Ko1}  of $T(4)$ that allows a uniform treatment of all primes including $p=2$. 
  If $\lambda(p)$ are the  Hecke eigenvalues of $f_h$ (normalized so that the Deligne's bound reads $|\lambda(p)| \leq 2$), then
\begin{equation}\label{fourier-relation}
\lambda(p) c_h(n) = p^{3/2 - k} c_h(p^2 n) + p^{-1/2} \chi_{ (-1)^{k+1} n}(p) c_h(n)  + p^{k - 3/2} c_h(n/p^2)
\end{equation}
for all primes $p$   with the convention $c_h(x) = 0$ for $x \not \in\{n \in \Bbb{N} \mid (-1)^k n \equiv 0, 3\, (\text{mod 4}) \}$.  Iterating this formula gives
\begin{equation}\label{fourier-relation1}
\lambda(m) c_h(n) = \sum_{\substack{d_1 \mid d_2 \mid m\\ (d_1d_2)^2 \mid m^2n}} \left(\frac{d_1}{d_2} \right)^{1/2} \chi_{  (-1)^{k+1} n}(d_1d_2) c_h\left( \frac{m^2}{(d_1d_2)^2}n\right) \left(\frac{m}{d_1d_2}\right)^{3/2 - k} 
\end{equation}
for squarefree $m \in \Bbb{N}$.

The space $S^{+}_{k-1/2}(4)$ can be characterized as an eigenspace of a certain operator acting on the space $S_{k-1/2}(4)$ of all holomorphic cusp forms of weight $k-1/2$ and level 4 \cite[Proposition 2]{Ko1}. It possesses Poincar\'e series $P^+_n \in S^+_{k-1/2}(4)$ satisfying the usual relation \cite[(4)]{Ko2}
$$\langle h, P^+_n\rangle =   \frac{\Gamma(k-3/2)}{(4\pi n)^{k-3/2}} c_h(n)$$
for all $(-1)^k n \equiv 0, 3\, (\text{mod 4})$ and all $h \in S^+_{k-1/2}(4)$ with Fourier expansion \eqref{four}. These Poincar\'e series are the orthogonal projections of the Poincar\'e series $P_n \in S_{k-1/2}(4)$ onto $S^{+}_{k-1/2}(4)$ and their Fourier coefficients are computed explicitly in \cite[Proposition 4]{Ko2}. This gives us the following Petersson formula for Kohnen's plus space. 

\begin{lemma}\label{lem1} 
Let $k \geq 3$ be an integer, $\kappa = k-1/2$.  Let $\{h_j\}$ be an orthogonal basis of $S^+_{\kappa}(4)$ with Fourier coefficients $c_j(n)$ as in \eqref{four}. Let $n, m$ be positive integers with $(-1)^kn, (-1)^km \equiv 0, 3\,  ({\rm mod } \,\,4)$. Then 
\begin{displaymath}
\begin{split}
&\frac{\Gamma(\kappa - 1)}{(4\pi)^{\kappa - 1}} \sum_j \frac{c_j(n) \overline{c_j(m)}}{\| h_j \|^2 (  \sqrt{nm})^{\kappa-1}} = \frac{2}{3}\left(\delta_{m=n} + 2 \pi e(-\kappa/4) \sum_{c}   \frac{K^+_{\kappa}(n, m, c)}{c} J_{\kappa - 1}\left(\frac{4\pi \sqrt{mn}}{c}\right)\right).
\end{split}
\end{displaymath}
\end{lemma}

For a positive   integer $k$ we denote by $S^{(2)}_k$ the space of Siegel cusp forms of degree 2 of weight $k$ for the symplectic group ${\rm Sp}_4(\Bbb{Z})$ with Fourier expansion
\begin{equation}\label{four1}
F(Z) = \sum_{T \in \mathcal{P}(\Bbb{Z})} a(T) e(\text{tr}(TZ))
\end{equation}
for $Z = X + iY \in \Bbb{H}^{(2)}$ on the Siegel upper half plane,  
where $\mathcal{P}(\Bbb{Z})$ is the set of symmetric, positive definite 2-by-2 matrices   with integral diagonal elements and half-integral off-diagonal elements. This is a finite-dimensional Hilbert space with respect to the inner product 
 \begin{equation*}
 \langle F, G \rangle  = \int_{{\rm Sp}_4(\Bbb{Z})\backslash \Bbb{H}^{(2)}} F(Z) \overline{G(Z)} (\det Y)^k \frac{dX \, dY}{(\det Y)^3}.
 \end{equation*}
 The norm $\| F \|^2_2$ can be expressed in terms of the adjoint $L$-function at $s=1$, a connection that has only very recently been proved by Chen and Ichino \cite{ChIc}.  There is a special family of Siegel cusp forms that are derived from elliptic cusp forms (Saito--Kurokawa lifts or Maa{\ss} Spezialschar). Let $k$ be  an even positive integer. Let $h \in S^+_{k-1/2}(4)$ be  a Hecke eigenform of weight $k-1/2$  with Fourier expansion as in \eqref{four}, and let 
 $f_h \in S_{2k-2}$ denote the corresponding Shimura lift. 
The Saito--Kurokawa lift associates to $h$ (or $f_h$) a Siegel cusp form $F_h$ of weight $k$ for ${\rm Sp}_4(\Bbb{Z})$ with Fourier expansion \eqref{four1}, where
 \begin{equation}\label{four2}
 a(T) =    \sum_{d \mid (n, r, m)}d^{k-1} c \left(\frac{4 \det T}{d^2}\right), \quad T = \left(\begin{matrix} n & r/2\\ r/2 & m\end{matrix}\right) \in \mathcal{P}(\Bbb{Z}),
 \end{equation}
see e.g.\ \cite[\S 6]{EZ}. 
If   $L (f_h, s)$ is the standard $L$-function  of $f_h$ (normalized so that the functional equation sends $s$ to $1-s$), then the norms of $F_h$ and  $h$ 
are related by
 (\cite[p.\ 551]{KoSk}, \cite[Lemma 4.2 \& 5.2 with $M=1$]{Br})
 \begin{equation}\label{normF}
 \| F_h \|^2 =  \| h \|^2 \frac{\Gamma(k) L(f_h, 3/2)}{4\pi^k}. 
 \end{equation}
\emph{Remarks:} 1) Note that  the formula three lines after (4) in \cite{KoSk} is off by a factor of 2. \\
2) This inner product relation was generalized to Ikeda lifts in \cite{KK}. \\
3) For future reference we note that for $\Re s > 1$ we have
\begin{equation}\label{1overL}
\frac{1}{L(f_h, s)} = \prod_p \left(1 - \frac{\lambda(p)}{p^s} + \frac{1}{p^{2s}}\right) = \sum_{(n, m) = 1} \frac{\lambda(n) \mu(n)\mu^2(m)}{n^sm^{2s}}
\end{equation}
if $\lambda(n)$ denotes the $n$-th Hecke eigenvalue of $f_h$. 

\section{Non-holomorphic automorphic forms}\label{sec4}

We recall the spectral decomposition of $L^2(\Gamma\backslash \Bbb{H})$ with $\Gamma = {\rm SL}_2(\Bbb{Z})$, consisting of   the constant function $u_0 = \sqrt{3/\pi}$, a countable orthonormal basis  $\{u_j, j = 1, 2, \ldots \}$ of Hecke--Maa{\ss} cusp forms 
and Eisenstein series $E(., 1/2 + it)$, $t \in \Bbb{R}$.  As in the introduction we call the collection of these functions $\Lambda$ and the subset of even members $\Lambda_{\text{ev}}$, and we write $\int_{\Lambda}$ resp.\ $\int_{\Lambda_{\text{ev}}}$ for the corresponding spectral averages.  We also introduce the notation  
$\int^{\ast}_{\Lambda_{\text{ev}}} d{\tt u}$ for a spectral average without the residual spectrum which in this case consists only of the constant function.  We use the general notational convention that an element in $\Lambda$ or $\Lambda_{\text{ev}}$ is denoted by ${\tt u}$ while a cusp form is usually denoted by $u$.

For $t \in \Bbb{R}$ let $U_t^{\text{ev}}$ denote the space of even weight zero Maa{\ss} cusp forms  for $\Gamma$ with Laplacian eigenvalue $1/4 + t^2$. It is equipped with the inner product 
\begin{equation}\label{inner-maass}
\langle u_1, u_2 \rangle = \int_{\Gamma \backslash \Bbb{H}} u_1(z) \overline{u_2(z)} \frac{dx \, dy}{y^2}.
\end{equation}
We write the Fourier expansion as
$$u(z) = \sum_{n \not= 0} a(n) W_{0, it}(4 \pi |n| y) e(nx)$$
with $a(-n) = a(n)$, 
where $W_{0, it}(4\pi y) = 2 y^{1/2} K_{it}(2\pi y)$ is the Whittaker function.  The Hecke operators $T(n)$,  normalized as in \cite[(1.1)]{KS}, act on $U_t^{\text{ev}}$ as a commutative family of normal operators. We call $t = t_u$ the spectral parameter of $u$. The Eisenstein series $$E(z, s) =  \sum_{\gamma \in \overline{\Gamma}_{\infty} \backslash \overline{\Gamma}} (\Im \gamma z)^s = \sum_{\substack{ (c, d)\in \Bbb{Z}^2/\{\pm 1\}\\ {\rm gcd}(c, d) = 1}} \frac{y^s}{|cz+d|^{2s}}$$ for $\overline{\Gamma} = {\rm PSL}_2(\Bbb{Z})$ are eigenfunctions of all $T(n)$ with eigenvalue
\begin{equation}\label{defrho}
\rho_s(n) := \sum_{ab = n} (a/b)^{s - 1/2} = n^{s - 1/2} \sigma_{1 - 2s}(n), \quad \sigma_s(n) = \sum_{d \mid n} d^s,
\end{equation}
and an eigenfunction of the Laplacian with eigenvalue $s(1-s)$. We call $(s - 1/2)/i$ the spectral parameter of $E(., s)$. If ${\tt u}$ is   an Eisenstein series or a Hecke--Maa{\ss} cusp form with Hecke eigenvalues $\lambda_{\tt u}(n)$ we define the corresponding $L$-function $L({\tt u}, s) = \sum_n \lambda_{\tt u}(n) n^{-s}$. In particular 
\begin{equation}\label{eisen-L}
   L(E(., 1/2 + it), s) = \zeta(s + it) \zeta(s - it).
   \end{equation} 

If $u \in U_t^{\text{ev}}$ is a cuspidal  Hecke eigenform with eigenvalues $\lambda(n)$, then $n^{1/2}a(n) = a(1) \lambda(|n|)$,  and  by Rankin--Selberg theory (and \cite[6.576.4]{GR} with $a = b = 4\pi$, $\nu = \mu = it$)
\begin{equation}\label{norm}
\begin{split}
\| u \|^2 & = \frac{\pi}{3} \underset{s=1}{\text{res} }\langle |u|^2, E(., s) \rangle = \frac{2\pi}{3} \underset{s=1}{\text{res} } \sum_{n > 0} \frac{|a(n)|^2}{|n|^{s-1}} \int_0^{\infty} W_{0, it}(4 \pi y)^2 y^{s-2} dy\\
&=  \frac{2\pi |a(1)|^2}{3} \underset{s=1}{\text{res} } \sum_{n > 0} \frac{|\lambda(n)|^2}{|n|^{s}}  \frac{\pi^{1/2} \Gamma(s/2) \Gamma(s/2 - it)\Gamma(s/2 + it)}{(2\pi)^s \Gamma((1 + s)/2)}  
= \frac{2  |a(1)|^2 L(\text{sym}^2 u, 1)}{\cosh(\pi t)}. 
\end{split}
\end{equation}
We recall the Kuznetsov formula \cite{Ku} and combine directly the ``same sign'' and the ``opposite sign'' formula to obtain a version for the even part of the spectrum. The conversion between Hecke eigenvalues and Fourier coefficients in the cuspidal case follows from \eqref{norm}. 
\begin{lemma}\label{kuz-even}
Let $n, m \in \Bbb{N}$. Let $h$ be an even holomorphic function in $|\Im t| \leq 2/3$ with $h(t) \ll (1 + |t|)^{-3}$.  For non-constant ${\tt u} \in \Lambda$ let $\mathcal{L}({\tt u}) = L(\text{{\rm sym}}^2 {\tt u}, 1)$ if ${\tt u}$ is cuspidal and\footnote{Note that the measure $d{\tt u}$ is $dt/4\pi$ in the Eisenstein case which explains the factor $1/2$ in the definition of $\mathcal{L}(E(., 1/2 + it))$.}  $\mathcal{L}({\tt u})  = \frac{1}{2}|\zeta(1  + 2it)|^2$ if ${\tt u} = E(., 1/2 + it)$ is Eisenstein\footnote{with the obvious interpretation in \eqref{kuz-even-form} for $t=0$}. Then
\begin{equation}\label{kuz-even-form}
\begin{split}
\int_{\Lambda_{\text{{\rm ev}}}}^{\ast} \frac{\lambda_{\tt u}(n) \lambda_{\tt u}(m) }{\mathcal{L}({\tt u})} h(t_{\tt u})d{\tt u}= &\delta_{n= m} \int_{-\infty}^{\infty} h(t) t \tanh(\pi t) \frac{dt}{4\pi^2} \\
&+   \sum_c \frac{S(m, n, c)}{c}h^*\Big( \frac{\sqrt{nm}}{c}\Big)+  \sum_c \frac{S(m, -n, c)}{c} h^{\ast\ast}\Big( \frac{\sqrt{nm}}{c}\Big)
\end{split}
\end{equation} 
where
\begin{equation}\label{hast}
 h^{\ast}(x) =   2i \int_{-\infty}^{\infty} \frac{J_{2it}(4 \pi x)}{\sinh(\pi t)} h(t)  t \tanh(\pi t) \frac{dt}{4\pi}, \quad h^{\ast\ast}(x) = \frac{4}{\pi}\int_{-\infty}^{\infty}  K_{2it}(4 \pi x)\sinh(\pi t) h(t)  t \frac{dt}{4\pi}.
\end{equation}
\end{lemma}

We turn to half-integral weight forms. Let $V_t^+(4)$ denote the (``Kohnen'') space of weight $1/2$ Maa{\ss} cusp forms for $\Gamma_0(4)$ with eigenvalue $1/4 + t^2$ with respect to the weight $1/2$ Laplacian and  Fourier expansion
\begin{equation}\label{four-half}
v(z) = \sum_{ \substack{n\not= 0 \\ n \equiv 0, 1 \, (\text{mod } 4)}} b(n) W_{\frac{1}{4} \text{sgn}(n), it}(4 \pi |n|y) e(nx).
\end{equation}
The congruence condition on the indices can be encoded in an eigenvalue equation: the functions $v \in V_t^+(4)$ are invariant under the operator $L$ defined in \cite[(0.7), (0.8)]{KS}, cf.\ also \cite[(A.1)]{Bi1}.  The space $V_t^{+}(4)$  is a finite-dimensional Hilbert space with respect to the inner product 
\begin{equation}\label{inner}
\langle v_1, v_2 \rangle =   \int_{\Gamma_0(4)\backslash \Bbb{H}} v_1(z) \overline{v_2(z)}   \frac{dx\, dy}{y^2}.
\end{equation}
The Hecke operators $T(p^2)$, $p$ prime,  act on $V^+_t(4)$ as a commutative family of normal operators that commute with the weight $1/2$ Laplacian (again we use Kohnen's modification for $T(4)$ in order to treat all primes uniformly). Explicitly, if $T(p^2) v = \lambda(p) v$, then (see \cite[(1.3)]{KS}) the Fourier coefficients of $v$ satisfy
\begin{equation}\label{hecke-half}
\lambda(p) b(n) = p b(np^2) + p^{-1/2} \chi_n(p) b(n) + p^{-1} b(n/p^2)
\end{equation}
for all primes $p$ and all $n \in \Bbb{Z} \setminus \{0\}$ with $n \equiv 0, 1$ (mod 4) with the convention $b(x) = 0$ for $x \not\in \Bbb{Z}$. 
If $v \in V_t^+(4)$  is an eigenfunction of all Hecke operators $T(p^2)$ with eigenvalues $\lambda(p)$, 
the relation \eqref{hecke-half} can be captured in the identity
$$   \sum_{f=1}^{\infty} \frac{b(\Delta f^2 )}{f^{s-1}} = b(\Delta)  \prod_p \frac{1 -  \chi_{\Delta}(p)p^{-s-1/2}}{1 -  \lambda(p) p^{-s} + p^{-2s}}$$
for a fundamental discriminant $\Delta$. Extending $\lambda(p)$ to all $n$ by the usual Hecke relations, we see that the denominator is just $\sum_{\nu} \lambda(p^{\nu})p^{-\nu s}$, so that  
\begin{equation}\label{non-fund}
   f b(\Delta f^2) =  b(\Delta) \sum_{d \mid f} \mu(d) \chi_{\Delta}(d) \lambda(f/d) d^{-1/2}
\end{equation}
for a fundamental discriminant $\Delta$ and $f \in \Bbb{N}$. \\



\section{Period formulae}

Let $v \in V_t^+(4)$. 
  Katok and Sarnak proved in  \cite[Proposition 4.1]{KS} that there is a linear map (a theta lift) $\mathscr{S}$ sending $v$ to a non-zero element in $U_{2t}^{\text{ev}}$ if $b(1)\not= 0$ and to $0$ otherwise.  A calculation  \cite[p.\ 221-223]{KS} shows that if $v$ is an eigenform of $T(p^2)$, then $\mathscr{S} v$ is an eigenform of $T(p)$   with the same eigenvalue, and this computation works verbatim for $p=2$, too.    Conversely, given an  eigenform  $u\in U_{2t}^{\text{ev}}$ with Hecke eigenvalues $\lambda(p)$, by \cite[Theorem 1.2]{BM}\footnote{Hecke operators are normalized differently in \cite{BM}, but this plays no role.} there is a unique (up to scaling)    $v \in V_t^+(4)$ having eigenvalues $\lambda(p)$ for $T(p^2)$, $p > 2$, (and then automatically also for $T(4)$, since $T(4)$ commutes with the other operators) which may or may not satisfy $b(1) \not= 0$.    In particular, for a given eigenform    $u\in U_{2t}^{\text{ev}}$  there is at most one eigenform $v \in V_t^+$ (up to scaling) with $\mathscr{S}v = u \in U_{2t}^{\text{ev}}$. If it exists, we normalize it to be $L^2$-normalized and denote its Fourier coefficients $b(n)$ as in\footnote{This determines $b(n)$ only up to a constant of absolute value 1, but we will only encounter products of the type $b(n_1)\overline{b(n_2)}$, so that this constant is irrelevant.}  \eqref{four-half}.  If no such $v$ exists, we define $b(n)$ to be 0.


If ${\tt u}$ is a Hecke--Maa{\ss} cusp form or an Eisenstein series, the absolute square of the periods $P(D; {\tt u})$ can be expressed in terms of central $L$-functions. We introduce the relevant notation. For a discriminant $D = \Delta f^2$ with a fundamental discriminant $\Delta$ let
\begin{equation}\label{basicL}
L(D, s) :=  L(\chi_{\Delta}, s)  
\sum_{d \mid f} \mu(d) \chi_{\Delta}(d) \sigma_{1-2s}(f/d) d^{-s}=: \sum_{n=1}^{\infty} \frac{\varepsilon_D(n)}{n^s}.
\end{equation}
With $\rho_s$ as in \eqref{defrho} we can re-write this as
\begin{equation}\label{basicL1}
L(D, s) =  L(\chi_{\Delta}, s) f^{1/2 - s} 
\sum_{d \mid f} \mu(d) \chi_{\Delta}(d) \rho_{s}(f/d) d^{-1/2}.
\end{equation}
Since $\rho_s = \rho_{1-s}$, we see that $L(D, s)$ satisfies the same type of functional equation as $L(\Delta, s)$ namely
\begin{equation}\label{Lfuncteq}
   \Lambda(D, s) := L(D, s) |D|^{s/2} \Gamma\Big(\frac{s+\mathfrak{a}}{2}\Big) \pi^{-s/2} = \Lambda(D, 1-s)
\end{equation}
with $\mathfrak{a} = 1$ if $D < 0$ and $\mathfrak{a} = 0$ if $D > 0$. 
For a Hecke--Maa{\ss} cusp form or an Eisenstein series  ${\tt u}  $ define
\begin{equation}\label{alt}
L({\tt u}, D, s) = \sum_{n=1}^{\infty} \frac{\varepsilon_D(n) \lambda_{\tt u}(n)}{n^s}.
\end{equation} 
The key point is that
\begin{equation}\label{key}
L({\tt u}, \Delta f^2, 1/2)  = L({\tt u}, \Delta, 1/2) \Bigl(\sum_{d \mid f} \mu(d) \chi_{\Delta}(d) \lambda_{\tt u}(f/d) d^{-1/2}\Bigr)^2
\end{equation}
as one can check by a formal computation with Euler products using the Hecke relation for the eigenvalues $\lambda_{\tt u}$  (which are identical for Maa{\ss} forms and Eisenstein series), see \cite[p.\ 188-189]{KZ}. For later purposes we record the simple bound
\begin{equation}\label{key-simple}
\Bigl(\sum_{d \mid f} \mu(d) \chi_{\Delta}(d) \lambda_{\tt u}(f/d) d^{-1/2}\Bigr)^2 \ll f^{1/3}
\end{equation}
uniformly in $\Delta$ and ${\tt u}$, which follows from the Kim-Sarnak bound with $ 2 \cdot 7/64 < 1/3$. 

From \eqref{basicL1}, the fact that $|L(\chi_{\Delta}, 1/2 + it)|^2 = L(E(., 1/2 + it), \Delta, 1/2)$ and \eqref{key} we   have 
\begin{equation}\label{eisen-L2}
\begin{split}
|L(\Delta f^2, 1/2+it)|^2 &= |L(\chi_{\Delta}, 1/2+it)|^2 \Bigl|\sum_{d \mid f} \mu(d) \chi_{\Delta}(d) \lambda_{E(., 1/2 + it)}(f/d) d^{-1/2}\Bigr|^2\\
& = L(E(., 1/2 + it), \Delta f^2, 1/2). 
\end{split}
   \end{equation}

The next key lemma expresses the periods $P(D; u)$ defined in \eqref{defP} for cusp forms $u$ as half-integral weight Fourier coefficients, and then their squares as $L$-functions. The first formula \eqref{mixed} is essentially a formula of Katok-Sarnak \cite[(0.16) \& (0.19)]{KS}, the passage from squares of metaplectic Fourier coefficients to $L$-functions in \eqref{BarMao} is a Kohnen-Zagier type formula of Baruch-Mao \cite[Theorem 1.4]{BM}. The combination \eqref{katok-Sarnak} of these two is a special case of a formula of Zhang \cite[Theorem 1.3.2]{Zh1} or \cite[Theorem 7.1]{Zh2}, derived independently by a different method.

\begin{lemma}\label{lem3} If ${ u} \in U_{2t}^{\text{{\rm ev}}}$ is an even Hecke--Maa{\ss} cusp form   and $D_1, D_2 <0$ are  two discriminants, then
\begin{equation}\label{mixed}
\frac{P(D_1;  {  u})\overline{P(D_2;  {  u})}  \|{  u}\|^{-2} }{|D_1D_2|^{1/4}}  = \frac{3}{\pi} L({  u}, 1/2)  \Gamma(1/4 + it)\Gamma(1/4 - it) |D_1D_2|^{1/2}b(D_1) \overline{b(D_2)}.
\end{equation}
  For a discriminant $D$ of either sign we have
\begin{equation}\label{BarMao}
|b(D)|^2 = \frac{1}{24\pi} \frac{L({  u},  D, 1/2)}{L(\text{{\rm sym}}^2 {  u}, 1)} \frac{\cosh(2 \pi t)\Gamma( \frac{1}{2} - \frac{1}{4} \text{{\rm sgn}}(D)+ it)\Gamma( \frac{1}{2} - \frac{1}{4} \text{{\rm sgn}}(D)- it)}{  |D|}.
\end{equation}
 For $D < 0$ we have 
\begin{equation}\label{katok-Sarnak}
\frac{|P(D;  {  u})|^2 \|{  u}\|^{-2} }{|D|^{1/2}}  = \frac{L({  u}, 1/2) L({  u},  D, 1/2)}{4L(\text{{\rm sym}}^2 {  u}, 1)}.
\end{equation}
\end{lemma}

\emph{Remark:} The exact shape of these formulas is an unexpectedly subtle matter, and the attentive reader might well be confused by the various and slightly contradictory versions in the literature. There are at least four sources of possible conflict:
\begin{itemize}
\item the Whittaker functions can be normalized in different ways;
\item the inner products can be normalized in different ways;
\item the translation from adelic language to classical can cause problems;
\item there can be ambiguities related to the groups ${\rm GL}(2)$ vs.\ ${\rm SL}(2)$ vs.\ ${\rm PSL}(2)$.
\end{itemize}
The Katok-Sarnak formula exists in the literature with proofs given in at least in three different versions: \cite[(0.16)]{KS}, \cite[Theorem A1]{Bi2} and \cite[Theorem 4]{DIT}. The original version of Katok-Sarnak was carefully revised by Bir\'o, but the latter seems to be still off by a factor 2 compared to the version in Duke-Imamo\u{g}lu-Toth, which was checked numerically. 
The Baruch-Mao formula \cite[Theorem 1.4]{BM} is quoted in \cite[(5.17)]{DIT}  with an additional factor 2. Zhang's result  \cite[Theorem 7.1]{Zh2} (and also the remark after \cite[Theorem 1.3.2]{Zh1}, the theorem itself being correct) is missing the stabilizer $\epsilon(z)$ in the period $P(D; u)$. This formula is slightly incorrectly reproduced in   \cite{LMY} and several follow-up papers,  based on a different normalization of the Whittaker function. Finally, neither combination of one of the three  Katok-Sarnak formulae with the Baruch-Mao formula in \cite[Theorem 1.4]{BM}  coincides with Zhang's formula. 

We therefore feel  that these beautiful and important results should be stated with correct constants and normalizations. For the proof of Theorem \ref{thm1}  and its connection to the mass equidistribution conjecture this is absolutely crucial. As \cite[Theorem 4]{DIT} was checked numerically by the authors, we follow their version of the Katok-Sarnak formula. This gives   \eqref{mixed}. We verified and confirmed the constant in Zhang's formula independently by proving an averaged version in Appendix \ref{appb}. This gives \eqref{katok-Sarnak}. By backwards engineering, we established the numerical constant in the Baruch-Mao formula, which gives \eqref{BarMao} and coincides with \cite[(5.17)]{DIT}. 

Note that \eqref{katok-Sarnak} is essentially universal: the right hand side of \eqref{katok-Sarnak} is independent of any normalization, the left hand side depends only on the normalization of the inner product \eqref{inner-maass} which is standard.


\begin{proof}
We start with the formula \cite[(0.16)]{KS} for a general discriminant $D < 0$, but use the numerical constants as in 
 \cite[Theorem 4]{DIT} (proved only for fundamental discriminants there). This formula  expresses  $P(D; {  u})$ for an arbitrary discriminant $D < 0$ as a sum over Fourier coefficients of all $v$ with $\mathscr{S}v = {  u}$. By the above remarks, there is at most one such $v$. If there is none, then both sides of \eqref{mixed}  and \eqref{katok-Sarnak} vanish by \cite[(0.16), (0.19)]{KS} and our convention that $b(n) = 0$ in this case, and there is nothing to prove. 
Also note that the left hand side of \eqref{mixed} and both sides of \eqref{katok-Sarnak}  are independent of the normalization of ${  u}$, so without loss of generality we may assume that ${  u}$ is Hecke-normalized as in \cite{KS}. 
We obtain
$$\frac{P(D_1;  {  u})\overline{P(D_2;  {  u})}  \|{  u}\|^{-2} }{|D_1D_2|^{1/4}}    
=   6 |D_1D_2|^{1/2} b(D_1)\overline{b (D_2)}  |b(1)|^2 \| {  u}\|^2.$$
Next we insert \cite[(0.19)]{KS} (again keeping in mind the different normalization of \eqref{inner} and observing that this is coincides with the numerically checked version of \cite[Theorem 4]{DIT}) getting \eqref{mixed}. 

If $D$ is a fundamental discriminant, then \eqref{BarMao}   follows from   \cite[(5.17)]{DIT}  
together with \eqref{norm} with $a(1) = 1$ and $2t$ in place of $t$. By \eqref{non-fund} and \eqref{key} this remains true for arbitrary discriminants. 
The formula \eqref{katok-Sarnak} is a direct consequence of \eqref{mixed}, \eqref{BarMao} and well-known properties of the gamma function, and was proved independently by Zhang \cite[Theorem 1.3.2]{Zh1}. 
\end{proof}

\emph{Remark:} The argument at the beginning of this proof shows that the $u$-sum in Proposition \ref{Lfunc}(a), up to terms of size 0,  can be replaced with the $v$-sum in \eqref{interpret}, using \eqref{BarMao}. This completes the proof of \eqref{interpret}. 
 \\

A similar result holds for Eisenstein series.    If $aX^2 + bXY + cY^2$ is an integral quadratic form of discriminant $D = b^2 -4ac < 0$ with Heegner point $z = (\sqrt{|D|}i - b)/(2a)$, then
$$E(z, s) = \frac{1}{  \zeta(2s)} \sum_{(u, v) \in (\Bbb{Z}^2 \setminus (0, 0))/\{\pm 1\}} \frac{(\sqrt{|D|}/2)^s}{(au^2 - buv + cv^2)^s}.$$
Hence $$P(D; E(., s)) = \frac{1}{ \zeta(2s)} \Big(\frac{\sqrt{|D|}}{2}\Big)^s \zeta(D, s)$$
where $\zeta(D, s)$ is defined\footnote{Note that Zagier defines $\Gamma = {\rm PSL}_2(\Bbb{Z})$, so his definition of equivalence coincides with ours. The quotient by $\{\pm 1\}$ in the $u, v$-sum is not spelled out explicitly in \cite[(6)]{Z}, but implicitly used in the proof of \cite[Proposition 3]{Z} on p.\ 131. }  in \cite[(6)]{Z}. By \cite[Proposition 3 iii]{Z} (or \cite[Theorem 3]{DIT}) we obtain the following lemma in analogy to Lemma \ref{lem3}.

\begin{lemma}\label{lem-eisen} If $D < 0$ is a discriminant, then
\begin{equation}\label{eisen1}
P(D; E(., s)) = \frac{1}{ \zeta(2s)} \Big(\frac{\sqrt{|D|}}{2}\Big)^s \zeta(s) L(D, s),
\end{equation}
and hence
\begin{equation}\label{eisen2}
\frac{|P(D; E(., 1/2 + it))|^2}{|D|^{1/2}} = \frac{L(E(., 1/2 + it), 1/2) L(E(., 1/2 + it), D, 1/2)}{ 2|\zeta(1 + 2 it)|^2} .
\end{equation}  
\end{lemma}

\emph{Remarks:} 1) The second formula follows from the first by \eqref{eisen-L} and  \eqref{eisen-L2}.

2) Here the verification of the numerical constants is much easier than in the cuspidal case. Taking residues at $s=1$ in \eqref{eisen1} for a fundamental discriminant $\Delta < 0$ returns the class number formula for $\Bbb{Q}(\sqrt{\Delta})$, which confirms the numerical constants.


3) For future reference we recall the standard   bounds
\begin{equation}\label{lower}
\zeta(1 + it) \gg |t|^{-\varepsilon}, \quad  |t_{  u}|^{\varepsilon} \gg L(\text{sym}^2 {  u}, 1) \gg |t_{  u}|^{-\varepsilon}
\end{equation}
for $\varepsilon > 0$. This is in particular  relevant to obtain upper bounds in \eqref{eisen2} and \eqref{katok-Sarnak}. \\


We close this section by stating a standard approximate functional equation \cite[Theorem 5.3]{IK} for the $L$-functions occurring in the previous period formulae. For $u \in U_t^{\text{ev}}$ and a fundamental discriminant $\Delta$ (possibly $\Delta = 1$) we have 
\begin{equation}\label{approx-basic}
L(u, \Delta, 1/2) = L(u \times \chi_{\Delta}, 1/2) = 2\sum_n \frac{\lambda(n)\chi_{\Delta}(n)}{n^{1/2}} W_{t}\Big(\frac{n}{|\Delta|}\Big)
\end{equation}
where
\begin{equation}\label{v-t}
W_{t}(n) = \frac{1}{2\pi i} \int_{(2)} \frac{\Gamma(\frac{1}{2}(\frac{1}{2}+ \mathfrak{a} + s + it))\Gamma(\frac{1}{2}(\frac{1}{2} + \mathfrak{a}+ s - it))}{\Gamma(\frac{1}{2}(\frac{1}{2} + \mathfrak{a}  + it))\Gamma(\frac{1}{2}(\frac{1}{2} + \mathfrak{a}  - it)) \pi^{s} }e^{s^2} n^{-s}
 \frac{ds}{s}
 \end{equation}
with $\mathfrak{a} = 1$ if $\Delta < 0$ and $\mathfrak{a} = 0$ if $\Delta > 0$. Note that $W_t$ depends on $\Delta$ only in terms of its sign. If we want to emphasize this we write $W_t^+$ and $W_t^-$ with $\pm = \text{sgn}(\Delta)$. 

A similar expression holds for $E(., 1/2 + it)$ in place of $u$ except that in the case $\Delta = 1$ we have $L(E(., 1/2 + it), 1, s) = \zeta(s+it)\zeta(s- it)$ and there is an additional polar term\footnote{Note that there is a sign error in \cite[Theorem 5.3]{IK}: the residue $R$ should be subtracted.}. We have
\begin{equation}\label{approx-basic1}
\begin{split}
L(E(., 1/2 + it), \Delta, 1/2)&  = 2\sum_n \frac{\rho_{1/2 + it}(n) \chi_{\Delta}(n)}{n^{1/2}} W_{t}\Big(\frac{n}{|\Delta|}\Big)\\
& -\delta_{\Delta = 1} \sum_{\pm} \frac{\zeta(1 \pm 2it)\Gamma(\frac{1}{2} \pm it)\pi^{\mp it} e^{(1/2 \pm  it)^2}}{(\frac{1}{2} \pm i t)\Gamma(\frac{1}{4} + \frac{it}{2})\Gamma(\frac{1}{4}  - \frac{it}{2})}
\end{split}
\end{equation}
with $\rho_{1/2 + it}$ as in \eqref{defrho}. 

\section{Half-integral weight summation formulae}\label{summation}

In this section we compile the Voronoi summation and the Kuznetsov formula for half-integral weight forms. 

\subsection{Voronoi summation}
As before let 
\begin{equation}\label{51}
v(z) = \sum_{  n\not= 0 } b(n) W_{\frac{k}{2} \text{{\rm sgn}}(n), it}(4 \pi |n|y) e(nx)
\end{equation}
 be a Maa{\ss} form of weight $k \in \{1/2, 3/2\}$ and spectral parameter $t$ for $\Gamma_0(4)$ with respect to the usual theta multiplier. 
We start with the   Voronoi summation formula \cite[Theorem 3]{By}.
\begin{lemma}\label{Vor}  Let $c\in \Bbb{N}$, $4 \mid c$, $(a, c) = 1$. Let $\phi$ be a smooth function with compact support in $(-\infty, 0) \cup (0, \infty)$ and for $y > 0$ define
 \begin{equation}\label{defPhi}
 \Phi(\pm y) = \int_0^{\infty} \mathcal{J}^{\pm, +}(ty) \phi(t) + \mathcal{J}^{\pm, -}(ty) \phi(-t) dt,
 \end{equation}
 where
 \begin{displaymath}
 \begin{split}
&   \mathcal{J}^{\pm, \pm}(x) = \frac{\cos(\pi k/2 \mp i \pi r)}{\sin(2\pi i r)} J_{-2 ir}(2\sqrt{x}) -  \frac{\cos(\pi k/2\pm   i \pi r)}{\sin(2\pi i r)} J_{2 ir}(2\sqrt{x}) = - \frac{F(2 \sqrt{x}, 2r, \mp k)}{\sin(2\pi i r)},\\
  &  \mathcal{J}^{\pm, \mp}(x)  = \frac{2 K_{2ir}(2\sqrt{x})}{\Gamma(1/2 \pm k/2 + ir)\Gamma(1/2 \pm k/2 - ir)} \end{split}
 \end{displaymath}
with $F$ as in \eqref{defF}.   Then
 $$\sum_{n \not = 0} b(n)\sqrt{|n|} e\Big(\frac{an}{c}\Big) \phi(n) = \Big(\frac{-c}{a}\Big) \epsilon^{2k}_a e\Big(\frac{k}{4}\Big) \sum_{n \not= 0} b(n) \sqrt{|n|}  \frac{2\pi}{c} e\Big(-\frac{\bar{a}n}{c}\Big) \Phi\Big(\frac{(2\pi)^2 n}{c^2}\Big)$$
 with $\epsilon_a$ as in \eqref{epsd}. 
\end{lemma}

The proof of the Voronoi formula (Lemma \ref{Vor}) follows from a certain vector-valued functional equation satisfied by the $L$-functions with coefficients $b(\pm n) e(\pm an/c)$. The same functional equation holds if $v$ is not cuspidal (this is clear from general principles and worked out explicitly in \cite{DG} along the same lines), but in this case the $L$-functions are not entire; they have various poles. We use this observation for two non-cuspidal modular forms. The first is a half-integral weight Eisenstein series 
\begin{equation}
\begin{split}
E^{\ast}\Big(z, \frac{1}{2} + it\Big) =&\frac{ 2^{1 + 2it} \pi^{-it} \Gamma(1/2 + 2it) \zeta(1+4it)  }{\Gamma(3/4 + it)\Gamma(1/4 + it)} y^{1/2 + it} + \frac{2^{1 - 2it} \pi^{3it} \Gamma(1/2 - 2it) \zeta(1-4it)}{\Gamma(3/4 + it)\Gamma(1/4 + it)} y^{1/2 - it} \\
&+ \sum_{D } \frac{L(D,1/2 + 2it) |D|^{it}  }{|D|^{1/2} }\frac{W_{\frac{1}{4} \text{sgn}(D), it}(4 \pi |D|y)}{\Gamma(\frac{1}{2} + \frac{1}{4} \text{sgn}(D) + it)}e(Dx)\\
\end{split}
\end{equation}
which transforms under $\Gamma_0(4)$ as a weight $1/2$ automorphic form with theta multiplier; see \cite[p.\ 964]{DIT}.  As usual, $D$ runs over all discriminants. The other is Zagier's weight 3/2 Eisenstein series\footnote{We have multiplied Zagier's definition by a factor $(4\pi y)^{3/4}$ in order to make $|\mathcal{H}(z)|$ invariant under $\Gamma_0(4)$.} \cite[Theorem 2]{HZ}
\begin{equation}\label{zag}
\begin{split}
\mathcal{H}(z) = &\sum_{D < 0} \frac{H(D)}{|D|^{3/4}} e(|D|x) W_{3/4, 1/4}(4\pi |D|y)  - \frac{(4\pi y)^{3/4}}{12} \\
& + \frac{1}{4\sqrt{\pi}}
\sum_{n = \square} \frac{e(-nx)}{n^{1/4}} W_{-3/4, 1/4}(4\pi n y)+ \frac{y^{1/4}}{\sqrt{8} \pi^{1/4}} \end{split}
\end{equation}
which transforms under $\Gamma_0(4)$ as a weight $3/2$ automorphic form with theta multiplier. As before $H(D)$ is the Hurwitz class number. For $\epsilon \in \{\pm 1\}$, $(a, c) = 1$, $4\mid c$, the Dirichlet series 
$$\sum_{\epsilon D > 0} \frac{L(D, 1/2 + 2it)e(aD/c) |D|^{it}}{|D|^{1/2 + w}}$$
has poles at $w = 1/2 \pm it$ with certain residues $R_{\epsilon}(\pm t, a/c)$, say, and consequently we obtain the following analogue of Lemma \ref{Vor}.

\begin{lemma}\label{Vor-eis}  Let $c\in \Bbb{N}$, $4 \mid c$, $(a, c) = 1$, $t \in \Bbb{R} \setminus \{0\}$.  Let $\phi$ be a smooth function with compact support in $(-\infty, 0) \cup (0, \infty)$ and for $y > 0$ define $\Phi$ as in \eqref{defPhi}.  Then
 \begin{displaymath}
 \begin{split}
 \sum_{D \not = 0}\frac{L(D, 1/2 + 2it)|D|^{it} }{\Gamma(\frac{1}{2} + it + \frac{1}{4} \text{{\rm sgn}}(D))} e\Big(\frac{aD}{c}\Big) \phi(D) = &\Big(\frac{-c}{a}\Big) \epsilon_a e\Big(\frac{1}{8}\Big)\Big[\sum_{\epsilon \in \{\pm 1\}} \sum_{\pm} \frac{R_{\epsilon}(\pm t, a/c)}{\Gamma(\frac{1}{2} + it + \frac{1}{4}\epsilon)}\int_0^{\infty} \phi(\epsilon x) x^{ \pm it} dx \\
& + \sum_{D \not= 0}\frac{L(D, 1/2 + 2it)|D|^{it} }{\Gamma(\frac{1}{2} + it + \frac{1}{4} \text{{\rm sgn}}(D))} \frac{2\pi}{c} e\Big(-\frac{\bar{a}D}{c}\Big) \Phi\Big(\frac{(2\pi)^2 D}{c^2}\Big) \Big]
 \end{split}
 \end{displaymath}
\end{lemma}

For $t=0$ one combines the $\pm$-terms and takes the limit as $t \rightarrow 0$. In our application, the values $R_{\epsilon}(t, a/c)$ are irrelevant (as long as they are polynomial in $c$ and $t$) since we apply the formula with a function $\phi$ that oscillates much more strongly that $x^{\pm it}$ so that the integral is negligible. 

We obtain a similar summation formula for Hurwitz class numbers. Although we do not need it for the present result, we compute in Appendix \ref{appa} the residues explicitly and get the following handsome formula. 
\begin{lemma}\label{class-num1} Let $c\in \Bbb{N}$, $4 \mid c$, $(a, c) = 1$. Let $\phi$ be a smooth function with compact support in $(0, \infty) $. Then 
\begin{displaymath}
\begin{split}
\sum_{D < 0}\frac{H(D)}{|D|^{1/4}} e\Big(\frac{a|D|}{c}\Big) \phi(|D|) =  \Big(\frac{-c}{a}\Big) \bar{\epsilon}_a e\Big(\frac{3}{8}\Big)&\Bigg[ \sum_{D <0} \frac{H(D)}{|D|^{1/4}}  \frac{2\pi}{c} e\Big(-\frac{\bar{a}|D|}{c}\Big) \int_0^{\infty} \mathcal{J}^+ \Big(t \frac{(2\pi)^2 |D|}{c^2}\Big) \phi(t) dt\\
& 
+ \frac{1}{4\sqrt{\pi}} \sum_{ n = \square}  n^{1/4}  \frac{2\pi}{c} e\Big(\frac{\bar{a}n}{c}\Big) \int_0^{\infty} \mathcal{J}^{-}\Big(t \frac{(2\pi)^2 n}{c^2}\Big) \phi(t) dt \\
&
+ \int_0^{\infty} \phi(x) \Big( \frac{1}{\sqrt{8}c^{1/2}} x^{-1/4}- \frac{\sqrt{2}\pi}{3 c^{3/2}} x^{1/4}   \Big) dx \Bigg]
\end{split}
\end{displaymath}
where 
\begin{equation}\label{Jclean}
\mathcal{J}^+(x) =   \frac{\sin(2\sqrt{x}) }{\sqrt{\pi} x^{1/4}}, \quad \mathcal{J}^-(x) = \frac{2e^{-2\sqrt{x}}}{x^{1/4}}.
\end{equation}
\end{lemma}

\emph{Remark:} Observing that for negative $D \equiv 0, 1$ (mod 4) we have 
$$e(|D|/4) + e(3|D|/4) = 2 \delta_{D \equiv 0 \, (\text{mod } 4)}, \quad -e(|D|/4) + e(3|D|/4) = 2i \delta_{D \equiv 1 \, (\text{mod } 4)},$$
it is a straightforward exercise to conclude
$$\sum_{\substack{-X < D < 0\\ D\equiv \delta \, (\text{mod } 4)}} H(D) = \frac{\pi}{36} X^{3/2} - \frac{1}{8} X + O(X^{3/4})$$
for $\delta = 0, 1$. Further congruence conditions on $D$ can be imposed, and the error term can be improved by a more careful treatment of the dual term in the Voronoi formula.  See \cite{Vi} for the corresponding result for the ordinary class number $h(d)$. The following table provides some numerical results (here we combined the cases $\delta = 0$ and $\delta = 1$). \\

{\footnotesize
\begin{tabular}{l|l|l|l|l|l|l}
$X$ & 1000 & 2000 & 4000 & 6000 & 8000 & 10000\\ \hline   
$\displaystyle\sum^{\substack{\quad\\ \quad}}_{D \leq X} H(D)$ & 5280.5 &15131.3& 43189.5& 79685.7 & 122967 & 172106 \\ \hline
$\displaystyle \frac{\pi}{18} \overset{\substack{\quad\\ \quad}}{X^{3/2}} - \frac{1}{4}X\quad$ & 5269.22 & 15110.7 & 43153.7 & 79615.6 & 122885.6 & $\underset{\substack{\quad\\\quad\\\quad}}{172032.9}$  \\ \hline
\end{tabular}}

\vspace{0.4cm}

\subsection{The Kuznetsov formula}
The Kuznetsov formula was generalized  by Proskurin \cite{Pr} to arbitrary weights and by Andersen-Duke \cite{AD} to Kohnen's subspace. Interestingly, only the direction from Kloosterman sums to spectral sums appears to be in the literature, but no complete version in the other direction. Bir\'o \cite[p.\ 151]{Bi1} has a version only valid for test functions on the spectral side whose spectral mean value is 0, Ahlgren-Andersen \cite[Section 3]{AA} use an approximate version. 

We take this opportunity to state and prove the relevant Kuznetsov formula both for the full space of half-integral weight forms and for Kohnen's subspace. For $\kappa \in \{1/2, 3/2\}$  the relevant Kloosterman sums are
\begin{equation}\label{klooster2}
K_{\kappa}(n, m, c) =  \sum_{\substack{d\, (\text{mod }c)\\ (d, c) = 1}} \epsilon^{2\kappa}_d \left(\frac{c}{d}\right) e\left(\frac{nd + m\bar{d}}{c}\right)  
\end{equation}
for $4 \mid c$. The Eisenstein series belong to the two essential cusps $\mathfrak{a} = \infty, 0$. We normalize and denote their Fourier coefficients by $\phi_{\mathfrak{a}m}(1/2 + it) = \phi^{(\kappa)}_{\mathfrak{a}m}(1/2 + it)$ as in \cite[(12) - (14)]{Pr}. We denote by $\left.\sum\right.^{(\kappa)}$ a sum over an orthonormal basis of the space of cusp forms of weight $\kappa$ and label the members by  $v_j$, $j = 1, 2, \ldots$ with Fourier coefficients  $b_j(n)$ as in \eqref{51} and  spectral parameters by $t_j$. 
\begin{prop}\label{kuz-full}
Let $\kappa \in \{1/2, 3/2\}$, $m, n > 0$. Let $h$ be an even function, holomorphic in $|\Im t| < 2/3$ with $h(t) \ll (1 + |t|)^{-4}$. Then
\begin{displaymath}
 \begin{split}
&   \left.\sum_j \right.^{(\kappa)}\frac{\sqrt{mn}\overline{b_j(m)} b_j(n)}{\cosh(\pi t_j)} h(t_j)   + \sum_{\mathfrak{a}} \int_{-\infty}^{\infty} \Big(\frac{n}{m}\Big)^{it} \frac{ \overline{\phi^{(\kappa)}_{\mathfrak{a}m}(1/2 + it)} \phi^{(\kappa)}_{\mathfrak{a}n}(1/2 + it)}{4\cosh(\pi t)  \Gamma(\frac{1+\kappa}{2} + it)\Gamma(\frac{1+\kappa}{2}- it)} h(t) dt \\
  &=  \delta_{n=m}\int_{-\infty}^{\infty} h(t) t \sinh(\pi t)\Gamma \Big( \frac{1-\kappa}{2} + it\Big)\Gamma \Big( \frac{1-\kappa}{2} - it\Big)  \frac{dt}{4\pi^3} \\
  &+ e\Big(\frac{1-\kappa}{4}\Big)\sum_c \frac{K_{\kappa}(m, n, c)}{c} \int_0^{\infty} \frac{F( 4\pi\sqrt{nm}/c, 2t, -\kappa)}{\cosh(\pi t)} \Gamma \Big( \frac{1-\kappa}{2} + it\Big)\Gamma \Big( \frac{1-\kappa}{2} - it\Big) h(t) t\, \frac{dt}{2\pi^2}
   \end{split}
\end{displaymath}
if in addition  $h(\pm i/4) = 0$. Moreover, regardless of the value of $h(\pm i/4)$ we have 
\begin{displaymath}
 \begin{split}
&   \left.\sum_j\right.^{(\kappa)} \frac{\sqrt{mn}\overline{b_j(-m)} b_j(-n)}{\cosh(\pi t_j)} h(t_j)   + \sum_{\mathfrak{a}} \int_{-\infty}^{\infty} \Big(\frac{n}{m}\Big)^{it} \frac{ \overline{\phi^{(\kappa)}_{\mathfrak{a}, -m}(1/2 + it)} \phi^{(\kappa)}_{\mathfrak{a}, -n}(1/2 + it)}{4\cosh(\pi t)\Gamma(\frac{1-\kappa}{2} + it)\Gamma(\frac{1-\kappa}{2}- it)} h(t) dt \\
  &=  \delta_{n=m}\int_{-\infty}^{\infty} h(t) t \sinh(\pi t)\Gamma\Big(\frac{1+\kappa}{2} + it\Big)\Gamma\Big(\frac{1+\kappa}{2}- it\Big) \frac{dt}{4\pi^3} \\&+ e\Big(\frac{1+\kappa}{4}\Big)\sum_c \frac{K_{-\kappa}(m, n, c)}{c} \int_0^{\infty} \frac{F( 4\pi\sqrt{nm}/c, 2t, \kappa)}{\cosh(\pi t)}\Gamma\Big(\frac{1+\kappa}{2} + it\Big)\Gamma\Big(\frac{1+\kappa}{2}- it\Big) h(t) t\, \frac{dt}{2\pi^2}.
   \end{split}
\end{displaymath}
\end{prop}

\emph{Remark:} Note that the space of weight 1/2 Maa{\ss} forms $v$ with spectral parameter $i/4$ are in the kernel of the Maa{\ss} lowering operator. Hence $y^{-1/4}v$ is holomorphic, so that $v$ has no non-vanishing negative Fourier coefficients. This is consistent with the fact that the Eisenstein contribution vanishes in this case because of the gamma factors in the denominator. 

\begin{proof} By \cite[Lemma 3]{Pr} with $\sigma = 1$, $t = 2\tau \in \Bbb{R}$ and the first formula in \cite[Lemma 6]{Pr} we have the ``pre-Kuznetsov'' formula
\begin{displaymath}
\begin{split}
&-\sum_c \frac{K_{\kappa}(m, n, c)}{c^2}  x^{-\kappa}\frac{\pi}{\sinh(2\pi \tau)} \int_0^{x} F(y, 2\tau, 1-\kappa)y^{\kappa - 1} dy+ \frac{\delta_{n=m}e((1+\kappa)/4)}{4\pi(n+m)} \\
&= \frac{\pi^2e((1+\kappa)/4)}{2\Gamma(1 - \frac{\kappa}{2} + i\tau)\Gamma(1 - \frac{\kappa}{2} - i\tau)}\Bigg(\sum_j \frac{\overline{b_j}(m)b_j(n) }{\cosh(\pi(t_j - \tau))\cosh(\pi(t_j + \tau))}\\
& + \frac{1}{4\sqrt{nm}} \sum_{\mathfrak{a}} \int_{-\infty}^{\infty} \Big(\frac{m}{n}\Big)^{-it} \frac{\overline{\phi_{\mathfrak{a}m}(1/2 + it)}\phi_{\mathfrak{a}n}(1/2 + it) }{\Gamma(\frac{1+\kappa}{2} + it) \Gamma(\frac{1+\kappa}{2} - it) \cosh(\pi(t- \tau))\cosh(\pi(t + \tau))} dt\Bigg).
\end{split}
\end{displaymath}
where $$x = 4\pi \sqrt{mn}/c.$$ 
 Note that taking $\sigma = 1$ is admissible in the present situation because we have the same  Weil-type bounds for the Kloosterman sums $K_{\kappa}(n, m, c)$ as for $K^+_{\kappa}(n, m, c)$ in \eqref{weil}. Also note that there is a typo in \cite[Lemma 6]{Pr} in the upper limit of the integral. 
 
 For $h$ as in the lemma and $t \in \Bbb{R}$ we have the following inversion formula
$$\int_{-\infty}^{\infty} \big(h(\tau + i/2)+ h(\tau - i/2)\big) \frac{\cosh(\pi \tau)}{\cosh(\pi(\tau - t)) \cosh(\pi(\tau + t))} d\tau = \frac{2 h(t)}{\cosh(\pi t)}.$$
This is the lemma on p.327 of \cite{Ku}\footnote{We have corrected a sign error. This sign error is cancelled  by another sign error in \cite[(6.6)]{Ku}. The inversion formula was re-produced in \cite[Lemma 16.4]{IK} with the same sign error. There the sign error is cancelled by sign errors in the first and fifth display on p.410.} which is readily proved by residue calculus. We now integrate the pre-Kuznetsov formula against
$$ \big(h(\tau + i/2)+ h(\tau - i/2)\big) \cosh(\pi \tau) \Gamma (1 - \kappa/2 + i\tau) \Gamma (1 - \kappa/2 - i\tau).$$
Our assumptions on $h$ ensure absolute convergence and the possibility to shift the contour up and down to $\Im t = \pm 1/2$ without crossing poles. In this way the
  $\delta$-term becomes
$$\delta_{m=n} \frac{e((1+\kappa)/4)}{4\pi (m+n)}\int_{-\infty}^{\infty} h(\tau) \Gamma(1/2 - \kappa/2 + i\tau)\Gamma(1/2 - \kappa/2 - i\tau) 2\tau\sinh(\pi \tau) d\tau.$$
For the Kloosterman term we insert the definition in \eqref{defF} getting
\begin{displaymath}
\begin{split}
- \sum_c \frac{K_{\kappa}(m, n, c)}{c^2}&\frac{\pi i}{2 x^{\kappa}} \int_{-\infty}^{\infty}\int_0^x h(\tau) \frac{\Gamma((1-\kappa)/2 + i\tau)\Gamma((1-\kappa)/2 - i\tau)}{\cosh(\pi \tau)} \\
&\sum_{\epsilon_1, \epsilon_2 \in \{\pm 1\}} J_{\epsilon_1 + \epsilon_2 2 i \tau}(y) \cos(\pi(\kappa/2 + \epsilon_2 i \tau))  ((1 - \kappa )/2+ \epsilon_1\epsilon_2  i \tau)
y^{\kappa - 1} dy \, d\tau
\end{split}
\end{displaymath}
We note that
\begin{displaymath}
\begin{split}
 y^{1-\kappa}\frac{d}{dy} \Big(\frac{J_{2i\tau}(y)}{y^{1-\kappa}}\Big) &= J'_{2i\tau}(y) - \frac{(\kappa - 1)J_{2i\tau}(y)}{y} \\
&= \frac{J_{2i\tau+1}(y)(i\tau + (1-\kappa)/2) + J_{2i\tau-1}(y)((1 - \kappa)/2 - i\tau)}{-2i\tau}
\end{split}
\end{displaymath}
where the last equality follows from the recurrence relations \cite[8.471.1\&2]{GR}. Substituting this, we can evaluate the $y$-integral by the fundamental theorem of calculus, arriving at
$$-\sum_c \frac{K_{\kappa}(m, n, c)}{c^2} \frac{\pi }{ x } \int_{-\infty}^{\infty} \tau h(\tau) \frac{\Gamma((1-\kappa)/2 + i\tau)\Gamma((1-\kappa)/2 - i\tau)}{\cosh(\pi r)}  F(x, 2\tau, -\kappa)   
  d\tau$$
after some elementary manipulations. We multiply the resulting expression by $\sqrt{mn}$ 
to obtain the first formula. Note that the $c$-sum is absolutely convergent by the power series expansion of the Bessel function contained in $F(x, 2t, -\kappa)$ and the fact that $h(\pm i/4) = 0$,  as we can shift the $t$-contour up and down to $|\Im t| = 1/2 - \varepsilon$.

There are two ways to derive the second formula from the first. We can either observe that in Proskurin's notation we can compute
$\langle \overline{\mathcal{U}_m(., s_1)}, \overline{\mathcal{U}_n(., \bar{s}_2)}\rangle$ instead of $\langle  \mathcal{U}_m(., s_1),  \mathcal{U}_n(., \bar{s}_2)\rangle$. This changes the signs of the coefficients in the spectral expansion and has the effect of changing $\kappa$ into $-\kappa$. Note that in this case we do not need the extra condition $h(\pm i/4) = 0$ to shift contours. Alternatively,  we use the following fact (cf.\ \cite[(4.17), (4.18), (4.27), (4.28), (4.64), p.507, p.509]{DFI}): the map  $$T_{\kappa} = ((\textstyle \frac{\kappa - 1}{2})^2 + t^2) X\Lambda_{\kappa}$$ with $X : f(z) \mapsto f(-\bar{z})$ and $\Lambda_{\kappa} = \kappa/2 + y(i\partial_x - \partial_y)$ the weight lowering operator is a bijective isometry between  weight $\kappa$ and weight $2-\kappa$ that exchanges positive and negative Fourier coefficients:
\begin{equation*}
\begin{split}
  T_{\kappa} & \Bigg(\sum_{n \not=0} b(n) e(nx)W_{\text{sgn}(n)\frac{\kappa}{2}, it}(4 \pi|n|y) \Bigg) \\
  & = \sum_{n \not= 0} b(-n) \text{sgn}(n)  \Big( \Big(\frac{\kappa - 1}{2}\Big)^2+t^2\Big)^{-\frac{1}{2}\text{sgn}(n)} e(nx)W_{\text{sgn}(n)\frac{2-\kappa}{2}, it}(4 \pi|n|y). 
  \end{split}
\end{equation*}
This yields again the second formula for the first. 
\end{proof}

In order to get a corresponding formula for the Kohnen space, we apply the $L$ operator
$$\frac{1}{2(1 + i^{2\kappa})} \sum_{w \, \text{mod } 4} \left( \begin{matrix} 1+w & 1/4\\ 4w & 1\end{matrix}\right)$$
to the formula. As in the case of the Petersson formula (Lemma \ref{lem1}), this has the effect that 
\begin{itemize}
\item the cuspidal term is restricted to forms in the Kohnen space;
\item the $\delta$-term is multiplied by $2/3$; the reason for the number 2/3 is that the dimension of the Kohnen space is 1/3 of the full space, but only half of the coefficients appear;
\item the Kloosterman sums \eqref{klooster2} are replaced with the Kloosterman sums \eqref{defKlo} and the Kloosterman term is also multiplied by 2/3.
\end{itemize}
The hardest part is to compute the Eisenstein coefficients. All of this  has been worked out in detail in \cite[Section 5]{AD}. The corresponding formula \cite[Theorem 5.3]{AD} can be inverted in the same way. In the following lemma let $\sum^+$ denote a sum over an orthonormal basis of weight 1/2 Maa{\ss} cusp forms in Kohnen's space. We recall the definition \eqref{basicL} and \eqref{basicL1} of $L(D, s)$ for a discriminant $D$. 

\begin{lemma}\label{lem10}  Let $\kappa = 1/2$ and let $n, m \in \Bbb{Z}$ that are congruent $0$ or $1$ modulo $4$. Let $h$ be an even function, holomorphic in $|\Im t| < 2/3$ with $h(t) \ll (1 + |t|)^{-4}$. Then 
\begin{displaymath}
 \begin{split}
&   \left.\sum_j \right.^{+}\frac{\sqrt{mn}\overline{b_j(m)} b_j(n)}{\cosh(\pi t_j)} h(t_j)   + \frac{1}{12}\int_{-\infty}^{\infty} \Big(\frac{n}{m}\Big)^{it} \frac{L(m, 1/2 - 2it)L(n, 1/2 + it)  }{|\zeta(1 + 4it)|^2\cosh(\pi t)  \Gamma(\frac{1+\kappa}{2} + it)\Gamma(\frac{1+\kappa}{2}- it)} h(t) dt \\
  &=  \frac{2}{3}\delta_{n=m}\int_{-\infty}^{\infty} h(t) t \sinh(\pi t)\Gamma \Big( \frac{1-\kappa}{2} + it\Big)\Gamma \Big( \frac{1-\kappa}{2} - it\Big)  \frac{dt}{4\pi^3} \\
  &+ \frac{2}{3}e\Big(\frac{1-\kappa}{4}\Big)\sum_c \frac{K^+_{\kappa}(m, n, c)}{c} \int_0^{\infty} \frac{F( 4\pi\sqrt{nm}/c, 2t, -\kappa)}{\cosh(\pi t)} \Gamma \Big( \frac{1-\kappa}{2} + it\Big)\Gamma \Big( \frac{1-\kappa}{2} - it\Big) h(t) t\, \frac{dt}{2\pi^2}
   \end{split}
\end{displaymath}
for $n, m > 0$ if in addition $h(\pm i/4) = 0$. Moreover, regardless of the value of $h(\pm i/4)$ we have 
\begin{displaymath}
 \begin{split}
&   \left.\sum_j\right.^{(+)} \frac{\sqrt{|mn|} \, \overline{b_j(m)} b_j(n)}{\cosh(\pi t_j)} h(t_j)   + \frac{1}{12}\int_{-\infty}^{\infty} \Big(\frac{|n|}{|m|}\Big)^{it} \frac{L(m, 1/2 - 2it)L(n, 1/2 + it)  }{|\zeta(1 + 4it)|^2\cosh(\pi t)\Gamma(\frac{1-\kappa}{2} + it)\Gamma(\frac{1-\kappa}{2}- it)} h(t) dt \\
  &= \frac{2}{3} \delta_{n=m}\int_{-\infty}^{\infty} h(t) t \sinh(\pi t)\Gamma\Big(\frac{1+\kappa}{2} + it\Big)\Gamma\Big(\frac{1+\kappa}{2}+ it\Big) \frac{dt}{4\pi^3} \\&+ \frac{2}{3}e\Big(\frac{1+\kappa}{4}\Big)\sum_c \frac{K^+_{-\kappa}(|m|, |n|, c)}{c} \int_0^{\infty} \frac{F( 4\pi\sqrt{|nm|}/c, 2t, \kappa)}{\cosh(\pi t)}\Gamma\Big(\frac{1+\kappa}{2} + it\Big)\Gamma\Big(\frac{1+\kappa}{2}- it\Big) h(t) t\, \frac{dt}{2\pi^2}
   \end{split}
\end{displaymath}
if $n, m < 0$. 
\end{lemma}


 


\section{Harmonic analysis on positive definite matrices}\label{sec6}

Here we prove Proposition \ref{prop1}. We identify the Hilbert spaces 
\begin{equation}\label{product}
   (\Bbb{H}, y^{-2} dx\, dy) \times (\Bbb{R}_{>0}, r^{-1} dr) \cong (\mathcal{P}(\Bbb{R}), (\det Y)^{-3/2} dY)
   \end{equation} via
$$\iota : (x + iy, r) \mapsto \sqrt{r}\left(\begin{matrix} y^{-1} & -xy^{-1}\\-x y^{-1} & y^{-1} (x^2 + y^2)\end{matrix}\right).$$
Note that for  $ Y   =     \iota(x+iy, r)$ with $ r = \det Y$ we have
$$\Big|\det \frac{dY}{d(r, x, y)} \Big|   = \frac{\sqrt{r}}{y^2},$$
so that the measures coincide. 
The group $\overline{\Gamma} = {\rm PSL}_2(\Bbb{Z})$ acts faithfully on $\mathcal{P}(\Bbb{R})$ and $\mathcal{P}(\Bbb{Z})$ by $T \mapsto U^{\top} T U$ for   $U \in  \overline{\Gamma}$. This is compatible with the action of $\overline{\Gamma}$ on $\Bbb{H}$ by M\"obius transforms.   Every   smooth function $f \in L^2(\Gamma\backslash \Bbb{H})$ has a spectral decomposition
$$f(z) = \sum_{j \geq 0} \langle f, u_j \rangle u_j(z) + \int_{-\infty}^{\infty} \langle f, E(., 1/2 + it)\rangle E(z, 1/2 + it) \frac{dt}{4\pi} = \int_{\Lambda} \langle f, {\tt u} \rangle {\tt u}(z) d{\tt u}. $$
 We recall the notion of the spectral parameter $t_{\tt u}$ for ${\tt u} \in \Lambda$; the constant function has spectral parameter $i/2$. 
Combining the spectral decomposition on $\Gamma \backslash \Bbb{H}$ with Mellin inversion, we conclude that, for a smooth  function $\Phi \in L^2(\Gamma \backslash \mathcal{P}(\Bbb{R}))$,   the spectral decomposition
$$\Phi(Y) = \Phi(\iota(x + iy, r)) =  \int_{(0)}  \int_{\Lambda} \langle \widehat{\Phi}(s), {\tt u} \rangle {\tt u}(x + iy) d {\tt u} \,\, r^{-s} \frac{ds}{2\pi i}$$
holds provided $\Phi$ is sufficiently rapidly decaying as $r\rightarrow 0$ and $r \rightarrow \infty$. Here,  $$\widehat{\Phi}(s)(x+iy) = \int_0^{\infty} \Phi(\iota(x + iy, r)) r^{s} \frac{dr}{r}$$
is the Mellin transform with respect to the $r$-variable. This gives the Parseval formula
\begin{equation}\label{pars}
\| \Phi(Y) \|^2  = \int_{\Gamma \backslash \mathcal{P}(\Bbb{R})} |\Phi(Y)|^2 \frac{dY}{(\det Y)^{3/2}} = \int_{-\infty}^{\infty}  \int_{\Lambda} |\langle \widehat{\Phi}(it), {\tt u} \rangle|^2 d{\tt u}\, \frac{dt }{2\pi}.
\end{equation}


For an automorphic  form  ${\tt u} \in \Lambda$   and a Siegel cusp form $F \in S^{(2)}_k$ with Fourier expansion \eqref{four1} we define the twisted Koecher--Maa{\ss} series by
\begin{equation}\label{KM}
L(F \times {\tt u}, s) := \sum_{T\in \mathcal{P}(\Bbb{Z})/\overline{\Gamma}} \frac{a(T)}{\epsilon(T)( \det T)^{s + (k-1)/2} }{\tt u}\left( (\det T)^{-1/2}T\right)
\end{equation}
where $\epsilon(T) = \{U \in \overline{\Gamma} \mid U^{\top} T U = T\}$ is the stabilizer and $\Re s$ is (initially) sufficiently large. This series is often defined in terms of ${\rm GL}_2(\Bbb{Z})$-equivalence instead of $\overline{\Gamma}$-equivalence, but for us the present version is more convenient. This function has no Euler product, but it does have a functional equation. Let
\begin{equation}\label{G-KM}
G(F \times {\tt u}, s) := 4(2\pi)^{-k+1-2s}\prod_{\pm} \Gamma\left(\frac{k-1}{2} +s - \frac{1}{4} \pm \frac{it_{\tt u}}{2}\right).
\end{equation}
Let us assume that $k$ is even. Then for all even ${\tt u}$ (including Eisenstein series and the constant function), the function $L(F \times {\tt u}, s)$ has an analytic continuation to $\Bbb{C}$ that is bounded in vertical strips and satisfies the functional equation \cite[Theorem 3.5]{Im}
\begin{equation*}
L(F \times {\tt u}, s) G(F \times {\tt u}, s) = L(F \times {\tt u}, 1-s) G(F \times {\tt u}, 1-s).
\end{equation*}
The functional equation is a consequence of the following period formula.   For $\Re s$ sufficiently large \cite[p.\ 927-928]{Im} we have
\begin{displaymath}
\begin{split}
\int_0^{\infty} & \langle F(i \cdot (., r)), {\tt u}\rangle r^{\frac{k-1}{2} + s} \frac{dr}{r}   = \int_{\overline{\Gamma} \backslash \Bbb{H}} \int_0^{\infty} F(i \cdot (z, r)) r^{\frac{k-1}{2} + s} \frac{dr}{r} \overline{{\tt u}}(z) \frac{dx \,dy}{y^2}\\
& = \int_{\overline{\Gamma}  \backslash \mathcal{P}(\Bbb{R})} \sum_{T \in \mathcal{P}(\Bbb{Z})} a(T) e^{-2\pi \text{tr}(YT)} (\det Y)^{\frac{k-1}{2} + s}\overline{{\tt u}}\left( (\det Y)^{-1/2}Y\right) \frac{dY}{(\det Y)^{3/2}}.
\end{split}
\end{displaymath}
Now splitting the $T$-sum into equivalence classes modulo $\overline{\Gamma}$ and unfolding the integral, this equals
$$ \sum_{T\in \mathcal{P}(\Bbb{Z})/\overline{\Gamma}} \frac{a(T)}{\epsilon(T) }  \int_{  \mathcal{P}(\Bbb{R})}   e^{-2\pi \text{tr}(YT)} (\det Y)^{\frac{k-1}{2} + s} \overline{{\tt u}}\left( (\det Y)^{-1/2}Y \right) \frac{dY}{(\det Y)^{3/2}}.$$
Note that it is important that the action of $\overline{\Gamma}$ is faithful. The last integral over $\mathcal{P}(\Bbb{R})$ was evaluated by Maa{\ss} \cite[p.\ 85 and p.\ 94]{Ma}:
$$\int_{  \mathcal{P}(\Bbb{R})}   e^{- \text{tr}(YT)} (\det Y)^{  s} \overline{{\tt u}}\left( (\det Y)^{-1/2}Y \right) \frac{dY}{(\det Y)^{3/2}} = \frac{\pi^{1/2}}{(\det T)^{s}} \overline{\tt u}\left((\det T)^{-1/2} T\right)\prod_{\pm}  \Gamma\Big(s - \frac{1}{4} \pm  \frac{it_{\tt u}}{2}\Big) $$
for any ${\tt u} \in \Lambda$. Thus we obtain 
$$\int_0^{\infty}  \langle F(i \cdot (., r)), {\tt u}\rangle r^{\frac{k-1}{2} + s} \frac{dr}{r} = \frac{\sqrt{\pi}}{4}  L(F \times \overline{{\tt u}}, s)   G(F \times \overline{{\tt u}}, s), $$
initially for $\Re s$ sufficiently large, but then by analytic continuation for all $s \in \Bbb{C}$. For odd ${\tt u}$, the left hand side vanishes, since $F(i \cdot (., r)), {\tt u}\rangle = 0$ for every $r$. 
The Parseval formula \eqref{pars} now implies Proposition \ref{prop1} in the introduction. \\
  

In the special case where $F$ is a Saito--Kurokawa lift, the Koecher--Maa{\ss} series simplifies. 

\begin{lemma}\label{lem7} If $k$ is even and $F = F_h \in S_k^{(2)}$ is a Saito--Kurokawa lift with Fourier expansion as in \eqref{four1} and \eqref{four2} and ${\tt u} \in \Lambda$, then
$$L(F_h \times {\tt u}, s)   =   4^{s + \frac{k-1}{2}} \zeta(2s) \sum_{\substack{D< 0\\ D \equiv 0, 1 \, (\text{{\rm mod  }}4)}} \frac{c_h(|D|)P(D; {\tt u}) }{|D|^{s + \frac{k-1}{2}}}$$
for $\Re s$ sufficiently large and $P(D; {\tt u})$ as in \eqref{defP}. 
\end{lemma}

\emph{Remark:} One would expect that $c_h(|D|)$ is roughly of size $|D|^{\frac{k}{2} - \frac{3}{4}}$ and that $P(D; {\tt u})$  is roughly of size $|D|^{\frac{1}{4}}$ (for non-constant ${\tt u}$) so that typically $c_h(|D|)P(D; {\tt u})|D|^{-\frac{k-1}{2}}$ is roughly of constant size (with respect to $D$). Using   trivial bounds for $P(D; {\tt u})$ and $c_h(|D|)$, the quantity $c_h(|D|)P(D; {\tt u})|D|^{-\frac{k-1}{2}}$ is certainly $\ll |D|^{3/4+\varepsilon}$.

\begin{proof}
We copy the argument from \cite[p.\ 22]{Bo}.   Let $\Delta$ be a negative fundamental discriminant, $D = \Delta f^2$ a negative discriminant   and $T \in  \mathcal{P}(\Bbb{Z})/\overline{\Gamma}$. For such $T  = \left(\begin{smallmatrix}  n & r/2\\ r/2 & m\end{smallmatrix}\right)$ we write $e(T) = (n, r, m)$. 
It follows  from \eqref{four2} that
\begin{displaymath}
\begin{split}
L(F_h \times {\tt u}, s) &= 4^{s + \frac{k-1}{2}} \sum_{D = \Delta f^2}  \sum_{t \mid f} \frac{1}{|D|^{s + \frac{k-1}{2}}} \sum_{\substack{  \text{det}(T) = - D/4\\ e(T) = t}} \frac{a(T)}{\epsilon(T)} {\tt u}((\det T)^{-1/2} T)\\
&= 4^{s + \frac{k-1}{2}} \sum_{D = \Delta f^2}  \sum_{d \mid t \mid f} \frac{d^{k-1} c_h(|D|/d^2)}{|D|^{s + \frac{k-1}{2}}} \sum_{\substack{  \text{det}(T) = - D/4\\ e(T) = t}} \frac{ {\tt u}((\det T)^{-1/2} T)}{\epsilon(T)} .
\end{split}
\end{displaymath}
Writing $t = t'd$ with $t' \mid f/d$,  we can evaluate the $t'$-sum getting
\begin{displaymath}
\begin{split}
L(F_h \times {\tt u}, s)  & = 4^{s + \frac{k-1}{2}} \sum_{D = \Delta f^2}  \sum_{d   \mid f} \frac{d^{k-1} c_h(|D|/d^2)}{|D|^{s + \frac{k-1}{2}}}   P(D/d^2, {\tt u}) =  4^{s + \frac{k-1}{2}} \zeta(2s) \sum_{D} \frac{c_h(|D|)P(D; {\tt u}) }{|D|^{s + \frac{k-1}{2}}}
\end{split}
\end{displaymath}
as desired.
\end{proof}

We combine Proposition \ref{prop1} and Lemma \ref{lem7} with \eqref{normF} and use the notation of these formulas to derive the following basic spectral formula for the restricted norm 
$\mathcal{N}(F_h)$ 
of a Saito--Kurokawa lift $F_h \in S_k^{(2)}$. 

\begin{prop}\label{prop2} Let $k$ be even and let $F = F_h \in S_k^{(2)}$ be a Saito--Kurokawa lift  with Fourier expansion as in \eqref{four1} and \eqref{four2}. Let $f_h \in S_{2k-2}$ denote the corresponding Shimura lift of $h$. 
 Then
\begin{displaymath}
 \mathcal{N}(F_h) = \frac{\pi^2}{90 } \cdot \frac{18\sqrt{\pi}}{2 L(f_h, 3/2)} \int_{-\infty}^{\infty}  \int_{\Lambda_{\text{\rm ev}}} |\mathcal{G}(k, t_{ {\tt u}}, 1/2 + it)\mathcal{L}(h, \overline{{\tt u}}, 1/2 + it)|^2 d{\tt u} \, dt,
\end{displaymath}
where
\begin{equation}\label{defG}
\begin{split}
\mathcal{G}(k, t_{\tt u}, s) &=  \pi^{  - 2s} \frac{2^k\Gamma(\frac{k-1}{2}  + s  - \frac{1}{4} + \frac{i t_{\tt u}}{2})\Gamma(\frac{k-1}{2}  + s - \frac{1}{4}  - \frac{i t_{\tt u}}{2})}{(\Gamma(k) \Gamma(k-3/2))^{1/2}}
\end{split}
\end{equation}
and $\mathcal{L}(h, {\tt u}, s) $ is the analytic continuation of 
\begin{equation}\label{curlyL}
\mathcal{L}(h, {\tt u}, s) = \left(\frac{\Gamma(k - 3/2)}{(4\pi)^{k-3/2}}\right)^{1/2} \zeta(2s) \sum_{\substack{D< 0\\ D \equiv 0, 1 \, (\text{{\rm mod  }}4)}} \frac{c_h(|D|)P(D; {\tt u}) }{ \| h \| \cdot |D|^{s + \frac{k-1}{2}}}
\end{equation}
for $\Re s$ sufficiently large. 
\end{prop}

The renormalized functions still satisfy the functional equation
$$\mathcal{L}(h,\overline{ {\tt u}} , s) \mathcal{G}(k,   t_{{\tt u}}, s) = \mathcal{L}(h ,\overline{{\tt u}}, 1-s) \mathcal{G}(k,   t_{ {\tt u}}, 1-s).$$
The inclusion of the gamma factor in \eqref{curlyL} is motivated by the formula in Lemma \ref{lem1}. \\


From the Dirichlet series expansion and the functional equation we obtain an approximate functional equation, cf.\    \cite[Theorem 5.3]{IK}. 

\begin{lemma}\label{approx1} Let $F = F_h\in S_k^{(2)}$ be a Saito--Kurokawa lift and ${\tt u}\in \Lambda$ with spectral parameter $t_{\tt u}$. 
For $t, x \in \Bbb{R}$ let 
\begin{equation}\label{defV3}
  V_t(x; k,   t_{\tt u}) := \frac{1}{2\pi i}\int_{(3)} e^{v^2}    \mathcal{G}(k, t_{\tt u}, v + 1/2 + it) \mathcal{G}(k, t_{\tt u}, v + 1/2 - it)  x^{-v}  \frac{dv}{v}.
  \end{equation}
Then
\begin{displaymath}
\begin{split}
& |\mathcal{G}(k, t_{\tt u}, 1/2 + it)\mathcal{L}(h,  {\tt u}, 1/2 + it)|^2 \\
&=2  \frac{\Gamma(k - 3/2)}{(4\pi)^{k-3/2}}  \sum_{f_1, f_2}  \sum_{D_1, D_2 < 0} \frac{c_h(|D_1|)c_h(|D_2|)P(D_1; {\tt u}) \overline{P(D_2; {\tt u}) }}{ \| h \| \cdot f_1^{1 +2it}f_2^{1-2it}|D_1|^{ k/2 + it}|D_2|^{ k/2 + it}}V_{ t}(|D_1D_2| (f_1f_2)^2; k,  t_{\tt u}). 
  \end{split}
\end{displaymath}      
\end{lemma}
Combining Proposition \ref{prop2}, Lemma \ref{approx1} and \eqref{1overL} we obtain the following basic formula:
\begin{equation}\label{Fh}
\begin{split}
 \mathcal{N}(F_h) &= \frac{\pi^2}{90} \cdot \frac{18\sqrt{\pi}}{2  }\cdot 2\sum_{(n, m) = 1} \frac{\lambda(n) \mu(n)\mu^2(m)}{n^{3/2}m^{3}}  \int_{-\infty}^{\infty}  \int_{\Lambda_{\text{\rm ev}}} \frac{\Gamma(k - 3/2)}{(4\pi)^{k-3/2} \| h \|^2}  \\
 &\sum_{f_1, f_2, D_1, D_2}  \frac{c(|D_1|) c(|D_2|) P(D_1; {\tt u})\overline{P(D_2; {\tt u})}}{f_1f_2 |D_1D_2|^{k/2}} \Big(\frac{|D_2|f_2^2}{|D_1|f_1^2}\Big)^{it} V_{t}(|D_1D_2|(f_1f_2)^2; k, t_{\tt u})   d{\tt u} \, dt.
\end{split}
\end{equation}

 \section{A relative trace formula}\label{secproof2}
 
 In this section we prove a slightly more general formula than that stated in Theorem \ref{thm2}. Let $D_1 = \Delta_1f_1^2, D_2 = \Delta_2f_2^2 < 0$ be two arbitrary negative discriminants and $h$ as specified in Theorem \ref{thm2}. Combining \eqref{mixed} (with $t/2$ in place of $t$) and \eqref{eisen1} we have
 \begin{displaymath}
 \begin{split}
  \frac{1}{|D_1D_2|^{1/4}} &\int_{\Lambda_{\text{ev}}} P(D_1; {\tt u}) \overline{P(D_2; {\tt u})} h(t_{\tt u}) d{\tt u} = \frac{3}{\pi} \frac{H(D_1) H(D_2)}{|D_1D_2|^{1/4}} h(i/2) \\
   &+  \sum_{u \text{ cuspidal, even}}  \frac{3}{\pi} L(u, 1/2) \Gamma(1/4 + it_u/2)\Gamma(1/4 - it_u/2) |D_1D_2|^{1/2} b(D_1)\overline{b(D_2)} h(t_u)\\
   &+ \int_{-\infty}^{\infty}\left( \frac{|D_1|}{|D_2|}\right)^{it/2}  |\zeta(1/2 + it)|^2 \frac{L(D_1, 1/2 + it)L(D_2, 1/2 - it)}{2|\zeta(1 + 2it)|^2} h(t) \frac{dt}{4\pi}.
 \end{split}
 \end{displaymath}
By  the argument of the proof of Lemma \ref{lem3} we can re-write the sum over even cusp forms $u$ as a sum over weight 1/2 cusp forms $v$ in Kohnen's space. Thus the last two terms  of the preceding display become
\begin{displaymath}
\begin{split}
& \left.\sum_{j}\right.^{(+)} \frac{3}{\pi} L(u_j, 1/2) \Gamma(1/4 + it_j)\Gamma(1/4 - it_j) |D_1D_2|^{1/2} b_j(D_1)\overline{b_j(D_2)} h(2t_j)\\
&+ \int_{-\infty}^{\infty}\left( \frac{|D_1|}{|D_2|}\right)^{it}  |\zeta(1/2 + 2it)|^2 \frac{L(D_1, 1/2 + 2it)L(D_2, 1/2 - 2it)}{2|\zeta(1 + 4it)|^2} h(2t) \frac{dt}{2\pi}.
\end{split}
\end{displaymath}
Here, as in Lemma \ref{lem10}, $\sum^{(+)}$ indicates a sum over an orthonormal basis of weight 1/2 cusp forms $v_j$ in Kohnen's space with spectral parameter $t_j$ and Fourier coefficients $b_j(D)$, and $u_j$ is the corresponding Shimura lift with spectral parameter $2t_j = t_u$.  If $\lambda_j(n)$ are the Hecke eigenvalues of $u_j$, then by the   approximate functional equations \eqref{approx-basic}  and \eqref{approx-basic1}  with $\Delta = 1$ 
we can write 
$$L(u_j, 1/2) = 2\sum_n \frac{\lambda_j(n)}{n^{1/2}} W_{2t_j}(n)$$
and $$
 |\zeta(1/2 + 2it)|^2 = 2\sum_n \frac{\rho_{1/2 + 2it}(n)}{n^{1/2}} W_{2t}(n) -\sum_{\pm} \frac{\zeta(1 \pm 4it)\Gamma(\frac{1}{2} \pm 2it)\pi^{\mp 2it} e^{(1/2 \pm 2 it)^2}}{(\frac{1}{2} \pm 2i t)\Gamma(\frac{1}{4} + it)\Gamma(\frac{1}{4}  - it)}  $$
with $W_t$ as in \eqref{v-t} with $\mathfrak{a} = 0$. Note that $W_{2t}(x)$ is even and holomorphic in $|\Im t | <2/3$, satisfying the uniform bound $W_{2t}(x)\ll   (1+|t|^2)/x^2 $  in this region (by trivial bounds).  Moreover $W_{2t}(x)$ vanishes at $t = \pm i/4$. 
We first deal with  residue term in the formula for $|\zeta(1/2 + 2it)|^2$ and substitute this back into the Eisenstein term. This gives
$$- \sum_{\pm} \int_{-\infty}^{\infty} \left( \frac{|D_1|}{|D_2|}\right)^{it}   \frac{ \Gamma(\frac{1}{2} \pm 2it)\pi^{\mp 2it} e^{(1/2 \pm 2 it)^2}}{(\frac{1}{2} \pm 2i t)\Gamma(\frac{1}{4} + it)\Gamma(\frac{1}{4}  - it)}  \frac{L(D_1, 1/2 + 2it)L(D_2, 1/2 - 2it)}{2\zeta(1 \mp 4it)} h(2t) \frac{dt}{2\pi}.$$
We can slightly simplify this by applying in the plus-term  the functional equation for $L(D_1, 1/2 + 2it)$ and changing $t$ to $-t$, and in the minus-term   the functional equation \eqref{Lfuncteq} for $L(D_2, 1/2 - 2 i t)$. In this way we see that the two terms are equal and after some simplification we obtain
$$2  \int_{-\infty}^{\infty} |D_1D_2|^{ it}   \frac{ 2^{-\frac{3}{2} - 2 i t}\Gamma(-\frac{1}{4} + it)   e^{(1/2- 2 it)^2}}{ \sqrt{\pi} \Gamma(\frac{1}{4} + i t)}  \frac{L(D_1, 1/2 + 2it)L(D_2, 1/2 + 2it)}{2\zeta(1+ 4it)} h(2t) \frac{dt}{2\pi}.$$
Our next goal is to use the half-integral weight Hecke relations to combine $\lambda_j(n) b_j(D_1)$. For notational simplicity let us define $\tilde{b}_j(D_1) = \sqrt{|D_1|}b_j(D_1)$. From \eqref{non-fund} we obtain
\begin{displaymath}
\begin{split}
\lambda(n) \tilde{b}_j(D_1)& = \tilde{b}_j(D_1) = \tilde{b}(\Delta_1) \sum_{d \mid f_1} \frac{\mu(d) \chi_{\Delta_1}(d)}{d^{1/2}} \lambda\Big(\frac{f_1}{d}\Big) \lambda(n) 
  = \tilde{b}_j(\Delta_1) \sum_{\substack{d_1rs = f_1\\ r \mid n}} \frac{\mu(d_1) \chi_{\Delta_1}(d_1)}{d_1^{1/2}}   \lambda\Big(\frac{sn}{r}\Big) \\
 & =    \sum_{\substack{d_1rs = f_1\\ r \mid n}} \frac{\mu(d_1) \chi_{\Delta_1}(d_1)}{d_1^{1/2}}    \sum_{m \mid sn/r} \frac{\chi_{\Delta_1}(m)}{\sqrt{m}} \tilde{b}_j\Big(\Delta_1 \Big(\frac{sn}{rm}\Big)^2\Big).
  \end{split}
\end{displaymath} 
In the last step we used \eqref{non-fund} again together with M\"obius inversion.  This yields
 \begin{displaymath}
\begin{split}
 &\sum_n \frac{\lambda_j(n)}{n^{1/2}} W_{2t}(n)\tilde{b}_j(D_1) = \sum_{ d_1rs = f_1 }\frac{\mu(d_1) \chi_{\Delta_1}(d_1)}{d_1^{1/2}}  \sum_n \frac{W_{2t}(rn)}{\sqrt{rn}} \sum_{m \mid sn} \frac{\chi_{\Delta_1}(m)}{\sqrt{m}} \tilde{b}_j\Big(\Delta_1 \Big(\frac{sn}{m}\Big)^2\Big) \\
&=  \sum_{ d_1rs = f_1 }\frac{\mu(d_1) \chi_{\Delta_1}(d_1)}{d_1^{1/2}}  \sum_{n, m} \frac{W_{2t}(rn m/(m, s))}{\sqrt{rn m/(m, s)}}  \frac{\chi_{\Delta_1}(m)}{\sqrt{m}} \tilde{b}_j\Big(\Delta_1 \Big(\frac{sn}{(m, s)}\Big)^2\Big) \\
&=  \sum_{ d_1rs = f_1 }\frac{\mu(d_1) \chi_{\Delta_1}(d_1)}{d_1^{1/2}} \sum_{\tau u =  s} \sum_{n}\sum_{(m, u) = 1} \frac{W_{2t}(rn m)}{\sqrt{rn m }}  \frac{\chi_{\Delta_1}(m\tau)}{\sqrt{m\tau}} \tilde{b}_j(\Delta_1 (un)^2)\\
&=  \sum_{ d_1rs = f_1 }\frac{\mu(d_1) \chi_{\Delta_1}(d_1)}{d_1^{1/2}} \sum_{\tau u =  s} \sum_{n, m} \sum_{v\mid u} \mu(v)  \frac{W_{2t}(rn vm)}{\sqrt{rn vm }}  \frac{\chi_{\Delta_1}(vm\tau)}{\sqrt{vm\tau}} \tilde{b}_j(\Delta_1 (un)^2)\\
&=  \sum_{ d_1r\tau vw = f_1 }\sum_{n, m}  \frac{\mu(d_1)\mu(v)  \chi_{\Delta_1}(d_1vm\tau)}{\sqrt{d_1rn\tau} vm}      W_{2t}(rn vm) \tilde{b}_j(\Delta_1 (vwn)^2). 
 \end{split}
\end{displaymath}  
Comparing \eqref{non-fund} and \eqref{basicL1}, we see that the same Hecke relations hold for Eisenstein series, and we therefore have
\begin{displaymath}
\begin{split}
&\sum_n \frac{\rho_{1/2 + 2it}(n)}{n^{1/2}} W_{2t}(n) |D_1|^{it} L(D_1, 1/2 + 2 it) \\
&=  \sum_{ d_1r\tau vw = f_1 }\sum_{n, m}  \frac{\mu(d_1)\mu(v)  \chi_{\Delta_1}(d_1vm\tau)}{\sqrt{d_1rn\tau} vm}      W_{2t}(rn vm) (\Delta_1(vwn)^2)^{it} L(\Delta_1 (vwn)^2, 1/2 + 2it).
\end{split}
\end{displaymath}
We are now in a position to apply  the second Kohnen-Kuznetsov formula in Lemma \ref{lem10} 
with
$$\frac{6}{\pi} \Gamma(1/4 + it)\Gamma(1/4 - it) \cosh(\pi t) W_{2t}(rnvm) h(2t)$$
in place of $h(t)$. This function satisfies the hypotheses of that formula (and decays rapidly enough in $n$ and $m$), recall that $W_{2t}(n)$ vanishes at $t = \pm i/4$.  The diagonal exists only if $\Delta_1 = \Delta_2$ and $vwn = f_2$. 
For a function $H$ we introduce the integral transform 
\begin{equation}\label{h0}
H^{\dagger}(x) =   \int_{-\infty}^{\infty}  \frac{F(x, t, 1/2)}{\cosh(\pi t)}  H(t)   t \frac{dt}{\pi}
\end{equation}
with $F$ as in  \eqref{defF}.

\begin{theorem}\label{thm5} Let  Let $D_1 = \Delta_1 f_1^2, D_2 = \Delta_2f^2_2$ be negative  discriminants and let $h$ be an even function, holomorphic in $|\Im t | < 2/3$ with $h(t) \ll (1+|t|)^{-10}$. Define $W_t$ as in \eqref{v-t},  $L(D, s)$ as in \eqref{basicL} and $K^+_{3/2}(a, b, c)$ as in \eqref{defKlo}. Then
\begin{displaymath}
\begin{split}
&\frac{1}{|D_1D_2|^{1/4}}\int_{\Lambda_{\text{ev}}} P(D_1; {\tt u}) \overline{P(D_2; {\tt u})} h(t_{\tt u}) d{\tt u} = \frac{3}{\pi} \frac{H(D_1)H(D_2)}{|D_1D_2|^{1/4}} h(i/2)\\
& \quad +   \int_{-\infty}^{\infty} \Big|\frac{D_1D_2}{4}\Big|^{ it/2}   \frac{  \Gamma(-\frac{1}{4} + \frac{it}{2})   e^{(1/2- it)^2}}{ \sqrt{8\pi} \Gamma(\frac{1}{4} + \frac{i t}{2})}  \frac{L(D_1, 1/2 + it)L(D_2, 1/2 + it)}{ \zeta(1+ 2it)} h(t) \frac{dt}{4\pi}\\
& \quad + \delta_{\Delta_1= \Delta_2} \sum_{\substack{d_1r \tau v w = f_1\\ vwn = f_2}} \sum_m \frac{\mu(d_1)\mu(v)  \chi_{\Delta_1}(d_1vm\tau)}{\sqrt{d_1rn\tau} vm}   
 \int_{-\infty}^{\infty}  W_{t}( rnvm) h(t) t \tanh(\pi t)  \frac{dt}{4\pi^2} \\
&\quad + e(3/8)\sum_{ d_1r\tau vw = f_1 }\sum_{n,  c, m}  \frac{\mu(d_1)\mu(v)  \chi_{\Delta_1}(d_1vm\tau)}{\sqrt{d_1rn\tau} vm}     \frac{  K_{3/2}^+(|\Delta_1|(vwn)^2, |D_2|, c)}{ c} H_{ rnvm}^{\dagger}\Big(\frac{4\pi vwn\sqrt{|\Delta_1D_2|}}{c}\Big)\\
\end{split}
\end{displaymath}
where $H_{b}(t) = h(t)W_{t}(b)$ and $H^{\dagger}$ is given by \eqref{h0}. 
\end{theorem}

\emph{Remarks:}\\
 1) Specializing $f_1 = f_2 = 1$ we obtain Theorem \ref{thm2}. \\
2) Recall again that $W_t(x) = 0$ for $t = \pm i/2$, so that contour shifts in the $t$-integral ensure that the $c, n$-sum is absolutely convergent.\\
3) The first term on the right hand side corresponds to the constant function and is obviously indispensable. In all practical applications, the $t$-integral in the second term is of bounded length due to the factor $\exp((1/2 - it)^2)$, so that by subconvexity bounds for $L(D, 1/2 + it)$ this term is dominated by the class number term. The last term can be analyzed as the off-diagonal term in the Kuznetsov formula except that it contains an extra $n$-sum of length $\approx |t|$ from the approximate functional equation, cf.\ \eqref{bound-wt} below.  Contrary to its appearance, the diagonal term is symmetric in $f_1, f_2$ (as it should be) and can be written as
$$\delta_{\Delta_1 = \Delta_2} \int_{-\infty}^{\infty} \int_{(2)} \frac{\Gamma(\frac{1}{2}(\frac{1}{2} + s +2 it))\Gamma(\frac{1}{2}(\frac{1}{2} + s - 2it))}{\Gamma(\frac{1}{4}+it)\Gamma(\frac{1}{4}-it) \pi^{s} }\frac{e^{s^2} }{s} L(\chi_{\Delta_1}, s)P(f_1, f_2, s)  h(t) t \tanh(\pi t)   \frac{ds}{2\pi i} \frac{dt}{4\pi^2} $$
where
$$P(f_1, f_2, s) = \prod_p  \frac{p^{-\alpha_p(s + 1/2)} - p^{(\beta_p - 2)(s + 1/2)} - \chi_{\Delta_1}(p)p^{2s-1/2}(p^{-\alpha_p(s+1/2)} - p^{-\beta_p(s+1/2)} ) }{1 - p^{2s + 1}}$$
with $\alpha_p = \max(v_p(f_1), v_p(f_2))$ and $\beta_p = \min(v_p(f_1), v_p(f_2))$ for  the usual $p$-adic valuation $v_p$.

\section{Mean values of \texorpdfstring{$L$}{L}-functions}\label{proof}

This section is devoted to the proof of Proposition \ref{Lfunc}. To begin with, we recall Heath-Brown's large sieve in two variations \cite[Corollaries 3 \& 4]{HB}\footnote{In the original version of \cite[Corollary 4]{HB}, the $n$-sum is restricted to odd numbers $n$, but in the case of fundamental discriminants $\Delta$, the symbol $(\frac{\Delta}{n})$ is also defined for even $n$, and the proof works in the same way.}

\begin{lemma}\label{HBlemma} {\rm a)} Let $N, Q \geq 1$, let $S(Q)$ denote the set of real primitive characters of conductor up to $Q$ and let $(a_n)$ be a sequence of complex numbers with $|a_n| \leq 1$. Then
$$\sum_{\chi \in S(Q)} \Big|\sum_{n \leq N} a_n\chi(n)\Big|^2 \ll N(Q+N) (QN)^{\varepsilon}$$
for every $\varepsilon > 0$. 

{\rm b)}  Let $D, N \geq 1$, $(a_n), (b_{\Delta})$ be two sequences of complex numbers with $|a_m|, |b_{\Delta}| \leq 1$, where $b_{\Delta}$ is supported on the set of fundamental discriminants $\Delta$. Then 
$$\sum_{|\Delta| \leq D} \sum_{n \leq N}a_n b_{\Delta}  \chi_{\Delta}(n)  \ll (DN)^{1+\varepsilon} (D^{-1/2} + N^{-1/2}).$$
 \end{lemma}

We start with part (a) of the proposition. As mentioned in the introduction, $L(u \times \chi, 1/2) = 0$ if $u$ is odd for root number reasons. For even $u$ we use the approximate functional equation \eqref{approx-basic} and write
\begin{equation}\label{approx}
L({  u} \times \chi_{\Delta}, 1/2 )  =  \frac{1}{2\pi i}\int_{(2)}  \sum_{n} \frac{\lambda_{  u}(n) \chi_{\Delta}(n)}{n^{1/2 + s}} |\Delta|^s G(s, t_{  u} ) ds
\end{equation}
with
$$G(s, t_{  u} )  =   \frac{2  e^{s^2}\Gamma((1/2 + \mathfrak{a} + s + it_{  u})/2)\Gamma((1/2 + \mathfrak{a} + s - it_{  u})/2)}{\Gamma((1/2 + \mathfrak{a}+ it_{  u})/2)\Gamma((1/2 + \mathfrak{a}- it_{  u})/2)\pi^s s} $$
where $\mathfrak{a} = 1$ if $\Delta < 0$ and $\mathfrak{a} = 0$ if $\Delta > 0$. We can treat positive and negative discriminants separately, so that we may assume that $G(s, t_{  u} ) $ is independent of $\Delta$.  
Eventually we would like to apply the Cauchy-Schwarz inequality,  the spectral large sieve inequality and Heath-Brown's large sieve for quadratic characters. The latter requires that the $n$-sum is restricted to odd squarefree integers. Therefore we uniquely factorise  $n = 2^{\alpha} n_1n_2^2$ with $n_1, n_2$ odd, $n_1$ squarefree and use the Hecke relations to write
\begin{displaymath}
\begin{split}
&\sum_{n  }\frac{\lambda_{  u}(n) \chi_{\Delta}(n)}{n^{1/2 + s}} = \sum_{\alpha}\frac{\lambda_{  u}(2^{\alpha}) \chi_{\Delta}(2^{\alpha})}{2^{\alpha(1/2 + s)}} \sum_{2 \nmid n_1 } \frac{\mu^2(n_1)\lambda_{  u}(n_1) \chi_{\Delta}(n_1)}{n_1^{1/2 + s}} \sum_{(n_2, 2\Delta) = 1 } \frac{ \lambda_{  u}(n_2^2) }{n_2^{1 + 2s}} \sum_{2 \nmid d}   \frac{\mu(d) \chi_{\Delta}(d)}{d^{3/2 + 3s}}.\\
\end{split}
\end{displaymath}
We use M\"obius inversion to detect the condition $(n_2, \Delta) = 1$ and we observe that the $\alpha$-sum depends only on $\Delta$ modulo 8. Therefore
\begin{displaymath}
\begin{split}
 & \sum_{     |\Delta| \leq \mathcal{D}}   \sum_{n  } \frac{\lambda_{  u}(n) \chi_{\Delta}(n)}{n^{1/2 + s}} |\Delta|^s \\
  &=  \sum_{2 \nmid d  f}  \frac{\mu(d)\mu(f) \chi_f(d)}{d^{3/2 +3s}f^{1+s}} \sum_{\delta \in \{0, 1, 4, 5\}}  P(s; {  u}, \delta, f)     \sum_{  \substack{   |\Delta'| \leq\mathcal{D}/f\\ \Delta' \equiv \bar{f}\delta\, (\text{mod } 8)}}  \sum_{ 2 \nmid n_1 } \frac{\mu^2(n_1) \chi_f(n_1)\lambda_{  u}(n_1) (\frac{\Delta'}{n_1d})}{n_1^{1/2 + s}}  |\Delta'|^s \end{split}
\end{displaymath}
where $\Delta' f$ is restricted to   negative fundamental discriminants and  
$$P(s; {  u}, \delta, f)  :=  \sum_{\alpha}\frac{\lambda_{  u}(2^{\alpha}) \chi(\delta, 2^{\alpha})}{2^{\alpha(1/2 + s)}} \sum_{2 \nmid n_2 } \frac{ \lambda_{  u}(f^2n_2^2) }{n_2^{1 + 2s}} \ll f^{2\theta+\varepsilon}$$
uniformly in $\Re s \geq \varepsilon$ for $\theta = 7/64$ by the Kim-Sarnak bound. In the above formula we define  $\chi(\delta, 2^{\alpha}) := \chi_{\Delta}(2^{\alpha})$ for any $\Delta \equiv \delta$ (mod 8).)

We substitute this back into \eqref{approx}.  Shifting the contour to the far right, we can truncate the $n_1$-sum at $n_1  \leq  ( \mathcal{D}\mathcal{T})^{1+\varepsilon} $ for $|t_u| \leq \mathcal{T}$ at the cost of a negligible error.  Having done this, we shift the contour back to $\Re s = \varepsilon $, truncate the integral  at $|\Im s| \leq  ( \mathcal{D}\mathcal{T})^{\varepsilon}$, again with a negligible error,  so that 
\begin{displaymath}
\begin{split}
\sum_{t_{ u} \leq \mathcal{T}}& \sum_{\substack{|\Delta| \leq \mathcal{D}\\ \Delta \text{ {\rm  fund.  discr.}}}} \alpha(u)  L({u} \times \chi_{\Delta}, 1/2)  \ll     ( \mathcal{D}\mathcal{T})^{O(\varepsilon)}  \sup_{ \substack{N \leq  ( \mathcal{D}\mathcal{T})^{1+\varepsilon}\\ \Re s = \varepsilon}} \sum_{2 \nmid df} \frac{\mu^2(d)}{d^{3/2+\varepsilon }f^{1-2\theta}} \\
    & \times \sum_{t_{  u} \leq \mathcal{T}} \Bigg| \alpha(u)   \sum_{\substack{ N \leq n_1 \leq 2N\\ 2 \nmid n_1} }  \frac{\mu^2(n_1) \chi_f(n_1)\lambda_{  u}(n_1)}{n_1^{1/2 + s}}   \sum_{  \substack{  |\Delta'| \leq \mathcal{D}/f\\ \Delta' \equiv \bar{f}\delta\, (\text{mod } 8)}}\Big(\frac{\Delta'}{n_1d}\Big) |\Delta'|^s    \Bigg| .
    \end{split}
    \end{displaymath}
    A priori, the right hand side is restricted to even $u$, but by positivity we can extend it to all $u$. 
Next we apply the Cauchy-Schwarz inequality.   
In the second factor we artificially insert $1/L(\text{sym}^2{  u}, 1)$ at the cost of a factor of $\mathcal{T}^{\varepsilon}$ (by \eqref{lower}) to convert the Hecke eigenvalues into Fourier coefficients and apply 
  the spectral large sieve inequality \cite[Theorem 2]{DesIw}. This leaves us with bounding 
    \begin{equation}\label{largesieve}
 \begin{split}
&     ( \mathcal{D}\mathcal{T})^{O(\varepsilon)}  \Big(\sum_{t_u \leq \mathcal{T}} |\alpha(u)|^2\Big)^{1/2} \sum_{2 \nmid df} \frac{\mu^2(d)}{d^{3/2+\varepsilon }f^{1-2\theta}}  \Bigg( \frac{\mathcal{T}^2 + N}{N}  \sum_{\substack{ N \leq n_1 \leq  2N \\ 2 \nmid n_1}}  \mu^2(n_1)  \Big| \sum_{  \substack{   |\Delta'| \leq \mathcal{D}/f\\ \Delta' \equiv \bar{f}\delta\, (\text{mod } 8)}} \Big(\frac{\Delta'}{n_1d}\Big)|\Delta'|^s  \Big|^2\Bigg)^{1/2}
    \end{split}
  \end{equation}
  for $N \leq  ( \mathcal{D}\mathcal{T})^{1+\varepsilon}$. 
 For a given odd, squarefree   $d \in\Bbb{N}$, the $n_1$-sum equals
 \begin{displaymath}
 \begin{split}
& \sum_{ r_1r_2 = d }  \sum_{\substack{N/r_1 \leq   n_1 \leq 2N/r_1\\ ( n_1, 2 r_2) = 1}}  \mu^2(r_1n_1)  \Big| \sum_{  \substack{   |\Delta'| \leq \mathcal{D}/f\\ \Delta' \equiv \bar{f}\delta\, (\text{mod } 8)\\ (\Delta', r_1) = 1}} \Big(\frac{\Delta'}{n_1r_2}\Big) |\Delta'|^s  \Big|^2\\
& \leq  \sum_{ r_1r_2 = d }  \sum_{\substack{ n \leq 2Nr_2\\ 2 \nmid n}}  \tau(n) \mu^2(n)  \Big| \sum_{  \substack{    |\Delta'| \leq \mathcal{D}/f\\ \Delta' \equiv \bar{f}\delta\, (\text{mod } 8)\\ (\Delta', r_1) = 1}} \Big(\frac{\Delta'}{n}\Big) |\Delta'|^s  \Big|^2.
 \end{split} 
 \end{displaymath}     
For odd squarefree $n \not= 1$, the map $\Delta' \mapsto (\frac{\Delta'}{n})$ is a primitive quadratic character of conductor $n$, so that by Heath-Brown's large sieve for quadratic characters (Lemma \ref {HBlemma}a) this expression is bounded by 
$$\ll (\mathcal{D}\mathcal{T})^{\varepsilon}\Big(Nd + \frac{\mathcal{D}}{f}\Big) \frac{\mathcal{D}}{f}.$$
Putting everything together, we complete the proof of Proposition \ref{Lfunc}(a).\\ 

The proof of part (b) is almost identical except that the spectral large sieve is replaced with the standard bound \cite[Theorem 9.1]{IK} for Dirichlet polynomials. Here we use the approximate functional equation
$$L( \chi_{\Delta}, 1/2 + i t )^2  =  \frac{1}{2\pi i}\int_{(2)} \Big( \sum_{n} \frac{\tau(n) \chi_{\Delta}(n)}{n^{1/2 + it + s}} |\Delta|^s \tilde{G}(s, t  ) + \epsilon(t)\sum_{n} \frac{\tau(n) \chi_{\Delta}(n)}{n^{1/2 - it + s}} |\Delta|^s \tilde{G}(s, -t  ) \Big) ds$$
where $|\epsilon(t)| = 1$ and 
\begin{equation*}
\tilde{G}(s, t) =  \frac{ e^{s^2}\Gamma(1/2 +\mathfrak{a} + s + it/2)^2}{\Gamma(1/2+\mathfrak{a} + it/2)^2 \pi^s s}  \frac{1}{(\frac{1}{4} + t^2)^2} \prod_{\epsilon_1, \epsilon_2 \in \{\pm 1\}} \hspace{-0.3cm}\Big(\epsilon_1 s - \Big(\frac{1}{2} +\epsilon_2 it\Big)\Big).
\end{equation*}
We included the polynomial in order to counteract the pole at $s=1/2 - it$ of $L(\chi_{\Delta}, s + 1/2 + it)^2$ for $\Delta = 1$, so that no residual term arises in the approximate functional equation. 
The function $\tilde{G}$ has similar analytic properties as $G$ above, and the divisor function $\tau$ satisfies the same Hecke relations as $\lambda_u$. The proof is now almost literally the same, except that the   factor $(\mathcal{T}^2 + \mathcal{N})/\mathcal{N}$ in \eqref{largesieve} is $(\mathcal{T} + \mathcal{N})/\mathcal{N}$. Thus the proof of Proposition \ref{Lfunc} is concluded.

\section{Interlude: special functions and oscillatory integrals}
In this  rather technical section we compile various sums and integrals over Bessel functions and other oscillatory integrals that we need as a preparation for the proof of Theorem \ref{thm1}. To start with, the following lemma is a half-integral weight version of  \cite[Lemma 5.8]{Iw1}, but with a somewhat different proof. 
\begin{lemma}\label{lem2}
 Let $x > 0$ $  A \geq 0$, $K > 1$.  Let $w$ be a  smooth function with   support in $[1, 2]$ satisfying $w^{(j)}(x) \ll_{\varepsilon} K^{j\varepsilon}$ for $j \in \Bbb{N}_0$. Then there exist smooth  functions $w_0, w_+, w_-$ such that for every $j  \in \Bbb{N}_0$ we have 
$$w_0(x) \ll_A K^{-A},$$ 
\begin{equation}\label{boundsw}
\frac{d^j}{dx^j}w_{\pm}(x) \ll_{j, A }  \Big(1 + \frac{K^2}{x}\Big)^{-A} \frac{1}{x^j}
\end{equation}
and
\begin{equation}\label{bessel-sum}
\sum_{k\text{ {\rm  even}}} i^k w\left(\frac{k}{K}\right) J_{k-3/2}(x) = \sum_{\pm} e^{\pm ix} w_{\pm}(x)  + w_0(x).
\end{equation}
The implied constants in \eqref{boundsw} depend on the $B$-th Sobolev norm of $w$ for a suitable $B= B(A, j)$. The functions  $w_{\pm}$ are explicitly given in \eqref{W2pm} and \eqref{W2m}. 
\end{lemma}

\begin{proof}
We denote the left hand side of  \eqref{bessel-sum} by $W(x)$.  For $k \in \Bbb{N}$, by \cite[8.411.13]{GR} we have
\begin{equation*}
J_{k - 3/2}(x) = \int_{-1/2}^{1/2} e(k \theta) e(-3\theta/2) e^{-ix \sin(2\pi \theta)} d\theta - \frac{(-1)^k}{\pi} \int_0^{\infty} e^{-(k - 3/2)\theta - x \sinh\theta} d\theta.
\end{equation*}
We put
$$W_1(x) =- \sum_{k\text{ {\rm  even}}} i^k w\left(\frac{k}{K}\right)\frac{1}{\pi} \int_0^{\infty} e^{-(k - 3/2)\theta - x \sinh\theta} d\theta,$$
$$W_1^{\pm}(x) = \frac{1}{2}e^{\mp i x} \sum_{k \in \Bbb{Z}} (\pm i)^k w\left(\frac{k}{K}\right) \int_{-1/2}^{1/2} e(k \theta) e(-3\theta/2) e^{-ix \sin(2\pi \theta)} d\theta. $$
Applying Poisson summation modulo 4, we have
$$ \sum_{k\text{ even}} i^k w\left(\frac{k}{K}\right) e^{-k\theta} = \frac{1}{4} \sum_{\substack{\kappa \, (\text{mod }4)\\ \kappa \equiv 0, 2 \, (\text{mod } 4)}} i^{\kappa} \sum_{h \in \Bbb{Z}} e\left(\frac{h\kappa}{4}\right) \int_{-\infty}^{\infty} w\left(\frac{y}{K}\right) e^{-y\theta} e\left(\frac{y h}{4}\right) dy.$$
 The $h = 0$ term vanishes (as do all even $h$), and by partial integration the other terms are bounded by $O_n( K |h|^{-n}( K^{-(1-\varepsilon)n} + \theta^n) e^{-K\theta})$ 
for any $n \in \Bbb{N}$, so that   $$W_1(x)   \ll_n  K \int_0^{\infty} (K^{-(1-\varepsilon)n} + \theta^n) e^{-K\theta} e^{-x \sinh \theta} d\theta \ll_n K^{-(1-\varepsilon)n}.$$
Again by Poisson summation we have
\begin{displaymath}
\begin{split}
W^+_1(x) & =\frac{1}{2}e^{- i x} \int_{-1/2}^{1/2} e(-3\theta/2) e^{-ix \sin(2\pi \theta)}     \sum_{h \in \Bbb{Z}} \int_{-\infty}^{\infty} e\left(\frac{y}{4}\right) w\left(\frac{y}{K}\right) e(y\theta) e(-hy) dy \, d\theta\\
& = \frac{e(3/8)}{2}e^{- i x} \int_{-1/4}^{3/4} e(-3\theta/2) e^{ix \cos(2\pi \theta)}     \sum_{h \in \Bbb{Z}} \int_{-\infty}^{\infty}  w\left(\frac{y}{K}\right) e(y\theta) e(-hy) dy \, d\theta.
\end{split}
\end{displaymath}
Since $-1/4 < \theta  \leq 3/4$, we see by partial integration in the $y$-integral that the contribution of $h \not= 0$ is $O_A(K^{-A})$. Let $v$ be a smooth function with compact support in $[-2, 2]$, identically equal to $1$ on $[-1, 1]$. Then for $h = 0$ we can smoothly truncate the $\theta$-integral by inserting the function $v(\theta K^{9/10})$, the error being again $O_A(K^{-A})$  by partial integration. We obtain $W_1^+(x)  = W_2^+(x) + W_2(x)$, where  $W_2(x) \ll_A K^{-A}$ and  after changing variables 
\begin{equation}\label{W2pm}
\begin{split}
W_2^+(x)&  = \frac{e(3/8)}{2}e^{- i x} \int_{-\infty}^{\infty} v(\theta K^{-1/10}) e\Big(\frac{-3\theta}{2K}\Big) e^{ix \cos(2\pi \theta/K)}    \int_{-\infty}^{\infty}  w(y) e(y\theta)   dy \, d\theta.
\end{split}
\end{equation}
Then for $j \in \Bbb{N}_0$ we have
\begin{equation}\label{dW}
\begin{split}
\frac{d^j}{dx^j}W_2^+(x) = \frac{e(3/8)}{2}   \int_{-\infty}^{\infty}   w(y) \int_{-\infty}^{\infty}  v(\theta K^{-1/10})  e^{i \phi( \theta; x, y)} (i(\cos(2\pi \theta/K) - 1))^j d\theta \, dy
\end{split}
\end{equation}
with $\phi(\theta; x, y) = -3\pi \theta/K+ 2\pi \theta y + x ( \cos(2\pi \theta/K)-1)$  satisfying  $$\frac{d}{d\theta} \phi(\theta; x, y) = -\frac{3\pi}{K} + 2\pi y + 2\pi \frac{x}{K} \sin\Big(2\pi \frac{\theta}{K}\Big) \quad \text{and} \quad \frac{d^j}{d\theta^j} \phi(\theta; x, y) \ll \frac{x}{K^j}, \,\,j \geq 2.$$  In the following we frequently use the Taylor expansions $\sin(t) = t + O(t^3)$ and $\cos(t) = 1 + t^2/2 + O(t^4)$. 

We first extract smoothly the range $|\theta| \leq \frac{1}{100} K^2/x$. Here we observe that    the derivative $\frac{d}{d\theta} \phi(\theta; x, y)$  cannot be too small (it is important that $w$ is supported on $[1, 2]$, not on $[0, 1]$), and we  
  apply \cite[Lemma 8.1]{BKY} with
$$\beta - \alpha \ll \frac{K^2}{x}, \quad X = \Big(\frac{K}{x}\Big)^{2j}, \quad U = \min(K^{1/10}, K^2/x), \quad R = 1 , \quad Y = x, \quad Q = K.$$
In this way we obtain a contribution of  
$$\ll_n  (\beta - \alpha)X[(QR/\sqrt{Y})^{-n} + (RU)^{-n}]  \ll \frac{1}{x^j} \Big(\frac{K^2}{x}\Big)^{1 + j} \Big( \Big(\frac{K}{\sqrt{x}}\Big)^{-n} + \Big(\frac{K^2}{x}\Big)^{-n} + K^{- n/10}\Big)$$
to \eqref{dW} for every $n \geq 0$. 
This is easily seen to be
     $$\ll_{j, A} \frac{1}{x^j} \Big(1 + \frac{K^2}{x}\Big)^{-A}$$
       for every $A > 0$. For the portion $|\theta| \gg K^2/x$ we integrate by parts in the $y$ integral and apply trivial estimates to obtain a bound
       $$\ll_{j, A} \int_{|\theta| \gg K^2/x} (1 + \theta)^{-A} \Big(\frac{\theta}{K}\Big)^{2j} d\theta $$
which is easily seen to be
$$\ll_{j, A} \min\Big(\Big(\frac{K^2}{x}\Big)^{-A} \frac{1}{x^j}, \frac{1}{K^{2j}}\Big) \ll \frac{1}{x^j} \Big(1 + \frac{K^2}{x}\Big)^{-A}.$$
The same analysis works for $W_1^-(x) = W_2^-(x) + \tilde{W}_2(x)$ where
\begin{equation}\label{W2m}
\begin{split}
W_2^-(x)&  = \frac{e(-3/8)}{2}e^{+ i x} \int_{-\infty}^{\infty} v(\theta K^{-1/10}) e\Big(\frac{-3\theta}{2K}\Big) e^{-ix \cos(2\pi \theta/K)}    \int_{-\infty}^{\infty}  w(y) e(y\theta)   dy \, d\theta.
\end{split}
\end{equation}
 We put $w_{\pm} = W_2^{\pm}$ and $w_0 = W_1 + W_2 + \tilde{W}_2$, and the lemma follows on noting that $\frac{1}{2}(i^k + (-i)^k) = i^k \delta_{2\mid k}$. 
\end{proof}


\emph{Remarks:} 1) It is clear from the proof that if $w$ depends  on other parameters in a real- or complex-analytic way with control on derivatives, then $w_{\pm} = W_2^{\pm}$, defined in \eqref{W2pm}, depends on these parameters in the same way. We will use this  observation in Section \ref{thm1} and \eqref{off-off}. 

2) The bound \eqref{boundsw} remains true for $A \geq -1/2$. In the case,  the claim follows for  $x \geq K^2$ from the asymptotic formula \cite[8.451.1  \&  7 \&  8]{GR}.  We state this for completeness, but we do not need it here.  \\

We need a similar formula for the transforms occurring in \eqref{hast} and \eqref{h0}. 
\begin{lemma}\label{bessel-kuz}
Let $A, T \geq 2$ and let ${\tt h}$ be a smooth function with support in $[T, 2T]$ satisfying ${\tt h}^{(j)}(t) \ll T^{-j}$ for $j \in \Bbb{N}_0$.

{\rm a)}  We have
\begin{displaymath}
\begin{split}
&  {\tt h}^*(x) = \frac{T^2}{\sqrt{x}} \Big(1 + \frac{T^2}{x}\Big)^{-A} \sum_{\pm} e(\pm x) H^{\pm}_A(x) + K^{\pm}_A(x)\\
 \end{split}
 \end{displaymath}
where $K^{\pm}_A(x)   \ll_A (T+x)^{-A}$ and $x^j \partial_x^j H^{\pm}_A(x) \ll_{A, j} 1$. An analogous asymptotic formula holds for ${\tt h}^{\dagger}(x)$. 

{\rm b)} We have 
$$h^{**}(x) = T \Big(\frac{x}{T} + \frac{T}{x}\Big)^{-A} H_A(x) + \tilde{K}_A(x)$$
where $\tilde{K}_A(x) \ll_A (T+ x)^{-A}$ and $x^j \partial_x^j H_A(x) \ll_{A, j} 1$.
\end{lemma}

\begin{proof}  a) For the first part we recall the uniform asymptotic formula \cite[7.13.2(17)]{EMOT}.
$$\frac{\pi i}{\cosh(\pi t)}    J_{2 it}(x)  = \sum_{\pm} \frac{ e^{\pm ix \mp i \omega(x, t)}}{x^{1/2} } f^{\pm}_M(x, t) + O_M((|t| + x)^{-M} )$$
where 
$\omega(x, t) =   |t| \cdot \text{arcsinh} (|t|/x) - \sqrt{t^2 + x^2} + x$
and
\begin{equation}\label{flat}
x^i |t|^j\frac{\partial^i}{\partial x^i}  \frac{\partial^j}{\partial t^j}  f^{\pm}_{M}(x, t) \ll_{i, j, M}  1\end{equation}
for every $M \geq 0$.  The error term in \cite{EMOT} is $O(x^{M })$, but for small $x$ the error term $O(|t|^{-M} )$ follows from the power series expansion \cite[8.440]{GR}. 
Partial integration in the form of \cite[Lemma 8.1]{BKY} with $U = T$, $Y = Q = T+x$, $R= \text{arcsinh}(T/x) \gg T/x$ shows that 
$$x^j \frac{\partial^j}{\partial x^j}  \int_{\Bbb{R}} e^{  \mp i \omega(x, t)}  f^{\pm}_{ M}(x, t) h(t) t  \frac{dt}{4\pi^2} \ll_{j, A, M} T^2 \Big( 1 + \frac{T^2}{x}\Big)^{-A}$$
and the claim follows. 

b) For the proof of the second part we distinguish 3 ranges. For $x > 10 T$ the claim follows easily from  the rapid decay of the Bessel $K$-function and its derivatives. For $x < T/10$ we use the uniform asymptotic expansion \cite[7.13.2(19)]{EMOT} (along with the power series expansion \cite[8.485, 8.445]{GR} for very small $x$)
$$ \cosh(\pi t) {K}_{2it}(x) =  \sum_{\pm}  e^{\pm i   \tilde{\omega}(x, t)}  \tilde{f}^{\pm}_M(x, t)  
+ O(|t|^{-M}), \quad \tilde{\omega}(x, t) =  |t| \cdot \text{arccosh} \frac{|t|}{x} - \sqrt{t^2 - x^2}$$
 where $\tilde{f}^{\pm}_M$ satisfies the analogous bounds in \eqref{flat}. Again integration by parts (\cite[Lemma 8.1]{BKY} with $U =  Y = Q = T$, $R= \text{arccosh}(T/x) \geq 1$) confirms the claim in this range. Finally, for $x \asymp T$ we use the integral representation \cite[8.432.4]{GR}
 $$\cosh(\pi t) K_{2it}(x) = \frac{\pi}{2} \int_{-\infty}^{\infty} \cos(x \sinh \pi u) e(tu) du.$$
 This integral is not absolutely convergent, but partial integration shows that the tail is very small, and we can in fact truncate the integral at $|u| \leq \varepsilon \log T$ at the cost of an admissible error $O(T^{-A})$. Thus we are left with bounding 
\begin{displaymath}
\begin{split}
&  \frac{d^j}{dx^j} \int_{-\infty}^{\infty} \int_{-\varepsilon \log T}^{\varepsilon \log T}  \cos(x \sinh \pi u) e(tu) h(t) t \, du \, dt \\
& \ll  \int_{T \leq |t| \leq 2T} \int_{-\varepsilon \log T}^{\varepsilon \log T}
|\sinh(\pi u)|^j  (1+| u|T)^{-B} t \, du\, dt \ll T.\\
\end{split}
\end{displaymath}
 if $B$ is chosen sufficiently large with respect to $j$. 
\end{proof}

For large arguments, the Bessel function $J_{ir}(y)$ behaves like an exponential.  More precisely, by \cite[8.451.1 \& 7 \& 8]{GR} we have an asymptotic expansion which we will need later:
\begin{equation}\label{bessel-approx}
\begin{split}
\sum_{\pm} (\mp) &\frac{J_{ir}(2\pi x) \cos(\pi/4 \pm \pi i r/2)}{\sin(\pi i r)} = \sum_{\pm}  \frac{e(\pm x)}{2\pi\sqrt{x}}\sum_{k=0}^{n-1} \frac{i^k(\pm 1)^k}{(4\pi x)^{k}} \frac{\Gamma(ir + k + 1/2)}{k! \Gamma(ir - k + 1/2)}  + O\Big(\Big(\frac{|r|^2}{x}\Big)^{-n}\Big)
\end{split}
\end{equation}
for $r \in \Bbb{R}$,   $x \geq 1$ and fixed $n \in \Bbb{N}$. This is useful as soon as $x \geq r^2$.\\

 The following lemma is essentially an application of Stirling's formula.  

\begin{lemma} Let $k \geq 1$, $s = \sigma + it \in \Bbb{C}$ with $k + \sigma \geq 1/2$, $M \in \Bbb{N}$.  Then
$$\frac{\Gamma(k+s)}{\Gamma(k) }= k^s G_{M, \sigma}(k, t) + O_{\sigma, M}((k + |t|)^{-M})$$
where
\begin{equation}\label{flat1}
k^{\frac{i}{2} + j}    \frac{d^{i}}{dt^{i}} \frac{d^{j} }{dk^{j}} G_{M, \sigma}(k, t) \ll_{M, \sigma, i, j} \Big( 1 + \frac{t^2}{k}\Big)^{-M}
\end{equation}
for $i, j \in \Bbb{N}_0$. Moreover, 
\begin{equation}\label{asymp}
\frac{\Gamma(k+\sigma + it)}{\Gamma(k) }= k^s \exp\Big(-\frac{t^2}{2k}\Big)\Big(1 + O_{\sigma}\Big(\frac{|t|}{k}  + \frac{t^4}{k^3}\Big)\Big).
\end{equation}
\end{lemma}

\begin{proof}  This is a standard application of Stirling's formula. First of all, since
$$\frac{\Gamma(k+s)}{\Gamma(k) } k^{-s}= \frac{\Gamma(k+\sigma)}{\Gamma(k) }k^{-\sigma} \frac{\Gamma(k+\sigma + it)}{\Gamma(k + \sigma) } (k+\sigma)^{-it} (1 + \sigma/k)^{it}$$
with 
\begin{equation}\label{sigma}
   \frac{d^{j_1}}{dk^{j_2}} \frac{d^{j_1} }{dt^{j_2}}(1 + \sigma/k)^{it} \ll_{\sigma, j_1, j_2} \frac{1}{k^{j_1+j_2}} \Big(1 + \frac{|t|}{k}\Big)^{j_1}, \quad (1 + \sigma/k)^{it} = 1 + O_{\sigma}(|t|/k), 
   \end{equation}
 it suffices for both statements to treat the two cases $s= \sigma \in \Bbb{R}$ fixed and $s = it \in i\Bbb{R}$. The first case is very simple, so we display the details for the second case.  We have
$$\frac{\Gamma(k+it)}{\Gamma(k) } k^{-it} =  \exp\big(\alpha (k, t) + i\beta (k, t)\big)
\Big(\tilde{G}_{M }(k, t) + O_{M }((k+|t|)^{-M})\Big) $$
where $\tilde{G}_{M }$ satisfies  
\begin{equation}\label{G}
 (k+|t|)^{j_1+j_2}  \frac{d^{j_1}}{dk^{j_1}} \frac{d^{j_2} }{dt^{j_2}} \tilde{G}_{M }(k, t) \ll_{M, j_1,  j_2} 1, \quad \tilde{G}_{M }(k, t) = 1 + O((k+|t|)^{-1})
 \end{equation}
 and
\begin{displaymath}
\begin{split}
\alpha (k, t) & = - t \arctan\frac{t}{k } + \frac{k  - 1/2}{2} \log\Big(1 + \frac{  t^2}{k^2}\Big),\\
 \beta (k, t) &=  t \Big(\log\sqrt{1 + \frac{  t^2}{k^2}} - 1\Big) + \Big(k  - \frac{1}{2}\Big) \arctan\frac{t}{k }.
\end{split}
\end{displaymath}
 It is not hard to see that
$$ \alpha(k, t)  \leq - c \min\Big(\frac{t^2}{k}, |t|\Big)$$
for some absolute constant $c$ (in fact, $c = (\pi - \log 4)/4 = 0.438\ldots$ is the optimal constant). In particular, $\Gamma(k+it)/\Gamma(k)$ is exponentially decreasing as soon as $|t| \geq k^{1/2}$. Moreover, by a Taylor argument we have 
$$\alpha(k, t) = -\frac{t^2}{2k} + O\Big(\frac{t^2}{k^2} + \frac{t^4}{k^3}\Big), \quad \beta(k, t) \ll \frac{|t|}{k} + \frac{|t|^3}{k^2}.$$
This proves \eqref{asymp}. To prove \eqref{flat1}, we need to bound the derivatives of $\alpha$ and $\beta$ which is most quickly done by using Cauchy's integral formula. Note that both $\alpha$ and $\beta$ have a branch cut at the two rays $\pm t/k \in [  i,   i \infty)$. We assume that $k$ is sufficiently large (otherwise there is noting to prove) and we choose a circle $C_1$ about $k$ of radius $k/100$ and  a circle $C_2$ about $t$  of radius $\sqrt{k}/10$. Then $w/z$ is away from the branch cuts for $z \in C_1$, $w \in C_2$, and we have $$\alpha(z, w) \ll  \frac{|w|^2}{|z|} + |w| \ll \frac{|t|^2 +k}{k}, \quad \beta(z, w) \ll    \frac{|w|}{|z|} + \frac{|w|^3}{|z|^2}  \ll  \frac{|t| }{k} + \frac{|t|^3}{k^2} + \frac{1}{k^{1/2}}.$$
for $z \in C_1$, $w\in C_2$. From Cauchy's integral formula we conclude 
\begin{equation}\label{alpha}
\frac{d^{i}}{dt^{i}} \frac{d^{j} }{dk^{j}}\alpha (k, t)\ll_{i, j} \Big(1 + \frac{t^2}{k}\Big)k^{-\frac{i}{2} - j} , \quad 
\frac{d^{i}}{dt^{i}} \frac{d^{j} }{dk^{j}}\beta (k, t)\ll_{i, j} \Big(1 + \frac{t^2}{k}\Big)^2k^{-\frac{i}{2} - j} 
\end{equation}
for $i, j \in \Bbb{N}_0$. Combining \eqref{sigma}, \eqref{G}, \eqref{alpha}  completes the proof of \eqref{flat1}. \end{proof}

We apply this to the function  $\mathcal{G}(k, t_{\tt u}, s)$ defined in  \eqref{defG}.

\begin{cor}\label{cor19} Let $A \geq 0$, $ \sigma  \geq -1/4$ and let $t \in \Bbb{R}$, $t_{\tt u} \in \Bbb{R} \cup [-i/2, i/2]$, $k \in 2\Bbb{N}$.  Then  
\begin{equation}\label{G1}
\mathcal{G}(k, t_{\tt u}, \sigma + 1/2 + it) \ll_{A,   \sigma} k^{-1/4 + 2\sigma} \left(1 + \frac{|t|^2 + |t_{\tt u}|^2}{k} \right)^{-A} . 
\end{equation}
Moreover, for $v \in \Bbb{C}$ 
 we have
$$ \mathcal{G}(k, t_{\tt u}, v + 1/2 + it)\mathcal{G}(k, t_{\tt u}, v + 1/2 - it) = {\tt G}_M(k, t_{\tt u}, t, v) + O_{\Re v, \Re w, M}(k^{-M})$$
with 
\begin{equation}\label{G2}
\begin{split}
k^{j_1 + \frac{j_2}{2} + \frac{j_3}{2}}&  \frac{d^{j_1}}{dk^{j_1}} \frac{d^{j_2} }{dt^{j_2}}\frac{d^{j_3} }{d\tau^{j_3}} 
{\tt G}_M(k,\tau, t, v, w)   \ll_{ \textbf{j}, \Re v, \Re w, M} k^{-1/2 +2 \Re v + 2\Re w} (1 + |\Im v|)^{j_1}
 \end{split}
\end{equation}
 for $\textbf{j} \in \Bbb{N}_0^3$. Finally, for $t, \tau \ll k^{2/3}$ we have 
 \begin{equation}\label{taylor}
  \mathcal{G}(k, \tau,   1/2 + it)     \mathcal{G}(k, \tau,   1/2 - it)  = \frac{16}{\pi k^{1/2}} \exp\Big(-\frac{2(t + \tau/2)^2+ 2(t - \tau/2)^2}{k}\Big) \Big(1 + O(k^{-1/3})\Big).
  \end{equation}
\end{cor}
Recalling the definition \eqref{defV3} of $V_t(x; k, t_{\tt u})$ we conclude from \eqref{G1} and appropriate contour shifts  the uniform bounds 
\begin{equation}\label{size-restr}
\begin{split}
k^{j_1 + \frac{j_2}{2} + \frac{j_3}{2}}x^{j_4}  \frac{d^{j_1}}{dk^{j_1}} \frac{d^{j_2} }{dt^{j_2}}\frac{d^{j_3} }{d\tau^{j_3}}\frac{d^{j_4}}{dx^{j_4}}    V_t(x; k, \tau) \ll_{A, \textbf{j}}  k^{-1/2} \Big(1 + \frac{x}{k^4}\Big)^{-A}  \Big(1 + \frac{|t|^2 + |\tau|^2}{k} \Big)^{-A} 
\end{split}
   \end{equation}
  for $A > 0$, $\textbf{j} \in \Bbb{N}_0^4$. 

In a similar, but simpler fashion we also apply    this to the weight function $W_t$ defined in \eqref{v-t} and state the bound
\begin{equation}\label{bound-wt}
 (1+|t|)^{j_1} x^{j_2} \frac{d^{j_1}}{dt^{j_1}}\frac{d^{j_2}}{dx^{j_2}}  W_t(x) \ll_{A, j_1, j_2} \Big(1 + \frac{x}{1 + |t|}\Big)^{-A}
\end{equation}
for $A \geq 0$, $j_1, j_2 \in \Bbb{N}_0$. 

\section{A weak version of   Theorem \ref{thm1}}\label{weakversion}

In this section we present a relatively soft  argument that  provides the upper bound $\mathcal{N}_{\text{av}}(K) \ll K^{\varepsilon}$.  This will useful later in order to estimate certain error terms later. By \eqref{Nav} and  \eqref{Fh}  we have  
\begin{displaymath}
\begin{split}
\mathcal{N}_{\text{av}}(K) \ll &\frac{1}{K^2}\sum_{k \in 2\Bbb{N}}  W\Big(\frac{k}{K}\Big) \sum_{h \in B_{k-1/2}^+(4)} \int_{-\infty}^{\infty}  \int_{\Lambda_{\text{\rm ev}}} \frac{\Gamma(k-3/2)}{(4\pi)^k \| h \|^2}    \sum_{f_1, f_2} \frac{1}{f_1f_2}\\
& \sum_{D_1, D_2 < 0} \frac{c_h(|D_1|)c_h(|D_2|)P(D_1; {\tt u}) \overline{P(D_2; {\tt u})} }{   |D_1D_2|^{ k/2  }} \Big(\frac{|D_2|f_2^2}{|D_1|f_1^2}\Big)^{it}V_{  t}(|D_1D_2|(f_1f_2)^2 ; k,  t_{\tt u})    d{\tt u} \, dt.
\end{split}
\end{displaymath}
By \eqref{size-restr} we have  $t, t_{\tt u} \ll K^{1/2 + \varepsilon}$ (up to a negligible error). We insert a smooth partition of unity into the $t_u$-integral and attach a factor $w(|t_{\tt u}|/\mathcal{T}_{\text{spec}})$ 
 where $w$ has support in $[1, 2]$ unless $\mathcal{T}_{\text{spec}} = 1$, 
 in which case $w$ has support in $[0, 2]$. Let    \begin{displaymath}
\begin{split}
\mathcal{N}_{\text{av}}(K, ; &\mathcal{T}_{\text{spec}}):=\frac{1}{K^2}\sum_{k \in 2\Bbb{N}}  W\Big(\frac{k}{K}\Big) \sum_{h \in B_{k-1/2}^+(4)} \int_{-\infty}^{\infty}  \int_{\Lambda_{\text{\rm ev}}} \frac{\Gamma(k-3/2)}{(4\pi)^k \| h \|^2}  w\Big( \frac{|t_{\tt u}|}{\mathcal{T}_{\text{spec}}}\Big)
  \sum_{f_1, f_2} \frac{1}{f_1f_2}\\
& \sum_{D_1, D_2 < 0} \frac{c_h(|D_1|)c_h(|D_2|)P(D_1; {\tt u}) \overline{P(D_2; {\tt u})} }{   |D_1D_2|^{ k/2  }} \Big(\frac{|D_2|f_2^2}{|D_1|f_1^2}\Big)^{it}V_{  t}(|D_1D_2|(f_1f_2)^2 ; k,  t_{\tt u})    d{\tt u} \, dt.
\end{split}
\end{displaymath}
for $\nu  = 0, 1, 2, \ldots$  and $ \mathcal{T}_{\text{spec}} = 2^{\nu}\ll K^{1/2 + \varepsilon}$.
 This section is then devoted to the proof of the bound
\begin{equation}\label{weak}
 \mathcal{N}_{\text{av}}(K;  \mathcal{T}_{\text{spec}}) \ll_{\varepsilon} 1 +  \frac{\mathcal{T}_{\text{spec}}^2}{K^{1-\varepsilon}}
 \end{equation}
for $\varepsilon>0$.  We will now sum over $h$ using Lemma \ref{lem1}. This yields a diagonal term and an off-diagonal term that we treat separately in the following two subsection. Throughout, the letter  $D$, with or without subscripts, shall always denote a \emph{negative} discriminant unless stated otherwise. Let letter $A$ shall denote an arbitrarily large fixed constant, not necessarily the same on every occurrence. 

  
 \subsection{The diagonal term}\label{10.1}
 
The contribution to $\mathcal{N}_{\text{av}}(K;  \mathcal{T}_{\text{spec}})$ of the diagonal term from Lemma \ref{lem1}  is  bounded by    
 \begin{displaymath}
 \begin{split}
   & \ll\frac{1}{K }\int_{-\infty}^{\infty} \int_{ \Lambda_{\text{\rm ev}} }w\Big( \frac{|t_{\tt u}|}{\mathcal{T}_{\text{spec}}}\Big)
   \sum_{f_1,f_2}  \frac{1}{f_1f_2} \sum_{D} \frac{ |P(D ; {\tt u})|^2 }{   |D|^{3/2}} \sup_{k \ll K} |V_t((|D|f_1f_2)^2; k, t_{\tt u}) |d{\tt u} \, dt . 
 \end{split}
 \end{displaymath}

 For the constant function ${\tt u} = \sqrt{3/\pi}$ we have $P(D; \sqrt{3/\pi}) \ll H(D)$. Recalling \eqref{size-restr} and Lemma \ref{hur2},  this gives a total contribution of 
 $$\ll \frac{1}{K }  \int_{-\infty}^{\infty}  \sum_{f_1 ,f_2} \frac{1}{f_1f_2} \sum_D \frac{H(D)^2}{|D|^{3/2}} \Big( 1 + \frac{|D|f_1 f_2}{K^2}\Big)^{-10} \Big(1 + \frac{|t|^2}{K}\Big)^{-10}dt  \ll   \sum_{f_1 ,f_2} \frac{1}{(f_1f_2)^{3/2}} \ll 1$$
  if $\mathcal{T}_{\text{spec}} = 1$ and otherwise the contribution vanishes. 
 
Similarly, by \eqref{eisen2}, \eqref{eisen-L}, \eqref{eisen-L2} and \eqref{lower}, the Eisenstein spectrum contributes
 \begin{displaymath}
 \begin{split}
  & \ll  \max_{1 \leq R \leq K^{2+\varepsilon}}  \frac{1}{RK^{1-\varepsilon}}   \int_{|\tau| \ll \mathcal{T}_{\text{spec}}} |\zeta(1/2 + i\tau)|^2    \sum_{R \leq |\Delta| \leq 2R} |L(\chi_{\Delta}, 1/2 + i\tau)|^2 d\tau. 
  \end{split}
  \end{displaymath}
  We recall our convention that $\Delta$ denotes a  negative fundamental discriminant, $D = \Delta f^2$ an   arbitrary negative  discriminant (where $f$ here has nothing to do with $f_1, f_2$ above). In the above bound we have already executed the sum over $f$. By Proposition \ref{Lfunc}(b) and 
 a standard bound for the fourth moment of the Riemann zeta-function this is $O(K^{-1/2 + \varepsilon})$. 
 

 By \eqref{katok-Sarnak},  \eqref{key}, \eqref{key-simple} and \eqref{lower},  the contribution of the cuspidal spectrum is at most 
   \begin{displaymath}
 \begin{split}
&  \ll  \max_{1 \leq R \leq K^{2+\varepsilon}}   \frac{1}{R K^{1-\varepsilon}} \sum_{t_{ u} \ll \mathcal{T}_{\text{spec}}} L({ u}, 1/2)    \sum_{   R \leq  |\Delta| \leq 2R}    L({ u} \times \chi_{\Delta}, 1/2 )  \ll \mathcal{T}^2_{\text{spec}}K^{\varepsilon-1}   .   \end{split}
  \end{displaymath}
by Corollary \ref{Lfunc-cor}. 
   All of these bounds are consistent with \eqref{weak}.

\subsection{The off-diagonal term: generalities}\label{112}      We now consider the off-diagonal term in Lemma \ref{lem1} from the sum over $h$. 
Here we need to bound
\begin{displaymath}
\begin{split}
   \frac{1}{K^2}\sum_{k \in 2\Bbb{N}} &W\left(\frac{k}{K}\right)  i^k \int_{-\infty}^{\infty}   \int_{\Lambda_{\text{\rm ev}}}  w\Big( \frac{|t_{\tt u}|}{\mathcal{T}_{\text{spec}}}\Big)
   \sum_{f_1, f_2} \sum_{D_1, D_2  } \frac{P(D_1; {\tt u})\overline{P(D_2; {\tt u})} }{  f_1f_2 |D_1 D_2|^{ 3/4 }}   \Big(\frac{|D_2|f_2^2}{|D_1|f_1^2}\Big)^{it}\\
    & V_t(|D_1D_2|(f_1f_2)^2,  k,  t_{\tt u}) \sum_{4 \mid c}    \frac{K^+_{3/2}(|D_1|, |D_2|, c)}{c} J_{k-3/2}\Big(\frac{4\pi \sqrt{|D_1D_2|}}{c}\Big)    d{\tt u} \, dt . 
\end{split}
\end{displaymath}
 We first sum over $k$ using Lemma \ref{lem2}. Up to a negligible error, we obtain
 \begin{displaymath}
\begin{split}
    \frac{1}{K^2}   \int_{-\infty}^{\infty} &  \int_{\Lambda_{\text{\rm ev}}}  w\Big( \frac{|t_{\tt u}|}{\mathcal{T}_{\text{spec}}}\Big)
   \sum_{f_1, f_2} \sum_{D_1, D_2  } \frac{P(D_1; {\tt u})\overline{P(D_2; {\tt u})} }{  f_1f_2 |D_1 D_2|^{ 3/4 }}   \Big(\frac{|D_2|f_2^2}{|D_1|f_1^2}\Big)^{it}\\
    &  \sum_{4 \mid c}    \frac{K^+_{3/2}(|D_1|, |D_2|, c)}{c} e\Big(\pm \frac{2\sqrt{|D_1D_2|}}{c}\Big) \tilde{V}\Big(|D_1D_2|(f_1f_2)^2,   \frac{\sqrt{|D_1D_2|}}{c}, t, t_{\tt u}\Big)   d{\tt u} \, dt 
\end{split}
\end{displaymath}
where
 \begin{equation}\label{tildeV}
\begin{split}
y^{j_1} K^{\frac{1}{2}(j_2+j_3)} &   x^{j_4}  \frac{d^{j_1}}{dy^{j_1}} \frac{d^{j_2} }{dt^{j_2}}\frac{d^{j_3} }{d\tau^{j_3}}\frac{d^{j_4}}{dx^{j_4}}   \tilde{V}(x, y, t, \tau) \\
&\ll_{A,  \textbf{j}}  K^{-1/2} \Big(1 + \frac{x}{K^4}\Big)^{-A}  \Big(1 + \frac{|t|^2 + |\tau|^2}{K} \Big)^{-A} \Big(1 + \frac{K^2}{y}\Big)^{-A}
\end{split}
\end{equation}
for any $A \geq 0$,   $\textbf{j} \in \Bbb{N}_0^4$, cf.\ \eqref{size-restr} and  the   remark after Lemma \ref{lem2}.  In order to apply Voronoi summation, we open the Kloosterman sum and are left with bounding 
  \begin{equation}\label{leftwith}
\begin{split}
    \frac{1}{K^2}\sum_{4 \mid c}& \underset{\substack{d\,(\text{mod }c)\\ (d, c) = 1}}{\max}\Big|  \int_{\Lambda_{\text{\rm ev}}}  w\Big( \frac{|t_{\tt u}|}{\mathcal{T}_{\text{spec}}}\Big)
   \sum_{f_1, f_2} \sum_{D_1, D_2  } \frac{P(D_1; {\tt u})\overline{P(D_2; {\tt u})} }{  f_1f_2 |D_1 D_2|^{ 3/4 }}   \\
    &  e\Big(\pm \frac{2\sqrt{|D_1D_2|}}{c}\Big)  e\Big( \frac{|D_1|d + |D_2|\bar{d}}{c}\Big)V^{\ast}\Big(|D_1|f_1^2, |D_2|f_2^2,   \frac{\sqrt{|D_1D_2|}}{c}, t, t_{\tt u}\Big)   d{\tt u}   \Big|
\end{split}
\end{equation}
where 
\begin{equation}\label{Vastdef}
 V^{\ast}(x_1, x_2, y, \tau) = \int_{-\infty}^{\infty} 
  \Big(\frac{x_2}{x_1}\Big)^{it}  \tilde{V}(x_1x_2,  y, t, \tau)  dt. 
 \end{equation}
Integration by parts shows that
 \begin{equation}\label{astV}
\begin{split}
&y^{j_1} K^{\frac{1}{2}j_2} x_1^{j_3}x_2^{j_4}  \frac{d^{j_1}}{dy^{j_1}} \frac{d^{j_2} }{d\tau^{j_2}}\frac{d^{j_3} }{dx_1^{j_3}}\frac{d^{j_4}}{dx_2^{j_4}}   V^*(x_1, x_2, y, \tau) \\
&\ll_{A, \textbf{j}} \Big(1 + \frac{x_1x_2}{K^4}\Big)^{-A}  \Big(1 + \frac{  |\tau|^2}{K} \Big)^{-A} \Big(1 + \frac{K^2}{y}\Big)^{-A}\Big(1 +  K^{1/2}|\log x_2/x_1|\Big)^{-A}
\end{split}
\end{equation}
for any $A \geq 0$,  $\textbf{j} \in \Bbb{N}_0^4$. The first and third factor on the right hand side of the last display imply $|D_1 D_2| = K^{4+o(1)}$ and $c, f_1, f_2 = K^{o(1)}$ up to a negligible error. On the other hand,  the last factor implies  $D_1f_1^2 =D_2f_2^2(1 + O(1/K^{1/2}))$, so that in effect $D_1, D_2 = K^{2+o(1)}$. 

In the following we treat the cuspidal part, the Eisenstein part and the constant function separately. In principle we could treat them on equal footing and we will do this in Section \ref{off-off}, but for now we keep the prerequisites as simple as possible. 


\subsection{The Eisenstein contribution} We start with the contribution of the Eisenstein spectrum. As this is much smaller in size than the cuspidal spectrum, very simple bounds suffice.  By \eqref{eisen1}, \eqref{lower} and \eqref{astV}   we obtain a contribution of
$$\frac{K^{\varepsilon} }{K^{4} }   \int_{-2\mathcal{T}_{\text{spec}}}^{2\mathcal{T}_{\text{spec}}} \sum_{\substack{D_1f_1^2 = D_2 f_2^2 (1 + O(K^{\varepsilon-1/2}  ))\\   |D_1|, |D_2|  = K^{2 + o(1)}\\ f_1, f_2 \ll K^{\varepsilon}}}   |\zeta(1/2 + it)|^2 |L(D_1, 1/2 + it)L(D_2, 1/2 + it)| dt.  $$
We use the basic inequality $|L(D_1, 1/2 + it)L(D_2, 1/2 + it)| \leq\frac{1}{2}( |L(D_1, 1/2 + it)|^2 + |L(D_2, 1/2 + it)|^2) $. For fixed $f_1, f_2, D_1$ there are $O(  K^{3/2+\varepsilon} )$ values of $D_2$ satisfying the summation condition. Thus we obtain the bound
$$\frac{K^{\varepsilon} }{K^{5/2 } }   \int_{-2\mathcal{T}_{\text{spec}}}^{2\mathcal{T}_{\text{spec}}} \sum_{   |D| = K^{2 +o(1)}}  |\zeta(1/2 + it)|^2 |L(D, 1/2 + it)|^2  dt.  $$
By \eqref{basicL} and Proposition \ref{Lfunc}(b) along with a bound for the fourth moment of the Riemann zeta function we obtain the desired bound
$$\frac{K^{\varepsilon} }{K^{5/2 } }   \int_{-2\mathcal{T}_{\text{spec}}}^{2\mathcal{T}_{\text{spec}}} \sum_{   |\Delta|  \leq K^{2+\varepsilon}}  |\zeta(1/2 + it)|^2 |L(\Delta, 1/2 + it)|^2  dt \ll \frac{\mathcal{T}_{\text{spec}}}{K^{1/2 - \varepsilon}}.$$

\subsection{The cuspidal contribution}\label{104} Next we consider the cuspidal contribution and consider the following portion
\begin{equation}\label{portion}
\frac{\overline{P(D_2; u)}}{|D_2|^{ 3/4   }}\sum_{D_1  } \frac{P(D_1; u)}{|D_1|^{3/4  }}  e\Big( \frac{|D_1|d}{c}\Big) e\Big( \pm \frac{2\sqrt{|D_1D_2|}}{c}\Big)  V^{\ast}\Big(|D_1|f_1^2, |D_2| f_2^2,   \frac{\sqrt{|D_1D_2|}}{c},  t_{\tt u}\Big) 
\end{equation}
 of \eqref{leftwith}, i.e.\ we freeze $c, f_1, f_2$ and $u$ for the moment.  For notational simplicity we consider only the plus case, the minus case may be treated similarly. We insert \eqref{mixed} with $t = t_u/2$  and re-write \eqref{portion} as
  \begin{displaymath} 
  \begin{split}
 &  \frac{3}{\pi}\overline{b(D_2)}\Big|\Gamma\Big(\frac{1}{4} + \frac{it_u}{2}\Big)\Big|^2 L(u, 1/2)  \sum_{D_1  }  b(D_1)e\Big(- \frac{D_1d}{c}\Big)    e\Big(   \frac{2\sqrt{|D_1D_2|}}{c}\Big)V^{\ast}\Big(|D_1|f_1^2, |D_2|f_2^2, \frac{\sqrt{|D_1D_2|}}{c}, t_u\Big)   .
  \end{split}
  \end{displaymath}
To the $D_1$-sum we apply the Voronoi formula (Lemma \ref{Vor}) with weight function
\begin{equation}\label{phi-func}
\phi(x) = \phi_{c, f_1, f_2}(x; t, t_u, D_2)= \frac{1}{ |x|^{1/2   }}   e\Big(   \frac{2\sqrt{|xD_2|}}{c}\Big)V^{\ast}\Big(|x|f_1^2, |D_2|f_2^2, \frac{\sqrt{|xD_2|}}{c}, t_u\Big) 
\end{equation}
for $x < 0$ and $\phi(x) = 0$ for $x > 0$. 
We recall that $c, f_1, f_2 \ll K^{\varepsilon}$ are essentially fixed from the decay conditions of $V^{\ast}$, but we need to be uniform in   $|D_2| = K^{2+o(1)}$.  We define $$\Phi(\pm y) = \Phi_{c, f_1, f_2}(\pm y; t, t_u, D_2)  = \int_0^{\infty} \mathcal{J}^{\pm, -}(x y) \phi(-x) dx$$ as in \eqref{defPhi} with $r = t_u/2$ and obtain that \eqref{portion} is equal to
\begin{equation}\label{D1sum}
\begin{split}
&\frac{3}{\pi} \overline{b(D_2)}\Big|\Gamma\Big(\frac{1}{4} + \frac{it_u}{2}\Big)\Big|^2 L(u, 1/2)  \frac{2\pi}{c} \left( \frac{-c}{-d}\right)\epsilon_{-d} e\left( \frac{1}{8}\right) \sum_{D} b(D) \sqrt{|D|}e\left(\frac{\bar{d}D}{c}\right) \Phi\left( \frac{(2\pi)^2 D}{c^2}\right),
\end{split}
\end{equation}
where only in the above sum do we allow $D$ to be either positive or negative. 
If $D > 0$, then the integral transform $\Phi((2\pi)^2 D/c^2 )$ contains a factor   
\begin{equation}\label{BesselK}
\frac{K_{   i   t_u} ( 4\pi \sqrt{|xD|}/c) }{\Gamma(3/4 + it_u/2)\Gamma(3/4 - it_u/2)}, 
\end{equation}
and we recall that    $|x| = K^{2+o(1)}$, $c \ll K^{\varepsilon}$ up to a negligible error. Thus the argument of the Bessel function is $\gg K^{2-\varepsilon}$, while the index is $\ll K^{1/2+\varepsilon}$. By the rapid decay of the Bessel $K$-function this contribution is easily seen to be negligible (we use \eqref{BarMao} and bound $b(D)$ trivially), 
and we may restrict from now on to $D< 0$.   In this case 
\begin{equation*}
\begin{split}
 \Phi\left( \frac{(2\pi)^2 D}{c^2}\right) = \int_0^{\infty} &  \sum_{\pm} (\mp) \frac{\cos(\pi/4 \pm \pi i t_u/2)}{\sin(\pi i t_u)} J_{\pm i t_u}\Big( \frac{4\pi \sqrt{|Dx|}}{c}\Big)    e\Big(  \frac{2\sqrt{|xD_2|}}{c}\Big)\\
 &V^{\ast}\Big(xf_1^2, |D_2|f_2^2, \frac{\sqrt{|xD_2|}}{c},   t_u\Big)  \frac{dx}{ x^{1/2   }}.
  \end{split}
 \end{equation*}
Using \eqref{bessel-approx}, up to a negligible error we can  write 
\begin{equation*}
\begin{split}
 \Phi\left( \frac{(2\pi)^2 D}{c^2}\right) =  \frac{c^{1/2}}{|D|^{1/4}} \sum_{\pm} \int_0^{\infty} &       e\Big(  \frac{2 \sqrt{x} (\sqrt{|D_2|} \pm \sqrt{|D|})}{c}\Big)\\
 &f^{\pm}\Big(\frac{2 \sqrt{|Dx|}}{c}, t_u\Big) V^{\ast}\Big(xf_1^2, |D_2|f_2^2, \frac{\sqrt{|xD_2|}}{c}, t_u\Big)  \frac{dx}{ x^{3/4  }}
  \end{split}
 \end{equation*}
 with
\begin{equation}\label{besseldecay}
x^j \frac{\partial^j}{\partial x^j}    f^{\pm}(x, r) \ll_{j} 1
\end{equation}
for any $ j \in \Bbb{N}_0$. 
 We substitute this back into  \eqref{D1sum} which equals \eqref{portion}.  We substitute this back into \eqref{leftwith}. In this way we see that the cuspidal contribution to \eqref{leftwith} is at most 
   \begin{equation}\label{isnow}
\begin{split}
    \frac{1}{K^2}&\sum_{4 \mid c} \frac{1}{c^{1/2}}  \sum_{u \text{ even}} w\Big( \frac{|t_{\tt u}|}{\mathcal{T}_{\text{spec}}}\Big) \sum_{f_1, f_2} \frac{1}{f_1f_2}   \\
    & \sum_{D, D_2}  \overline{|b(D_2)|}\Big|\Gamma\Big(\frac{1}{4} + \frac{it_u}{2}\Big)\Big|^2  L(u, 1/2) |b(D)| |D|^{1/4}      | \Psi_{c, f_1, f_2}(  D, D_2, t_u)| \end{split}
\end{equation}
with 
 \begin{displaymath}
 \begin{split}
 \Psi_{c, f_1, f_2}(  D, D_2, t_u) & = 
  \int_0^{\infty}        e\Big(  \frac{2\sqrt{x} (\sqrt{|D_2|} \pm \sqrt{|D|})}{c}\Big)  f^{\pm}\Big(\frac{2 \sqrt{|Dx|}}{c}, t_u\Big)V^{\ast}\Big(xf_1^2, |D_2|f_2^2, \frac{\sqrt{|xD_2|}}{c}, t_u\Big)    \frac{dx}{ x^{3/4 }}    \\
 &=   2   \int_0^{\infty}     e\Big(  \frac{2x (\sqrt{|D_2|} \pm \sqrt{|D|})}{c}\Big)  f^{\pm}\Big(\frac{2 \sqrt{|Dx^2|}}{c}, t_u\Big) V^{\ast}\Big(x^2f_1^2, |D_2|f_2^2, \frac{\sqrt{|x^2D_2|}}{c}, t_u\Big) \frac{dx}{ x^{1/2 }}.   
 \end{split}
 \end{displaymath}
By \eqref{astV} and \eqref{besseldecay}, each integration  by parts with respect to $x$ introduces an additional factor 
\begin{equation}\label{add1}
\frac{c K^{1/2}}{x(\sqrt{|D_2|} \pm \sqrt{|D|})}, 
\end{equation}
and we conclude that
\begin{equation}\label{add2}
\begin{split}
 \Psi_{c, f_1, f_2}( D, D_2,& t_u)  \ll_A   \Big(1 + \frac{  |t_{\tt u}|^2}{K} \Big)^{-A}   \int_0^{\infty} \Big( 1 + \frac{x(\sqrt{|D_2|} \pm \sqrt{|D|})}{c K^{1/2}}\Big)^{-A} \\
 &  \Big(1+K^{1/2}\log \frac{f_2^2 |D_2|}{f_1^2 x^2} \Big)^{-A} 
  \Big(1 + \frac{|D_2|(x f_1f_2)^2}{K^4}\Big)^{-A} \Big(1 + \frac{K^2c}{x |D_2|^{1/2}}\Big)^{-A}  \frac{dx}{ x^{1/2 }}   
\end{split}
\end{equation}
 for every $A \geq 0$. Here   only the  negative part in the $\pm$ sign is relevant, since otherwise the expression is trivially negligible. The limiting factor for the size of the $x$-integral is the first factor in the second line of the previous display, so that we obtain
 \begin{equation*}
\begin{split}
 \Psi_{c, f_1, f_2}( D, D_2, t_u)  \ll_A  & \frac{f_2 ^{1/2}|D_2|^{1/4}  }{f_1^{1/2}K^{1/2}}  \Big(1 + \frac{  |t_{\tt u}|^2}{K} \Big)^{-A} \Big( 1 + \frac{f_2|D_2|^{1/2}(\sqrt{|D_2|}- \sqrt{|D|})}{f_1cK^{1/2}}\Big)^{-A}\\
 & \Big(1 + \frac{|D_2|f_2^2}{K^2}\Big)^{-A}  \Big(1 + \frac{K^2f_1c}{ |D_2|f_2}\Big)^{-A}
\end{split}
\end{equation*}
for every $A \geq 0$. It is not hard to see that this can be simplified as
 \begin{equation*}
\begin{split}
 \Psi_{c, f_1, f_2}( D, D_2, t_u)  \ll_A  & (1 + f_1f_2c)^{-A}     \Big(1 + \frac{  |t_{\tt u}|^2}{K} \Big)^{-A} \Big(1 + \frac{|D_2|}{K^2}\Big)^{-A} \Big( 1 + \frac{ |D_2| - |D|}{K^{1/2}}\Big)^{-A}. \end{split}
\end{equation*}
 With this bound we return to \eqref{isnow}, apply the simple bound $|b(D)b(D_2)| \leq |b(D)|^2 + |b(D_2)|^2$ together with \eqref{BarMao}, getting an upper bound of the shape
 \begin{equation*}
\begin{split}
    \frac{1}{K^2}\sum_{4 \mid c} \frac{1}{c^{1/2}}&  \sum_{u \text{ even}} w\Big( \frac{|t_{\tt u}|}{\mathcal{T}_{\text{spec}}}\Big)    \sum_{f_1, f_2} \sum_{D, D_2}  \frac{  |D|^{1/4} }{  f_1f_2|D_2|}   \frac{L(u, 1/2)L(u, D_2, 1/2)}{L(\text{sym}^2 u, 1)}  | \Psi_{c, f_1, f_2}(  D, D_2, t_u)| \end{split}
\end{equation*}
plus a similar expression that with $L(u, D, 1/2)/|D|$ in place of $L(u, D_2, 1/2)/|D_2|$ which can be treated in the same way. We sum over $D, f_1, f_2, c$    and end up with (after changing the value of $A$)
 \begin{equation*}
\begin{split}
    \frac{1}{K^{3}} &  \sum_{u \text{ even}}w\Big( \frac{|t_{\tt u}|}{\mathcal{T}_{\text{spec}}}\Big)  \Big(1 + \frac{  |t_{\tt u}|^2}{K} \Big)^{-A}     \sum_ {  D_2}     \Big(1 + \frac{|D_2| }{K^2}\Big)^{-A}   \frac{L(u, 1/2)L(u, D_2, 1/2)}{L(\text{sym}^2 u, 1)}   .   
\end{split}
\end{equation*}
We write $D_2 = \Delta f^2$ with a fundamental discriminant $\Delta$ and use \eqref{key}, \eqref{key-simple}.  Summing over $f$, we obtain
\begin{equation*}
\begin{split}
    \frac{1}{K^{3}} &  \sum_{u \text{ even}}w\Big( \frac{|t_{\tt u}|}{\mathcal{T}_{\text{spec}}}\Big)  \Big(1 + \frac{  |t_{\tt u}|^2}{K} \Big)^{-A}     \sum_ { \Delta } \frac{K^{4/3}}{|\Delta|^{2/3} }   \Big(1 + \frac{|\Delta| }{K^2}\Big)^{-A}   \frac{L(u, 1/2)L(u, \Delta, 1/2)}{L(\text{sym}^2 u, 1)} .
\end{split}
\end{equation*}
We estimate the denominator $L(\text{sym}^2 u, 1)$ by \eqref{lower} and apply Corollary \ref{Lfunc-cor}  to finally obtain the upper bound $\mathcal{T}_{\text{spec}}^2 K^{\varepsilon-1}$ in agreement with \eqref{weak}.
       
 \subsection{The constant function}   This is very similar to the preceding subsection,  so we can be brief.  In short, we win a factor $\mathcal{T}_{\text{spec}} \ll K^{1+\varepsilon}$ from the fact that the spectrum is reduced to one element, and we lose a factor $|D|^{1/2} \ll K^{1+\varepsilon}$ since each class number is a factor $|D|^{1/4}$ bigger than the generic period $P(D;u)$. The key point is that we end up with a pure bound without a $K^{\varepsilon}$-power. This $K^{\varepsilon}$-power is unavoidable when we apply Corollary \ref{Lfunc-cor}, but for a sum over class numbers $H(D)$ alone we can apply Lemma  \ref{hur2} below which avoids a $K^{\varepsilon}$-power. 
  The analogue of \eqref{portion} is  
$$\frac{3}{\pi} \frac{\overline{H(D_2)}}{|D_2|^{ 3/4   }}\sum_{D_1  } \frac{H(D_1)}{|D_1|^{3/4  }}  e\Big( \frac{|D_1|d}{c}\Big) e\Big( \pm \frac{2\sqrt{|D_1D_2|}}{c}\Big)  V^{\ast}\Big(|D_1|f_1^2, |D_2|f_2^2, \frac{\sqrt{|D_1D_2|}}{c},  t_{\tt u}\Big). $$
 We apply the Voronoi formula (Lemma \ref{class-num1}) to the $D_1$-sum as before. Due to the oscillatory behaviour of $\phi$ in \eqref{phi-func} the main terms are easily seen to be negligible, and as in the previous argument also one of the osciallatory terms is negligible due to the exponential decay of $\mathcal{J}^-$ in \eqref{Jclean}. The behaviour of $\mathcal{J}^+$ similar, but  much simpler, as no asymptotic formula of a Bessel function is necessary. The analogue of \eqref{isnow} then becomes   \begin{equation*}
\begin{split}
    \frac{1}{K^2}\sum_{4 \mid c} \frac{1}{c^{1/2}}&      \sum_{f_1, f_2} \sum_{D, D_2}  \frac{H(D_2) H(D) |D|^{1/4} }{ |DD_2|^{3/4}  f_1f_2}     | \Psi_{c, f_1, f_2}( D, D_2, t_u)|. \end{split}
\end{equation*}
In the same way as above this leads to
\begin{equation}\label{leadsto}
\begin{split}
    \frac{1}{K^{3}} &   \sum_ { D}   \Big(1 + \frac{D }{K^2}\Big)^{-A}   \frac{H(D)^2}{|D|^{1/2}} \ll 1. 
\end{split}
\end{equation}
The last step is justified by the following  simple lemma.
\begin{lemma}\label{hur2} For $x \geq 1$ we have
 $$\sum_{D \leq x} H(D)^2 \ll x^2.$$
\end{lemma}

\begin{proof} Let $h(D)$ denote the usual class number. Since
 $$H(D) \leq  \sum_{n^2 \mid D} h(D/n^2)$$
 we have
 $$\sum_{|D| \leq x} H(D)^2 \leq \sum_{n_1, n_2} \sum_{[n_1^2, n_2^2] \mid |D| < x} h\left(\frac{D}{n_1^2}\right)h\left(\frac{D}{n_2^2}\right) \leq \sum_{n_1, n_2, m} \sum_{|D| \leq \frac{x}{n_1^2n_2^2m^2} } h(n_2^2D) h(n_1^2D).$$
Since $h(\Delta n^2) = h(\Delta)  n \prod_{p \mid n} (1 - \chi_{\Delta}(p)/p) \leq h(\delta) n\tau(n)$ for a fundamental discriminant $\Delta$, we get from bounds for moments of the ordinary class number \cite{Ba} that 
 \begin{displaymath}
 \begin{split}
 \sum_{|D| \leq x} H(D)^2 &\leq\sum_{n_1, n_2, m, f}  \sum_{|\Delta| \leq x/(n_1 n_2 m f )^2 } h(\Delta)^2 n_1n_2f^2 \tau(n_1)\tau(n_2)\tau(f)^2\\
 & \ll  x^2\sum_{n_1, n_2, m, f} \frac{n_1n_2f^2 \tau(n_1)\tau(n_2)\tau(f)^2}{(n_1n_2mf)^4} \ll x^2.
 \end{split}
 \end{displaymath}
 \end{proof}

The bound \eqref{leadsto}   is in agreement with, and completes the proof of,  \eqref{weak}.\\

We conclude this section with a brief discussion. The bound \eqref{weak} along with $\mathcal{T}_{\text{spec}} \ll K^{1/2 + \varepsilon}$ (from the decay of $V_t$) implies immediately the upper bound $N_{\text{av}}(K) \ll K^{\varepsilon}$. The $\varepsilon$-power is unavoidable at this point because of the  use of Heath-Brown's large sieve in the proof of Proposition \ref{Lfunc}. Except for the spectral large sieve implicit in  proof of Proposition \ref{Lfunc}, we have not touched the spectral ${\tt u}$-sum, so any further improvement must involve a treatment of this sum. This is precisely the purpose of the relative trace formula for Heegner periods given in Theorem \ref{thm2}. The weaker result \eqref{weak} is nevertheless useful: it allows us to discard small eigenvalues $t_{\tt u} \ll K^{1/2 - \varepsilon}$ and it allows us to estimate efficiently some error terms later.    The following sections are devoted to the proof of Theorem \ref{thm1}. 


\section{Proof of Theorem \ref{thm1}: the preliminary argument}\label{sec12}

By \eqref{Nav} and  \eqref{Fh} we have
\begin{displaymath}
\begin{split}
\mathcal{N}_{\text{av}}(K) =  & \frac{12}{\omega  K^2}    \sum_{k \in 2\Bbb{N}} W\left(\frac{k}{K}\right) \sum_{h \in B_{k-1/2}^+(4)} \frac{\pi^2}{90} \cdot \frac{18\sqrt{\pi}}{2  }\cdot 2\sum_{(n, m) = 1} \frac{\lambda(n) \mu(n)\mu^2(m)}{n^{3/2}m^{3}}  \int_{-\infty}^{\infty}  \int_{\Lambda_{\text{\rm ev}}} \frac{\Gamma(k - 3/2)}{(4\pi)^{k-3/2} \| h \|^2} \\
& \sum_{f_1, f_2, D_1, D_2}  \frac{c(|D_1|) c(|D_2|) P(D_1; {\tt u})\overline{P(D_2; {\tt u})}}{f_1f_2 |D_1D_2|^{k/2}}\Big(\frac{|D_2|f_2^2}{|D_1|f_1^2}\Big)^{it} V_t(|D_1D_2|(f_1f_2)^2; k, t_{\tt u}) 
  d{\tt u} \, dt.
 \end{split}
 \end{displaymath}
 Throughout we agree on the convention that $D$ (with or without indices) denotes a negative discriminant and $\Delta$ denotes a negative fundamental discriminant. 
We make two immediate manipulations. Fix some $0 < \eta < 1/100$. By the concluding remark of the preceding section we can insert the function
\begin{equation}\label{def-omega}
\omega(t_{\tt u}) = 1 - e^{-(t_{\tt u}/K^{1/2- \eta})^{10^6/\eta}}
\end{equation}
(not to be confused with the constant $\omega$ in the previous display) into the ${\tt u}$-integral, and we can truncate the $n, m$-sum at 
$$n,m \ll K^{\eta}. $$
Both transformations induce an admissible error, the former due to that $\omega(t_{\tt u}) - 1 \ll_A K^{-A}$ for every $A > 0$ and $t_{\tt u} \gg K^{\frac{1}{2} - \frac{\eta}{2}}$. Note that $\omega$ is even, holomorphic, within $[0, 1]$ for $t_{\tt u} \in \Bbb{R} \cup \{i/2, -i/2\}$ and satisfies
\begin{equation}\label{prop-omega}
\begin{split}
 \omega(t_{\tt u}) & \ll  K^{-10^6}, \quad |t_{\tt u}| \leq K^{\frac{1}{2} - 2\eta},\\
 |t_{\tt u}|^j \frac{d^j}{dt^j_{\tt u}} \omega(t_{\tt u}) &  \ll_j 1
\end{split}
\end{equation}
for $j \in \Bbb{N}_0$. In particular, up to a negligible error we may ignore the constant function ${\tt u} = \sqrt{3/\pi}$ with $t_{\tt u} = i/2$. As before we denote by $\int_{\Lambda_{\text{ev}}}^{\ast}$ a spectral sum/integral over the non-residual spectrum, i.e.\ everything except the constant function. 

Note that $\lambda(n)$ depends on $h$, as it is a Hecke eigenvalue of $f_h$. Before we can sum over $h$ using Lemma \ref{lem1}, we must first combine $\lambda(n)$ with $c(|D_2|)$. To this end we recall \eqref{fourier-relation1} and recast $\mathcal{N}_{\text{av}}(K)$, up to a small error, coming from the truncation of the $n, m$-sum and the ${\tt u}$-integral, as
 \begin{equation}\label{slightly-simp}
\begin{split}
  &  \frac{12}{  \omega K^2}  \sum_{k \in 2\Bbb{N}} W\left(\frac{k}{K}\right) \sum_{h \in B_{k-1/2}^+(4)} \frac{\pi^2}{90} \cdot \frac{18\sqrt{\pi}}{2  }\cdot 2\sum_{\substack{(n, m) = 1\\ n, m \leq K^{\eta}}} \frac{  \mu(n)\mu^2(m)}{n^{3/2}m^{3}} \\
  & \int_{-\infty}^{\infty}  \int^{\ast}_{\Lambda_{\text{\rm ev}}} \omega(t_{\tt u})  \frac{\Gamma(k - 3/2)}{(4\pi)^{k-3/2} \| h \|^2}
    \sum_{f_1, f_2, D_1, D_2}  \sum_{\substack{d_1 \mid d_2 \mid n\\ (d_1d_2)^2 \mid n^2D_2}} \left(\frac{d_1}{d_2} \right)^{1/2} \chi_{  D_2}\Big(\frac{d_2}{d_1}\Big)\left(\frac{n}{d_1d_2}\right)^{3/2 - k}\\
   &   \frac{c(|D_1|) c(|D_2|n^2/(d_1d_2)^2) P(D_1; {\tt u})\overline{P(D_2; {\tt u})}}{f_1f_2 |D_1D_2|^{k/2}}  \Big(\frac{|D_2|f_2^2}{|D_1|f_1^2}\Big)^{it} V_t(|D_1D_2|(f_1f_2)^2; k, t_{\tt u}) 
  d{\tt u} \, dt .
 \end{split}
 \end{equation}
 We can now sum over $h$ using Lemma \ref{lem1}, and we start with an analysis of the diagonal term which is given by 
  \begin{displaymath}
\begin{split}
 \mathcal{N}^{\text{diag}}&(K) =  \frac{12}{\omega  K^2}   \sum_{k \in 2\Bbb{N}} W\left(\frac{k}{K}\right)   \frac{\pi^2}{90} \cdot \frac{18\sqrt{\pi}}{2  } \cdot 2\cdot \frac{2}{3}\sum_{\substack{(n, m) = 1\\ n, m \leq K^{\eta}}} \frac{  \mu(n)\mu^2(m)}{n^{3/2}m^{3}}  \int_{-\infty}^{\infty}   \int^{\ast}_{\Lambda_{\text{\rm ev}}} \omega(t_{\tt u})  \sum_{f_1, f_2, D}  \\
  &  \sum_{\substack{d_1 \mid d_2 \mid n\\ (d_1d_2)^2 \mid n^2D}} \left(\frac{d_1}{d_2} \right)^{1/2} \chi_{  D}\Big(\frac{d_2}{d_1}\Big)   \frac{  P(Dn^2/(d_1d_2)^2; {\tt u})\overline{P(D; {\tt u})}}{f_1f_2 (|D|n/(d_1d_2))^{3/2}}\Big(\frac{d_1d_2f_2}{nf_1}\Big)^{2it} V_t\Big(\Big(\frac{|D|nf_1f_2}{d_1d_2}\Big)^2; k, t_{\tt u}\Big) d{\tt u} \, dt .
  \end{split}
 \end{displaymath}
 The analysis of this term occupies this and the following two sections, and we will eventually show that $ \mathcal{N}^{\text{diag}}(K)  = 4\log K + O(1)$.  The discussion of the off-diagonal term is postponed to Section \ref{off-off}. 
 
 We write   $d_2 = d_1\delta$, $n = d_1\delta \nu$, so that  $d_1^2 \mid \nu^2 D$. Since $n$ is squarefree, this implies $d_1^2 \mid D$. We   write $d_1 = d$ and $D d^2$ in place of $D$. With this notation we recast $ \mathcal{N}^{\text{diag}}(K)$ as 
 \begin{displaymath}
\begin{split}
  \frac{12}{  \omega K^2} & \sum_{k \in 2\Bbb{N}} W\left(\frac{k}{K}\right)   \frac{\pi^2}{90} \cdot \frac{18\sqrt{\pi}}{2  } \cdot 2 \cdot \frac{2}{3}\sum_{\substack{(d\delta \nu, m) = 1\\ d\delta \nu, m \leq K^{\eta}}} \frac{  \mu(d\delta \nu)\mu^2(m)}{(d\delta\nu)^{3/2}m^{3}}  \int_{-\infty}^{\infty}   \int^{\ast}_{\Lambda_{\text{\rm ev}}} \omega(t_{\tt u})    \sum_{f_1, f_2, D}   \frac{\chi_D(\delta)}{\delta^{1/2}} \\ &\frac{  P(D\nu^2; {\tt u})\overline{P(Dd^2; {\tt u})}}{f_1f_2 (d|D|\nu)^{3/2}}  \Big(\frac{df_2}{\nu f_1}\Big)^{2it} V_t ( ( |D|\nu df_1f_2 )^2; k, t_{\tt u} )  d{\tt u} \, dt.
 \end{split}
 \end{displaymath}
  From the Katok-Sarnak formula in combination with \eqref{non-fund} in the cuspidal case and from \eqref{eisen1} in combination with \eqref{basicL1} in the Eisenstein case we conclude\footnote{This remains true in the excluded case ${\tt u} = \text{const}$ if we define   $\lambda_{\tt u} = \rho_1$.}
 $$ P(\Delta f^2, {\tt u}) =   P(\Delta , {\tt u})\alpha_{\tt u}(f), \quad \alpha_{\tt u}(f) = f^{1/2} \sum_{d \mid f} \mu(d) \chi_{\Delta}(d) \lambda_{\tt u}(f/d) d^{-1/2}$$
for a fundamental discriminant $\Delta$ where $\alpha_{\tt u}(f)$ depends also on $\Delta$, which we suppress from the notation. Using this notation along with \eqref{katok-Sarnak} in the cuspidal case and \eqref{eisen2} in the Eisenstein case we obtain
  \begin{equation*}
\begin{split}
 \mathcal{N}^{\text{diag}}(K) &= 
   \frac{12}{  \omega K^2}   \sum_{k \in 2\Bbb{N}} W\left(\frac{k}{K}\right)   \frac{\pi^2}{90} \cdot \frac{18\sqrt{\pi}}{2  }\cdot 2 \cdot \frac{2}{3} \cdot \frac{1}{4}\sum_{\substack{(d\delta \nu, m) = 1\\ d\delta \nu, m \leq K^{\eta}}} \frac{  \mu(d\delta \nu)\mu^2(m)}{(d\delta\nu)^{3/2}m^{3}}  \int_{-\infty}^{\infty}  \int^{\ast}_{\Lambda_{\text{\rm ev}}} \omega(t_{\tt u})\\
&    \sum_{f_1, f_2}\sum_{\substack{D = \Delta f^2\\ (\delta, f) = 1}}      \frac{\chi_{\Delta}(\delta) \alpha_{\tt u} (f\nu)\overline{ \alpha_{\tt u}(f d) }  }{\delta^{1/2}f_1f_2 (df^2\nu)^{3/2}|\Delta|}   \frac{L({\tt u}, 1/2) L({\tt u},  \Delta, 1/2)}{\mathcal{L}({\tt u})} \Big(\frac{df_2}{\nu f_1}\Big)^{2it} V_t ( ( |D|\nu df_1f_2 )^2; k, t_{\tt u} )    d{\tt u} \, dt 
 \end{split}
 \end{equation*}
 where $\mathcal{L}({\tt u}) = L(\text{sym}^2 {\tt u}, 1)$ if ${\tt u}$ is cuspidal and $\mathcal{L}({\tt u}) = \frac{1}{2}|\zeta(1 + 2 i t)|^2$ if ${\tt u} = E(., 1/2 + it)$ is Eisenstein (with the obvious interpretation in the case $t = 0$). With later transformations in mind, we also restrict the $f$-sum to $f \leq K^{\eta}$. By trivial estimates along with Corollary \ref{Lfunc-cor} and \eqref{size-restr}, this induces an error of $O(K^{\varepsilon - \eta})$.  By the usual Hecke relations we have
$$ \alpha_{\tt u} (f\nu)\overline{ \alpha_{\tt u}(f d)} = f (d\nu)^{1/2} \sum_{\substack{d_1 \mid fd\\ d_2 \mid f\nu}} \sum_{d_3 \mid (\frac{fd}{d_1}, \frac{f\nu}{d_2}) }\frac{\mu(d_1)\chi_{\Delta}(d_1)     \mu(d_2)\chi_{\Delta}(d_2)  }{\sqrt{d_1d_2}} \lambda_{\tt u}(f^2 d\nu/(d_1d_2d_3^2)) .$$
We summarize the previous discussion as
\begin{equation}\label{weobt}
\begin{split}
\mathcal{N}^{\text{diag}}(K) & =   \frac{12}{ \omega  K^2}    \sum_{k \in 2\Bbb{N}} W\left(\frac{k}{K}\right)   \frac{\pi^2}{90} \cdot \frac{18\sqrt{\pi}}{2  }\cdot 2 \cdot \frac{2}{3} \cdot \frac{1}{4}\sum_{\substack{(d\delta \nu, m) = 1\\ d\delta \nu, m \leq K^{\eta}}} \frac{  \mu(d\delta \nu)\mu^2(m)}{(d\delta\nu)^{3/2}m^{3}}  \\
&\int_{-\infty}^{\infty} 
 \sum_{f_1, f_2, \Delta}\sum_{\substack{  f\leq K^{\eta}\\ (\delta, f) = 1}}      \frac{\chi_{\Delta}(\delta)   }{\delta^{1/2}f_1f_2 df^2\nu|\Delta|  }\Big(\frac{df_2}{\nu f_1}\Big)^{2it}    \\
 &\sum_{\substack{d_1 \mid fd\\ d_2 \mid f\nu}} \sum_{d_3 \mid (\frac{fd}{d_1}, \frac{f\nu}{d_2}) }\frac{\mu(d_1)\chi_{\Delta}(d_1)     \mu(d_2)\chi_{\Delta}(d_2)  }{\sqrt{d_1d_2}}  \mathcal{I}\Big(\Delta, t,  k, \frac{f^2 d \nu}{d_1d_2d_3^2}\Big) dt
 + O(K^{\varepsilon - \eta})
\end{split}
\end{equation}
where 
$$ \mathcal{I}(\Delta, t,k, r)  =  \int^{\ast}_{\Lambda_{\text{\rm ev}}}  \frac{L({\tt u}, 1/2) L({\tt u},  \Delta, 1/2)}{\mathcal{L}({\tt u})} \lambda_{\tt u}(r) h(t_{\tt u}) d{\tt u}$$
with
\begin{equation}\label{htau}
h(\tau) = \omega(\tau) V_t ( ( |\Delta|f^2\nu df_1f_2 )^2; k,\tau ) .
\end{equation}
The expression $\mathcal{I}$ depends also on $f^2 \nu d f_1f_2$, but we suppress this from the notation.
We insert the approximate functional equations \eqref{approx-basic} and \eqref{approx-basic1} getting
\begin{displaymath}
\begin{split}
\mathcal{I}(\Delta, t,k, r)  = & \ 4\int^{\ast}_{\Lambda_{\text{\rm ev}}}\sum_{{\tt n},{\tt m}} \frac{\lambda_{\tt u}({\tt n}) \lambda_{\tt u}({\tt m}) \chi_{\Delta}({\tt m})\lambda_{\tt u}(r)}{({\tt n}{\tt m})^{1/2}}W^+_{t_{\tt u}}({\tt n}) W^-_{t_{\tt u}}\Big(\frac{{\tt m}}{|\Delta|}\Big) \frac{1}{\mathcal{L}({\tt u})}  h(t_{\tt u}) d{\tt u}\\
& - \int_{-\infty}^{\infty}  \sum_{\pm} \frac{\zeta(1 \pm 2i\tau)\Gamma(\frac{1}{2} \pm i\tau)\pi^{\mp i\tau} e^{(1/2 \pm  i\tau)^2}}{(\frac{1}{2} \pm i \tau)\Gamma(\frac{1}{4} + \frac{i\tau}{2})\Gamma(\frac{1}{4}  - \frac{i\tau}{2})}\frac{|L(\chi_{\Delta}, 1/2 + i\tau)|^2}{|\zeta(1 + 2i\tau)|^2} \rho_{1/2 + i\tau}(r) h(\tau) \frac{d\tau}{2\pi}.
 \end{split}
 \end{displaymath}
Since the $\tau$-integral is rapidly converging, it is easy  to see that the polar term contributes at most $O(K^{\varepsilon - 1})$ to \eqref{weobt}, so from now on we focus on the first term of the preceding display. By the Hecke relations we can recast it as
 \begin{equation}\label{diagdiag}
  4 \sum_{d \mid r} \int^{\ast}_{\Lambda_{\text{\rm ev}}}\sum_{{\tt n},{\tt m}} \frac{\lambda_{\tt u}({\tt n}r/d) \lambda_{\tt u}({\tt m}) \chi_{\Delta}({\tt m}) }{(d{\tt n}{\tt m})^{1/2}} W^+_{t_{\tt u}}(d{\tt n}) W^-_{t_{\tt u}}\Big(\frac{{\tt m}}{|\Delta|}\Big) \frac{1}{\mathcal{L}({\tt u})}  h(t_{\tt u}) d{\tt u} .
  \end{equation}
 This is now in shape to apply the Kuznetsov formula for the even spectrum, Lemma \ref{kuz-even}. We treat the three terms on the right hand side of \eqref{kuz-even-form} separately and start with the diagonal term to which the next section is devoted. The two off-diagonal terms are treated in Section \ref{diag-off}.

\section{The diagonal diagonal term}\label{diag-diag}
The diagonal contribution equals
$$ \mathcal{I}^{\text{diag}}(\Delta, t, k, r) = 4 \sum_{d \mid r} \sum_{{\tt n}} \frac{ \chi_{\Delta}({\tt n}r/d) }{{\tt n}r^{1/2}} \int_{-\infty}^{\infty} W^+_{\tau}(d{\tt n}) W^-_{\tau}\Big(\frac{{\tt n}r/d}{|\Delta|}\Big)    \tau \tanh(\pi \tau) h(\tau) \frac{d\tau}{4\pi^2} .$$
Opening up the Mellin transform in the definition \eqref{v-t}, this equals
\begin{displaymath}
\begin{split}
\frac{1}{\sqrt{r}}   \int_{-\infty}^{\infty} & h(\tau)
\int_{(2)} \int_{(2)} \prod_{\pm}\Big( \frac{\Gamma(\frac{1}{2}(\frac{1}{2}  + s_1 \pm i\tau))\Gamma(\frac{1}{2}(\frac{3}{2} + s_2 \pm i\tau))}{\Gamma(\frac{1}{2}(\frac{1}{2} \pm i\tau))\Gamma(\frac{1}{2}(\frac{3}{2} \pm i\tau))  }  \Big) |\Delta|^{s_2}  L( \chi_{\Delta}, 1 + s_1 + s_2) 
 \\
 & \sum_{r_1r_2 = r}  \frac{\chi_{\Delta}(r_2) }{r_1^{s_1} r_2^{s_2}}    \frac{e^{s_2^2 +s_1^2}}{\pi^{s_1+s_2}s_1s_2} \frac{ds_1 \, ds_2}{(2\pi i)^2}
  \tau \tanh(\pi \tau) \frac{d\tau}{\pi^2}.
\end{split}
\end{displaymath}
We shift the $s_1, s_2$-contours  to $\Re s_1 = \Re s_2 = -1/4$ getting
\begin{equation}\label{i-diag}
\begin{split}
\mathcal{I}^{\text{diag}}&(\Delta, t, k, r) = \frac{1}{\sqrt{r}} L( \chi_{\Delta}, 1) \sum_{r_1r_2 = r} \chi_{\Delta}(r_2) \int_{-\infty}^{\infty} h(\tau) \tau \tanh(\pi \tau) \frac{d\tau}{\pi^2} \\
&+ O\Big(\int_{-\infty}^{\infty} |h(\tau)| |\tau|^{3/4} d\tau \int_{-\infty}^{\infty} e^{\xi^2}\Big(\frac{| L(\chi_{\Delta}, 1/2 + i\xi)|}{|\Delta|^{1/4}|} +| L(\chi_{\Delta}, 3/4 + i\xi)|\Big)  d\xi \Big) .
 \end{split}
\end{equation}
We first deal with the error term and substitute it back into \eqref{weobt}. Roughly speaking the factor $|\tau|^{3/4}$ saves a factor $K^{1/8}$ from the trivial bound, while on average over $\Delta$ the $L$-values on the lines $1/2$ and $3/4$ are still bounded. More precisely, recalling \eqref{size-restr} the error term gives a total contribution of at most
$$\frac{1}{K^{3/2}}  \int_{-\infty}^{\infty}   \int_{-\infty}^{\infty}   \int_{-\infty}^{\infty} \sum_{\Delta} \Big(\frac{| L(\chi_{\Delta}, 1/2 + i\xi)|}{|\Delta|^{5/4}} +\frac{| L(\chi_{\Delta}, 3/4 + i\xi)|}{|\Delta|} \Big)\Big( 1 + \frac{|t|^2 + |\tau|^2}{K}\Big)^{-10} e^{\xi^2} d\xi \, |\tau|^{3/4}d\tau \, dt.$$
By a standard mean value bound for the $\Delta$-sum (e.g. \cite[Theorem 2]{HB}), the previous display can be bounded by $\ll K^{-1/8 + \varepsilon}$. We substitute the main term in \eqref{i-diag} into \eqref{weobt} getting
\begin{equation*}
\begin{split}
&\mathcal{N}^{\text{diag}, \text{diag}}(K)  =    \frac{12}{ \omega  K^2}   \sum_{k \in 2\Bbb{N}} W\left(\frac{k}{K}\right)   \frac{\pi^2}{90} \cdot \frac{18\sqrt{\pi}}{2  }\cdot 2 \cdot \frac{2}{3} \cdot \frac{1}{4}\sum_{\substack{(d\delta \nu, m) = 1\\ d\delta \nu, m \leq K^{\eta}}} \frac{  \mu(d\delta \nu)\mu^2(m)}{(d\delta\nu)^{3/2}m^{3}} \\
&\sum_{f_1, f_2, \Delta}\sum_{\substack{ f \leq K^{\eta}\\ (\delta, f) = 1}}    L( \chi_{\Delta}, 1)  \int_{-\infty}^{\infty} 
   \sum_{\substack{d_1 \mid fd\\ d_2 \mid f\nu}} \sum_{d_3 \mid (\frac{fd}{d_1}, \frac{f\nu}{d_2}) }      \sum_{r_1r_2 = \frac{f^2 d \nu}{d_1d_2d_3^2}}       \frac{\chi_{\Delta}(\delta)  \mu(d_1)\chi_{\Delta}(d_1)     \mu(d_2)\chi_{\Delta}(d_2) d_3 \chi_{\Delta}(r_2)  }{\delta^{1/2}f_1f_2 (df^2\nu)^{3/2} |\Delta|(f_1\nu/(f_2d))^{\pm 2it} }\\
   &   \int_{-\infty}^{\infty} h(\tau) \tau \tanh(\pi \tau) \frac{d\tau}{\pi^2}        dt
\end{split}
\end{equation*}
where we recall that $h(\tau)$ is given by \eqref{htau} and depends in particular on $t$, $k$ and $\Delta$ (as well as on $f, f_1, f_2, d, \nu$). A trivial estimate at this point using \eqref{size-restr} shows
\begin{equation}\label{trivial}
\begin{split}
\mathcal{N}^{\text{diag}, \text{diag}}(K)\ll &
\frac{1}{K^2}\sum_{k \in 2\Bbb{N}} W\left(\frac{k}{K}\right)  \sum_{f_1, f_2,\Delta} \frac{L(\chi_{\Delta}, 1)}{|\Delta| f_1f_2}\\
&\int_{-\infty}^{\infty} \int_{-\infty}^{\infty}\frac{1}{k^{1/2}}  \Big(1 + \frac{|\Delta| (f_1f_2)^2}{k^4}\Big)^{-A}  \Big(1 + \frac{|t|^2 + |\tau|^2}{k} \Big)^{-A}  |\tau| d\tau \, dt \ll \log K,
\end{split}
\end{equation}
(by standard mean value results for $L(\chi_{\Delta}, 1)$), but eventually we want an asymptotic formula, not an upper bound. 

Our next goal is to show that the $t$-integral forces $f_1\nu = f_2d$, up to a negligible error. To make this precise, we first observe that by the same computation as in \eqref{trivial} the portion $|t| \leq K^{2/5}$ contributes at most $O(K^{-1/10 + \varepsilon})$ to $ \mathcal{N}^{\text{diag}, \text{diag}}(K) $. We can therefore insert a smooth weight function that vanishes on $|t| \leq \frac{1}{2} K^{2/5}$ and is one on $|t| \geq K^{2/5}$. Integrating by parts sufficiently often using \eqref{size-restr}, we can then restrict to $$f_1\nu = f_2d (1 + O(K^{\varepsilon - 2/5}))$$
up to a negligible error. It is then easy to see that the terms $f_1 d \not = f_2 \nu$ contribute $O(K^{\varepsilon - 2/5})$ to $ \mathcal{N}^{\text{diag}, \text{diag}}(K) $. Having excluded these, we re-insert the portion $|t| \leq K^{2/5}$ to the $t$-integral, again  at the cost of an error $O(K^{\varepsilon-1/10})$. Finally we complete the $d, \delta, \nu, m, f$-sum at the cost of an error $O(K^{\varepsilon-\eta})$. Since $(\nu, d) = 1$, the equation $f_1d = f_2 \nu$ implies $f_1 = dg$, $f_2 = \nu g$ for some $g\in \Bbb{N}$. Substituting all this, we recast $\mathcal{N}^{\text{diag}, \text{diag}}(K)  $ as 
\begin{equation*}
\begin{split}
 &    \frac{12}{ \omega  K^2}    \sum_{k \in 2\Bbb{N}} W\left(\frac{k}{K}\right)   \frac{\pi^2}{90} \cdot \frac{18\sqrt{\pi}}{2  }\cdot 2 \cdot \frac{2}{3} \cdot \frac{1}{4}\sum_{ (d\delta \nu, m) = 1 } \frac{  \mu(d\delta \nu)\mu^2(m)}{(d\delta\nu)^{3/2}m^{3}} \sum_{\substack{g, f, \Delta \\ (\delta, f) = 1}}    L( \chi_{\Delta}, 1)  \\
&\int_{-\infty}^{\infty} 
   \sum_{\substack{d_1 \mid fd\\ d_2 \mid f\nu}} \sum_{d_3 \mid (\frac{fd}{d_1}, \frac{f\nu}{d_2}) }      \sum_{r_1r_2 = \frac{f^2 d \nu}{d_1d_2d_3^2}}       \frac{\chi_{\Delta}(\delta)  \mu(d_1)\chi_{\Delta}(d_1)     \mu(d_2)\chi_{\Delta}(d_2) d_3 \chi_{\Delta}(r_2)  }{\delta^{1/2}g^2d\nu (df^2\nu)^{3/2} |\Delta|}  \int_{-\infty}^{\infty} h(\tau) \tau \tanh(\pi \tau) \frac{d\tau}{\pi^2}        dt  
   \end{split}
\end{equation*}
up to an error of $O(K^{\varepsilon-\eta})$, where in the definition \eqref{htau} of $h$ we replace $f_1 = dg$, $f_2 = \nu g$. We substitute the definition in \eqref{htau} and open the Mellin transform \eqref{defV3}. So that the  previous display becomes
\begin{equation}\label{previousdisplay}
\begin{split}
&     \frac{12}{ \omega  K^2}   \sum_{k \in 2\Bbb{N}} W\left(\frac{k}{K}\right)   \frac{\pi^2}{90} \cdot \frac{18\sqrt{\pi}}{2  }\cdot 2 \cdot \frac{2}{3} \cdot \frac{1}{4} \int_{(3)} \frac{e^{v^2 }}{v} \mathscr{L}(2v) \\
&\int_{-\infty}^{\infty}  \int_{-\infty}^{\infty}  \mathcal{G}(k, \tau, v + 1/2 + it)   \mathcal{G}(k, \tau, v + 1/2 - it) \omega(\tau)   \tau \tanh(\pi \tau) \frac{d\tau}{\pi^2}        dt \frac{dv }{2\pi i} 
 \end{split}
\end{equation}
where
\begin{displaymath}
\begin{split}
\mathscr{L}(v) = \sum_{ (d\delta \nu, m) = 1 } \sum_{\substack{g, f, \Delta\\ (f, \delta) = 1}}  \sum_{\substack{d_1 \mid fd\\ d_2 \mid f\nu}} \sum_{d_3 \mid (\frac{fd}{d_1}, \frac{f\nu}{d_2}) }      \sum_{r_1r_2 = \frac{f^2 d \nu}{d_1d_2d_3^2}}    \frac{  \mu(d\delta \nu)\mu^2(m) \mu(d_1)\mu(d_2) \chi_{\Delta}(\delta d_1 d_2r_2) d_3 }{(d\nu)^{4 + 2v}( \delta g) ^{2+2v}f^{3+2v} m^3  |\Delta|^{1+v}}    L( \chi_{\Delta}, 1) .      
\end{split}
\end{displaymath}
Recall our general assumption that $\Delta$ runs over negative fundamental discriminants. The main term will come from the residue at $v = 0$ and we need to analytically continue $\mathscr{L}$ to $\Re v < 0$ and compute the residue at $v = 0$. The critical portion is the $\Delta$-sum which we analyze in the following lemma.
\begin{lemma}
\label{lemsip}
Let $\alpha \in \Bbb{N}$ and uniquely write $\alpha = \alpha_1^2 \alpha_2$ with $\mu^2(\alpha_2) = 1$. Then
$$\sum_{-X \leq  \Delta < 0} \chi_{\Delta}(\alpha) L( \chi_{\Delta}, 1) =  \frac{\zeta(2) X}{2\alpha_2} \prod_{p \mid \alpha}  \Big(1 + \frac{1}{p} - \frac{1}{p^{3}}\Big)^{-1} \prod_{p }\Big( 1 - \frac{1}{p^{2}} - \frac{1}{p^{3}} + \frac{1}{p^{4}}\Big)  + O\big(X^{13/18}\alpha^{1/4} (X\alpha)^{\varepsilon}\big).$$
In particular, the Dirichlet series
$$\mathscr{K}(v; \alpha) = \sum_{\Delta < 0} \frac{\chi_{\Delta}(\alpha) L( \chi_{\Delta}, 1)}{|\Delta|^{1+v}}$$
has analytic continuation to $\Re v > -5/18$ except for a simple pole at $v=0$ with residue
$$ \frac{\zeta(2)}{2\alpha_2} \prod_{p \mid \alpha}  \Big(1 + \frac{1}{p} - \frac{1}{p^{3}}\Big)^{-1} \prod_{p }\Big( 1 - \frac{1}{p^{2}} - \frac{1}{p^{3}} + \frac{1}{p^{4}}\Big) $$
and is bounded by $O_{\varepsilon}(|v|  \alpha^{1/4+\varepsilon})$ in the region $\Re v \geq -5/18 + \varepsilon$, $|v| \geq \varepsilon$.
\end{lemma}

\emph{Remark:} The computation of the leading constant in Lemma \ref{lemsip} seems to be a new result even in the case $\alpha = 1$ and features an interesting Euler product. See \cite{Ju} for similar Euler products for averages at the point 1/2. \\

 \begin{proof} Let $w$ be a fixed smooth function that is equal to $1$ on $[0, 1]$ and vanishes on $[2, \infty)$. Let $Y \geq 1$. We have
 $$\sum_n\frac{\chi_{\Delta}(n)}{n} w\left( \frac{n}{Y}\right) = \int_{(2)} L(\chi_{\Delta}, 1+s) \widehat{w}(s)Y^s \frac{ds}{2\pi i} =    L( \chi_{\Delta}, 1) + O(Y^{-1/2} |\Delta|^{1/6 + \varepsilon})$$
 where $\widehat{w}$ in the present case denotes the Mellin transform and the left hand side comes from a contour shift to $\Re s = -1/2$ and the Conrey-Iwaniec \cite{CoIw} subconvexity bound for real characters. We obtain
 \begin{equation}\label{char-sum}
 \sum_{-X \leq  \Delta < 0} \chi_{\Delta}(\alpha) L( \chi_{\Delta}, 1)= \sum_n \frac{1}{n}w\left( \frac{n}{Y}\right) \sum_{-X \leq \Delta < 0} \chi_{\Delta}(\alpha n) + O(X^{7/6+\varepsilon} Y^{-1/2}).
 \end{equation}
We decompose the main term as $S_{\square} + S_{\not= \square}$ depending on  whether $\alpha n$ is a square or not. 
We first consider the portion $S_{\square}$ where $n$ is restricted to  $n = \alpha_2k^2$ with $k \in \Bbb{N}$. This gives
 \begin{equation}\label{er1}
 \begin{split}
S_{\square} & =  \sum_k \frac{1}{\alpha_2 k^2} w\left( \frac{\alpha_2k^2}{Y}\right) \sum_{\substack{-X \leq \Delta < 0\\ (\Delta, \alpha k) = 1}} 1 = \sum_k \frac{1}{\alpha_2 k^2}  \sum_{\substack{-X \leq \Delta < 0\\ (\Delta, \alpha k) = 1}} 1 + O\Big(\frac{X}{Y}\Big).
\end{split}
 \end{equation}
We decompose the main term as $ S_{\square}^{\text{odd}} + S_{\square}^{\text{even, 4}} + S_{\square}^{\text{even, 8}}$ depending on whether $\Delta$ is odd, exactly divisible 4 or exactly divisible by 8. We have
$$S_{\square}^{\text{odd}}  = \frac{1}{2\alpha_2} \sum_k \frac{1}{k^2} \sum_{\substack{m \leq X\\ (m, \alpha k) = 1}} \mu^2(m)(\chi_0(m) - \chi_{-4}(m))$$
where $\chi_0$ is the trivial character modulo 4 and $\chi_{-4}$ the non-trivial character modulo 4. For $\chi \in \{\chi_0, \chi_{-4}\}$ we consider the Dirichlet series
\begin{displaymath}
 \begin{split}
&\frac{1}{2\alpha_2}\sum_k \frac{1}{k^2} \sum_{ (m, \alpha k) = 1} \frac{\mu^2(m)\chi(m)}{m^s}  = \frac{\zeta(2)}{2\alpha_2} \prod_{p \nmid \alpha}  \Big(1 + \frac{\chi(p)}{p^s} - \frac{\chi(p)}{p^{s+2}}\Big) \\
&= \frac{\zeta(2)}{2\alpha_2} \prod_{p \mid \alpha}  \Big(1 + \frac{\chi(p)}{p^s} - \frac{\chi(p)}{p^{s+2}}\Big)^{-1}  L(\chi, s) \prod_p\Big( 1 - \frac{\chi(p)^2}{p^{2s}} - \frac{\chi(p)}{p^{s+2}} + \frac{\chi(p)^2}{p^{2s+2}}\Big).
\end{split}
 \end{displaymath}
A standard application of Perron's formula (e.g.\ \cite[Corollary II.2.4]{Te}) shows now that
\begin{equation}\label{er2}
S_{\square}^{\text{odd}} = \frac{\zeta(2) X}{4\alpha_2} \prod_{\substack{p \mid \alpha\\ p \nmid 2}}  \Big(1 + \frac{1}{p} - \frac{1}{p^{3}}\Big)^{-1}  \prod_{p \nmid 2}\Big( 1 - \frac{1}{p^{2}} - \frac{1}{p^{3}} + \frac{1}{p^{4}}\Big) + O\Big( \frac{X^{1/2}}{\alpha_2}(\alpha_2 X)^{\varepsilon}\Big).
\end{equation}
We have $S_{\square}^{\text{even, 4}} \not= 0$ and $S_{\square}^{\text{even, 8}} \not= 0$ only if $\alpha$ is odd, which we assume from now on. Then 
$$S_{\square}^{\text{even, 4}}  = \frac{1}{\alpha_2} \sum_{(k, 2) = 1} \frac{1}{k^2} \sum_{\substack{m \leq X/4\\ (m, \alpha k) = 1\\ m \equiv 1  (\text{mod } 4)}} \mu^2(m) = \frac{1}{2\alpha_2} \sum_{(k, 2) = 1} \frac{1}{k^2} \sum_{\substack{m \leq X/4\\ (m, \alpha k) = 1 }} \mu^2(m)(\chi_0(m) + \chi_{-4}(m)) $$
and by the same computation we obtain
$$S_{\square}^{\text{even, 4}} = \frac{\zeta(2) X}{16\alpha_2} \Big(1 - \frac{1}{4}\Big) \prod_{p \mid \alpha}  \Big(1 + \frac{1}{p} - \frac{1}{p^{3}}\Big)^{-1}  \prod_{p \nmid 2}\Big( 1 - \frac{1}{p^{2}} - \frac{1}{p^{3}} + \frac{1}{p^{4}}\Big) + O\Big( \frac{X^{1/2}}{\alpha_2}(\alpha_2 X)^{\varepsilon}\Big).$$
Finally,
$$S_{\square}^{\text{even, 8}}  = \frac{1}{\alpha_2} \sum_{(k, 2) = 1} \frac{1}{k^2} \sum_{\substack{m \leq X/8\\ (m, \alpha k) = 1\\ m \equiv 1  (\text{mod } 2)}} \mu^2(m) = \frac{1}{\alpha_2} \sum_{(k, 2) = 1} \frac{1}{k^2} \sum_{\substack{m \leq X/8\\ (m, \alpha k) = 1 }} \mu^2(m)\chi_0(m)$$
where now $\chi_0$ is the unique character modulo 2, and we get the same main term for $S_{\square}^{\text{even, 8}} $. Putting everything together, we obtain for $\alpha_2$ odd that 
\begin{displaymath}
\begin{split}
S_{\square}& =  \frac{\zeta(2) X}{\alpha_2} \prod_{p \mid \alpha}  \Big(1 + \frac{1}{p} - \frac{1}{p^{3}}\Big)^{-1} \prod_{p }\Big( 1 - \frac{1}{p^{2}} - \frac{1}{p^{3}} + \frac{1}{p^{4}}\Big) \cdot \frac{16}{11} \Big( \frac{1}{4} + \frac{1}{16} \cdot\frac{3}{4} \cdot 2\Big) + O\Big(\frac{X}{Y}+ \frac{X^{1/2}}{\alpha_2}(\alpha_2 X)^{\varepsilon}\Big)\\
& = \frac{\zeta(2) X}{2\alpha_2} \prod_{p \mid \alpha}  \Big(1 + \frac{1}{p} - \frac{1}{p^{3}}\Big)^{-1} \prod_{p }\Big( 1 - \frac{1}{p^{2}} - \frac{1}{p^{3}} + \frac{1}{p^{4}}\Big)  + O\Big(\frac{X}{Y}+ \frac{X^{1/2}}{\alpha_2}(\alpha_2 X)^{\varepsilon}\Big)
\end{split}
\end{displaymath}
and for $\alpha_2$ even that
\begin{displaymath}
\begin{split}
S_{\square}& =  \frac{\zeta(2) X}{\alpha_2} \prod_{p \mid \alpha}  \Big(1 + \frac{1}{p} - \frac{1}{p^{3}}\Big)^{-1} \prod_{p }\Big( 1 - \frac{1}{p^{2}} - \frac{1}{p^{3}} + \frac{1}{p^{4}}\Big) \cdot \frac{16}{11}\cdot  \frac{1}{4}\Big(1  + \frac{1}{2} -\frac{1}{8}  \Big) + O\Big(\frac{X}{Y}+ \frac{X^{1/2}}{\alpha_2}(\alpha_2 X)^{\varepsilon}\Big)\\
& = \frac{\zeta(2) X}{2\alpha_2} \prod_{p \mid \alpha}  \Big(1 + \frac{1}{p} - \frac{1}{p^{3}}\Big)^{-1} \prod_{p }\Big( 1 - \frac{1}{p^{2}} - \frac{1}{p^{3}} + \frac{1}{p^{4}}\Big)  + O\Big(\frac{X}{Y}+ \frac{X^{1/2}}{\alpha_2}(\alpha_2 X)^{\varepsilon}\Big).
\end{split}
\end{displaymath}
Note how beautifully the constants fit together.  We return to \eqref{char-sum} and study  $S_{\not=\square}$ where the $n$-sum is restricted to $\alpha n \not= \square$. We write $n = 2^{\nu} n'$ and $\alpha = 2^a \alpha'$ with $n', \alpha'$ odd. We need to bound
$$ \sum_{\substack{2^{\nu} n' \leq 2Y\\ 2^{\nu+a}n'\alpha' \not= \square}} \frac{1}{2^{\nu} n'  }  \Big| \sum_{- X\leq \Delta < 0} \chi_{\Delta}(2^{a + \nu}) \chi_{\Delta}(\alpha'n')\Big|.$$
The number $\alpha'n'$ is not a square, unless it is 1, in which case $a+\nu$ is necessarily odd, so that the $\Delta$-sum and hence the entire expression is $O(1)$. The number $\chi_{\Delta}(2^{a + \nu})$ depends only on $\Delta$ modulo 8, and we can split the sum into residue classes modulo 8 to make it independent of that factor.  If $\alpha' n'$ is not a square and odd, then $\Delta \mapsto (\frac{\Delta}{\alpha' n'})$ is defined for every negative $\Delta$ and in fact a possibly imprimitive, but certainly non-trivial character modulo $\alpha'n'$. Splitting even into residue classes modulo 16, we can detect the condition that $\Delta$ is a fundamental discriminant by requiring $\Delta$ or $\Delta/4$ be squarefree. Thus it suffices to bound
$$\sum_{m \leq X'} \mu^2(m) \psi(m)$$
 for a character $\psi$ of conductor $\ll \alpha Y$ and $X' \leq X$.  Writing $\mu^2 (m) = \sum_{d^2 \mid m} \mu(d)$ and using the P\'olya-Vinogradov inequality,  we bound the previous display  by
 \begin{equation}\label{er3}
 \ll \sum_{d \leq X} \min\Big(\frac{X}{d^2}, (\alpha Y)^{1/2+\varepsilon}\Big) \ll (X^{1/2}( \alpha Y)^{1/4})^{1+\varepsilon}.
 \end{equation}
 Collecting error terms in \eqref{char-sum}, \eqref{er1}, \eqref{er2}, \eqref{er3}, the total error becomes
 $$\ll \Big(\frac{X^{7/6}}{Y^{1/2}} + \frac{X}{Y} + \frac{X^{1/2}}{\alpha_2} + X^{1/2}(\alpha Y)^{1/4}\Big)(Xy\alpha)^{\varepsilon}.$$
 We choose $Y = X^{8/9}\alpha^{-1/3} + 1$ to recover a total error of $(X^{13/18} \alpha^{1/6} + X^{1/9} \alpha^{1/3} + X^{1/2} \alpha^{1/4})(X\alpha)^{1+\varepsilon}$. This completes the proof of the asymptotic formula. The claim on the Dirichlet series $\mathscr{K}(v; \alpha)$ follows by partial summation. 
 \end{proof}
 
 With the notation of the previous lemma we can write
$$\mathscr{L}(v) = \sum_{ (d\delta \nu, m) = 1 } \sum_{\substack{g, f\\ (f, \delta) = 1}}  \sum_{\substack{d_1 \mid fd\\ d_2 \mid f\nu}} \sum_{d_3 \mid (\frac{fd}{d_1}, \frac{f\nu}{d_2}) }      \sum_{r_1r_2 = \frac{f^2 d \nu}{d_1d_2d_3^2}}    \frac{  \mu(d\delta \nu)\mu^2(m) \mu(d_1)\mu(d_2)  d_3 }{(d\nu)^{4 + 2v}( \delta g) ^{2+2v}f^{3+4v} m^3  }    \mathscr{K}(v,        \delta d_1 d_2r_2
),
$$
and we conclude that $\mathscr{L}$ has analytic continuation to $\Re v > -13/18$ except for a simple pole at $v = 0$ and polynomial (in fact linear) bounds on vertical lines. If $\text{rad}(n)$ denotes the squarefree kernel of $n$, then 
\begin{equation}\label{euler}
\begin{split}
&\underset{v=0}{\text{res}} \mathscr{L}(v) = \frac{\zeta(2)^2}{2}\prod_{p }\Big( 1 - \frac{1}{p^{2}} - \frac{1}{p^{3}} + \frac{1}{p^{4}}\Big) \\
&\sum_{ (d\delta \nu, m) = 1 } \sum_{ (f, \delta) = 1}  \sum_{\substack{d_1 \mid fd\\ d_2 \mid f\nu}} \sum_{d_3 \mid (\frac{fd}{d_1}, \frac{f\nu}{d_2}) }      \sum_{r_1r_2 = \frac{f^2 d \nu}{d_1d_2d_3^2}}    \frac{  \mu(d\delta \nu)\mu^2(m) \mu(d_1)\mu(d_2)  d_3 }{(d\nu)^{4  }\delta^2 f^{3 } m^3   \text{rad}(\delta d_1 d_2r_2) } \prod_{p \mid   \delta d_1 d_2r_2} \Big(1 + \frac{1}{p} - \frac{1}{p^3}\Big)^{-1}
\end{split}
\end{equation}
 where we have implicitly computed the $g$-sum as $\zeta(2)$. Now a massive computation with Euler products, best performed with a computer algebra system, shows gigantic cancellation, and we obtain the beautiful formula
 $$\underset{v=0}{\text{res}} \mathscr{L}(v) = \frac{\zeta(2)^2}{2\zeta(4)}.$$
 With this we return to \eqref{previousdisplay} and shift the $v$-contour to $\Re v = -1/10$.    By \eqref{G1} and trivial bounds the remaining integral expression is bounded by
 $$\ll \frac{1}{K^2} \sum_{K \leq k \leq 2K} K^{-\frac{1}{2} - \frac{4}{10}} \int_{-\infty}^{\infty} \int_{-\infty}^{\infty} \Big(1 + \frac{|t|^2 + |\tau|^2}{K}\Big) |\tau| d\tau \, dt \ll K^{-4/10}.$$
It remains to deal with the double pole at $v = 0$ whose residue is given by
 \begin{equation*}
\begin{split}
\mathcal{R} = &    \frac{12}{  \omega K^2}   \sum_{k \in 2\Bbb{N}} W\left(\frac{k}{K}\right)   \frac{\pi^2}{90} \cdot \frac{18\sqrt{\pi}}{2  }\cdot 2 \cdot \frac{2}{3} \cdot \frac{1}{4} \\
&\int_{-\infty}^{\infty}  \int_{-\infty}^{\infty}  \underset{v=0}{\text{res}}\Big( \frac{e^{v^2 }}{v} \mathscr{L}(2v)  \mathcal{G}(k, \tau, v+ 1/2 + it)  \mathcal{G}(k, \tau, v  + 1/2 - it)\Big) \omega(\tau)   \tau \tanh(\pi \tau) \frac{d\tau}{\pi^2}        dt  . 
 \end{split}
\end{equation*}
We can remove the factor $\omega(\tau)\tanh(\pi \tau)$ at the cost of an error $O(K^{-\eta/2})$, cf.\ the definition \eqref{def-omega}. 
Using the definition \eqref{defG}, at $v=0$ we have the following Taylor expansion 
\begin{displaymath}
\begin{split}
   & \mathcal{G}(k, \tau, v+ 1/2 + it)     \mathcal{G}(k, \tau, v  + 1/2 - it)\\
    &= \mathcal{G}(k, \tau,   1/2 + it)     \mathcal{G}(k, \tau,   1/2 - it)  \Big\{1 + v\Big(\sum_{\pm, \pm}  \frac{\Gamma'}{\Gamma}\Big( \frac{k-1/2}{2} \pm  it \pm \frac{i\tau}{2}\Big) - 4\log \pi\Big) + O(v^2)\Big\}.
\end{split}
\end{displaymath}
We have $\frac{\Gamma'}{\Gamma}(z) = \log z +O(|z|^{-1})$ (for $\Re z \geq 1$, say) and so
$$\sum_{\pm}  \frac{\Gamma'}{\Gamma}\Big( \frac{k-1/2}{2} - it \pm \frac{i\tau}{2}\Big) - 4\log \pi = 4 \log k + O(1)$$
for $t, \tau \ll k^{2/3}$, say. Otherwise, the $t$ and $\tau$ integrals are negligible outside this region. In this range we can insert \eqref{taylor} 
to conclude that
\begin{equation*}
\begin{split}
\mathcal{R} = &   \frac{12}{\omega  K^2}   \sum_{k \in 2\Bbb{N}} W\left(\frac{k}{K}\right)   \frac{\pi^2}{90} \cdot \frac{18\sqrt{\pi}}{2  } \cdot 2 \cdot \frac{2}{3} \cdot \frac{1}{4} \cdot  \frac{1}{2} \cdot  4\log k \int_{-\infty}^{\infty}  \int_{-\infty}^{\infty} \frac{\zeta(2)^2}{2\zeta(4)}\frac{16}{\pi k^{1/2}} e^{-(4t^2 + \tau^2)/k}   | \tau|  \frac{d\tau}{\pi^2}        dt   + O(1).
 \end{split}
\end{equation*}
We evaluate the two integrals, sum over $k$ (by Poisson, for instance) and recall the definition of $\omega$ just before \eqref{Nav} getting
\begin{equation*}
\begin{split}
\mathcal{R} &=      \frac{12}{\omega  K^2}  \sum_{k \in 2\Bbb{N}} W\left(\frac{k}{K}\right)   \frac{\pi^2}{90} \cdot \frac{18\sqrt{\pi}}{2  } \cdot 2\cdot \frac{2}{3} \cdot \frac{1}{4} \cdot \frac{1}{2} \cdot  4\log k  \frac{\zeta(2)^2}{2\zeta(4)}\frac{16}{\pi k^{1/2}}      k \cdot \frac{\sqrt{\pi k}}{2} \frac{1}{\pi^2}  + O(1)\\
& =   \frac{12}{2} \cdot   \frac{\pi^2}{90} \cdot \frac{18\sqrt{\pi}}{2  } \cdot 2 \cdot  \frac{2}{3} \cdot \frac{1}{4} \cdot  \frac{1}{2} \cdot 4 \log K \cdot  \frac{\zeta(2)^2}{2\zeta(4)}\frac{16}{\pi}       \cdot \frac{\sqrt{\pi }}{2} \frac{1}{\pi^2}  + O(1) = 4 \log K + O(1). 
 \end{split}
\end{equation*}
We have now detected the main term and we conclude this section by stating that 
$$\mathcal{N}^{\text{diag}, \text{diag}}(K) = \mathcal{R} + O(1) = 4\log K + O(1).$$
  
\section{The diagonal off-diagonal term}\label{diag-off}

\subsection{Preparing the stage} We return to \eqref{diagdiag} and consider the  off-diagonal terms on the right hand side of \eqref{kuz-even-form}. We treat the first off-diagonal term in detail in the following three subsections. In Section \ref{144} show the minor modifications  to treat the second off-diagonal term. The first off-diagonal term is given by 
\begin{displaymath}
\begin{split}
\mathcal{I}^{\text{off-diag,1}}(\Delta, t,k,  r):= 8i \sum_{d \mid r} &\sum_{{\tt n},{\tt m}} \frac{ \chi_{\Delta}({\tt m}) }{(d{\tt n}{\tt m})^{1/2}}  \sum_c \frac{S({\tt m}, {\tt n}r/d, c)}{c} \\
&\int_{-\infty}^{\infty}   \frac{J_{2i\tau}(4\pi \sqrt{({\tt m}{\tt n}r/d)}/c)}{\sinh(\pi \tau)} W^+_{\tau}(d{\tt n}) W^-_{\tau}\Big(\frac{{\tt m}}{|\Delta|}\Big)h(\tau) \tau \tanh(\pi \tau) \frac{d\tau}{4\pi}
\end{split}
\end{displaymath}
where $h$ was defined in \eqref{htau} and depends in particular on $t$,  $k$ and $\Delta$. 
This needs to be inserted into \eqref{weobt} with $r = f^2d\nu/(d_1 d_2 d_3^2) \leq K^{4\eta}$. In \eqref{weobt} we apply a smooth partition   of unity to the $\Delta$-sum and consider a typical portion
$$\mathcal{J}^{(1)}(X, t,k,  r) := \sum_{\Delta}w\Big( \frac{|\Delta|}{X}\Big) \mathcal{I}^{\text{off-diag,1}}(\Delta, t,k,  r)$$
for a smooth function $w$ with compact support in $[1, 2]$. Our aim in this section is to prove the bound
\begin{equation}\label{J-bound}
  \mathcal{J}^{(1)}(X, t,k,  r) \ll X K^{1/2 - \eta}
\end{equation}
uniformly in $k \asymp K$, $t \leq K^{1/2 + \varepsilon}$, 
\begin{equation}\label{X}
   X \leq K^{2+\varepsilon}, \quad r \leq  K^{4\eta}
   \end{equation}
    and $f, f_1, f_2, d, \nu \in \Bbb{N}$ which are implicit in the definition \eqref{htau}. Taking \eqref{J-bound} for granted, we estimate \eqref{weobt} trivially to obtain a contribution of $O(K^{\varepsilon - \eta})$, which is admissible. So it remains to show \eqref{J-bound}, and to this end we start with some initial discussion. 

The $c$-sum in $\mathcal{I}^{\text{off-diag,1}}(\Delta, t,k,  r)$ is absolutely convergent, as can be seen by using the Weil bound for the Kloosterman sum and  shifting the $t$-contour to  $\Re i\tau = 1/3$, say, without crossing any poles. By the power series expansion for the Bessel function \cite[8.440]{GR} we have  $$J_{2i\tau}(x) \ll_{\Im \tau}  x^{-2\Im \tau}e^{\pi|\tau|} (1+|\tau|)^{-1/2}, \quad x \leq 1$$  and we  can therefore truncate the $c$-sum at $c \leq K^{10^6}$, say, at the cost of a very small error. Having done this, we can sacrifice holomorphicity of the integrand in the $\tau$-integral, and we   insert a smooth partition of unity  into the $\tau$-integral restricting to $\tau \asymp T$, say, with 
\begin{equation}\label{T}
K^{1/2 - 2\eta} \leq T \leq K^{1/2+\varepsilon},
\end{equation} otherwise $h(\tau)$ is negligible by \eqref{htau}, \eqref{size-restr} and \eqref{prop-omega}. 
We insert smooth partitions of unity into the ${\tt n}, {\tt m}, c$-sums and thereby restrict to ${\tt n} \asymp N$, ${\tt m} \asymp M$ $c \asymp C$, say, where 
\begin{equation}\label{NM}
N \leq T^{1+\varepsilon} , M \leq X T^{1 + \varepsilon}
\end{equation}
by \eqref{bound-wt} and initially $C \leq K^{10^6}$. Next we want to evaluate asymptotically the $\tau$-integral. To this end we use Lemma \ref{bessel-kuz} with the weight function ${\tt h} = {\tt h}_{{\tt n}, {\tt m}, \Delta}$ given by 
$$W^+_{\tau}(d{\tt n}) W^-_{\tau}\Big(\frac{{\tt m}}{|\Delta|}\Big)h(\tau)  w\Big( \frac{|\tau|}{T}\Big) = W^+_{\tau}(d{\tt n}) W^-_{\tau}\Big(\frac{{\tt m}}{|\Delta|}\Big)\omega(\tau) V_{\pm t}((|\Delta|f^2f_1f_2d\nu)^2; k, \tau) $$
where $w$ is the weight function occurring in the smooth partition of unity of the $\tau$-integral. By \eqref{bound-wt}, \eqref{prop-omega} and \eqref{size-restr}, the function ${\tt h}$ is ``flat'' in all variables, i.e.\
  $$\frac{d^{j_1}}{d{\tt n}^{j_1}}\frac{d^{j_2}}{d{\tt m}^{j_2}}\frac{d^{j_3}}{d\Delta^{j_3}} \frac{d^j}{dt^j}{\tt h}_{{\tt n}, {\tt m}, \Delta}(t) \ll_{\textbf{j}} K^{-1/2}  T^{-j}N^{-j_1} M^{-j_2} X^{-j_3}$$ for $k \asymp K$, $|\Delta| \asymp X$, ${\tt n} \asymp N$, ${\tt m} \asymp M$ and $\textbf{j} \in \Bbb{N}_0^4$, uniformly in all other variables.  Lemma \ref{bessel-kuz}a implies that we can restrict, up to a negligible error, to
\begin{equation}\label{C}
C \leq \frac{\sqrt{NMr_2}}{T^2} K^{\varepsilon},
\end{equation}
where $r_2 = r/d$, and up to a negligible error we are left with bounding
\begin{equation}\label{E1}
\begin{split}
\mathscr{J}^{(1)}_r(X, N, M, C, T) = &\frac{K^{\varepsilon}T^2}{(NM)^{3/4}(KC)^{1/2}}\sum_{r_1r_2=r} \frac{1}{(rr_1)^{1/4}} \\
&\sum_{\Delta, {\tt n} , {\tt m}, c}  \chi_{\Delta}({\tt m}) S({\tt m}, {\tt n} r_2, c) e\Big( \pm 2 \frac{\sqrt{{\tt n} {\tt m} r_2}}{  c}\Big) W\Big( \frac{|\Delta|}{X}, \frac{{\tt n}}{N}, \frac{{\tt m}}{M}, \frac{c}{C}\Big)
\end{split}
\end{equation}
 for a smooth   function $W$  with compact support in $[1, 2]^4$ and  bounded Sobolev norms with variables $X, N, M, C, T$ satisfying \eqref{X}, \eqref{NM}, \eqref{C}, \eqref{T}, respectively. 

 The basic strategy is now to apply Poisson summation first in the ${\tt n}$-sum and then in the ${\tt m}$-sum in Section \ref{132}. This shortens the variables in generic ranges, so that a trivial bound turns out to be of size $XK^{1/2 + 2\eta+\varepsilon}$. This is very close to our target \eqref{J-bound}. In Section \ref{133} we will  extract a character sum from this expression where the P\'olya-Vinogradov inequality produces the final saving, at least in generic ranges of the variables. In order to also treat non-generic ranges where some of the variables are relatively short, at each step we also apply trivial bounds along with Heath-Brown's large sieve. In particular, by Lemma \ref{HBlemma}b) we can bound
 \eqref{E1} by
 \begin{equation}\label{E2}
 \begin{split}
\mathscr{J}^{(1)}_r(X, N, M, C, T) &\ll \frac{K^{\varepsilon}T^2}{(NM)^{3/4}(KC)^{1/2}} X NMC^{3/2} \Big( \frac{1}{X^{1/2}} + \frac{1}{M^{1/2}}\Big) \\
&= \frac{K^{\varepsilon}T^2}{ K^{1/2}} CN^{1/4} \Big(M^{1/4} X^{1/2}  + \frac{X}{M^{1/4}} \Big).
\end{split}
 \end{equation}
 \subsection{Poisson summation}\label{132}
 Now we open the Kloosterman sum in \eqref{E1} and apply Poisson summation in ${\tt n}$ in residue classes modulo $c$. In this way we re-write $\mathscr{J}^{(1)}_r(X, N, M, C, T) $ as
 \begin{equation}\label{E3}
 \begin{split}
 &\frac{K^{\varepsilon}T^2}{(NM)^{3/4}(KC)^{1/2}}\sum_{r_1r_2 = r} \frac{1}{(rr_1)^{1/4}} \sum_{\Delta,  {\tt m}, c}  \chi_{\Delta}({\tt m}) \underset{\gamma\, (\text{mod } c)}{\left.\sum\right.^{\ast}}e\Big(\frac{{\tt m}\gamma}{c}\Big) \sum_{\nu \, (\text{mod } c)}e\Big( \frac{\bar{\gamma}\nu r_2}{ c}\Big)\\
& \quad\quad\quad\quad\frac{1}{c} \sum_{n \in \Bbb{Z}}e\Big(\frac{n \nu}{c}\Big) \int_{0}^{\infty} e\Big( \pm 2 \frac{\sqrt{x {\tt m} r_2}}{  c}\Big) W\Big( \frac{|\Delta|}{X}, \frac{x}{N}, \frac{{\tt m}}{M}, \frac{c}{C}\Big) e\Big( - \frac{x n}{c}\Big) dx. 
\end{split}
\end{equation}
We consider the character sum
\begin{equation}\label{char1}
\frac{1}{c}\underset{\gamma\, (\text{mod } c)}{\left.\sum\right.^{\ast}}e\Big(\frac{{\tt m}\gamma}{c}\Big) \sum_{\nu \, (\text{mod } c)}e\Big( \frac{\bar{\gamma}\nu r_2}{ c}\Big)e\Big(\frac{n \nu}{c}\Big) =  \underset{\substack{\gamma\, (\text{mod } c)\\ \bar{\gamma}r_2 \equiv -n\, (\text{mod } c)}}{\left.\sum\right.^{\ast}}e\Big(\frac{{\tt m}\gamma}{c}\Big). 
\end{equation}
This is non-zero only if $(n, c) = (r_2, c)$. We write $(r_2, c) = \delta$, $r_2 = \delta r_2'$, $c = \delta c'$, $n = \delta n'$ with $(n'r_2', c') = 1$ and recast \eqref{char1} as
$$  \underset{\substack{\gamma\, (\text{mod } c)\\ \gamma \equiv -r_2'\overline{n'}\, (\text{mod } c')}}{\left.\sum\right.^{\ast}}e\Big(\frac{{\tt m}\gamma}{c}\Big).$$
We decompose $\delta = \delta_1 \delta_2$ with $\delta_2$ maximal so that $(\delta_2, c') = 1$. Note that $r_2' \bar{n'}$ is coprime to $c'$, so the condition $(c, \gamma) = 1$ is equivalent to $(\delta_2, \gamma) = 1$. We factor $c = c'\delta_1 \cdot \delta_2$ with $(c'\delta_1, \delta_2) = 1$ and apply the Chinese Remainder Theorem to see that the previous display vanishes unless $\delta_1 \mid {\tt m}$, say ${\tt m} = \delta_1{\tt m}'$, in which case it equals
$$\delta_1 R_{\delta_2}({\tt m}') e\Big( -\frac{{\tt m}'r_2' \overline{n' \delta_2}}{c'}\Big)$$
where $R$ denotes the Ramanujan sum. 

Next we consider the $x$-integral in \eqref{E3}. The phase has a unique stationary point at $x = r_2{\tt m}/n^2$ if $\text{sgn}(n) = \pm$ and no stationary point if $\text{sgn}(n) = \mp$. If $N \leq r_2{\tt m}/n^2 \leq 2N$ and $\text{sgn}(n) = \pm$, we can apply the stationary phase lemma \cite[Proposition 8.2]{BKY} with $X=1$, $V = V_1 = Q = N$, $Y = \sqrt{NMr_2}/C \geq T^2K^{-\varepsilon} \geq K^{1/2 - \varepsilon}$ to see that the integral is given by
$$\Big(\frac{cr_2{\tt m}}{|n|^3}\Big)^{1/2} e\Big( \pm \frac{r_2{\tt m}}{c|n|}\Big)W_1\Big( \frac{|\Delta|}{X}, \frac{r_2{\tt m}}{n^2N}, \frac{{\tt m}}{M}, \frac{c}{C}\Big) $$
for a smooth function $W_1$  with compact support in $[1, 2]^4$ and  bounded Sobolev norms, up to a negligible error from truncating the series in \cite[(8.9)]{BKY}. Otherwise we apply integration by parts in the form of \cite[Lemma 8.1]{BKY} with  $X = 1$, $U = Q = N$, $Y = R = \sqrt{NMr_2}/C$ to conclude that the integral is negligible. 

Noting that with our previous notation
$$e\Big( \pm \frac{r_2{\tt m}}{c|n|}\Big) = e\Big(   \frac{d'_2{\tt m}' }{c'n' \delta_2}\Big)$$
for $\text{sgn}(n) = \pm$, we can now apply the additive reciprocity formula $e(1/ab) = e(\bar{a}/b)e(\bar{b}/a)$ for $(a, b) = 1$ to conclude that $\mathscr{J}^{(1)}_r(X, N, M, C, T) $ equals, up to a negligible error term, 
\begin{equation}\label{E2a}
 \begin{split}
 \frac{K^{\varepsilon}T^2}{(NM)^{3/4}(KC)^{1/2}}\sum_{r_1\delta_1\delta_2 d'_2 = r} &\frac{1}{(rr_1)^{1/4}} \underset{\substack{(c', n'r_2'\delta_2) = 1\\ \delta_1 \mid (c')^{\infty}}}{\sum_{\Delta,  {\tt m}', c'} \sum_{\pm n' \in \Bbb{N}}} \chi_{\Delta}({\tt m}'\delta_1)  \delta_1 R_{\delta_2}({\tt m}') e\Big( \frac{{\tt m}'r_2' \overline{c'}}{n' \delta_2}\Big)  \\
 &\Big(\frac{c'd'_2{\tt m}'}{\delta_2|n'|^3}\Big)^{1/2}W_1\Big( \frac{|\Delta|}{X}, \frac{d'_2{\tt m}'}{\delta_2 (n')^2N}, \frac{{\tt m}'\delta_1}{M}, \frac{c'\delta_1\delta_2}{C}\Big).  
 \end{split}
\end{equation}
Here we recall the notation conventions from Section \ref{15} regarding expressions $\delta \mid c^{\infty}$ etc. 
 For easier readability we remove all the dashes at the variables, and we define 
$$W_2(x, y, z, w) = w^{1/2} y^{-3/2} z^{1/2} W_1(x, z/y^2, z, w).$$
We also open the Ramanujan sum $R_{\delta_2}({\tt m}) = \sum_{\delta_3 \mid (\delta_2, {\tt m})} \delta_3\mu(\delta_2/\delta_3)$,  write $\delta_2 = \delta_3\delta_4$ and  replace $m$ with $m\delta_3$. Finally we drop the $\pm$-sign in the summation condition on $n$ (both cases are identical). With this notation we can re-write \eqref{E2a} as
\begin{equation}\label{E4}
 \begin{split}
 \frac{K^{\varepsilon}T^2}{M K^{1/2}}&\sum_{r_1\delta_1\delta_3\delta_4 r_2 = r} \frac{\delta_1^{3/4} \delta_3 \mu(\delta_4)}{(rr_1r_2 \delta_4)^{1/4} } \\
 &\underset{\substack{(c, nr_2\delta_3\delta_4) = 1\\ \delta_1 \mid c^{\infty}}}{\sum_{\Delta,  {\tt m}, c, n}  } \chi_{\Delta}({\tt m}\delta_1\delta_3)     e\Big( \frac{{\tt m}  r_2 \overline{c}}{n \delta_4}\Big)   W_2\Big( \frac{|\Delta|}{X}, \frac{n\delta_3 \sqrt{N \delta_1\delta_4}}{\sqrt{Mr_2}}, \frac{{\tt m}\delta_1\delta_3}{M}, \frac{c\delta_1\delta_3\delta_4}{C}\Big).  
 \end{split}
\end{equation}
Before we continue to transform this expression, we estimate trivially with the large sieve (Lemma \ref{HBlemma}b) 
\begin{equation}\label{E4a}
 \begin{split}
&\mathscr{J}^{(1)}_r(X, N, M, C, T) \\
& \ll  \frac{K^{\varepsilon}T^2}{M K^{1/2}}\sum_{r_1\delta_1\delta_3\delta_4 r_2 = r} \frac{\delta_1^{3/4}\delta_3}{(rr_1r_2\delta_4)^{1/4} }  X  \frac{\sqrt{Mr_2} }{\delta_3\sqrt{N\delta_1\delta_4}} \frac{M}{\delta_1\delta_3} \frac{C}{\delta_1\delta_3\delta_4} \Big( \Big(\frac{\delta_1}{M}\Big)^{1/2} + X^{-1/2}\Big) \\
&\ll \frac{K^{\varepsilon} T^2 C}{(KN)^{1/2}} (X + (XM)^{1/2}).
 \end{split}
\end{equation}
Our next goal is to apply Poisson summation in ${\tt m}$ restricted to residue classes $\delta_4 n |\Delta|$. For $m \in \Bbb{Z}$, this leads to the character sum
$$\sum_{\mu \, (\text{mod } \delta_4 n |\Delta|)} \chi_{\Delta}(\mu)  e\Big( \frac{\mu r_2 \overline{c}}{n \delta_4}\Big) e\Big(\frac{\mu m}{\delta_4 n|\Delta|}\Big) =  \sum_{\mu \, (\text{mod } \delta_4 n |\Delta|)} \chi_{\Delta}(\mu)  e\Big( \frac{\mu (  r_2 \overline{c}|\Delta| + m)}{ \delta_4n|\Delta|}\Big) . $$ 
We decompose both $n$ and $\delta_4$ according to their common divisor with $\Delta$. We write $n = n_1 n_2$ and $\delta_4 = \delta_5 \delta_6$ where $n_1\delta_5 \mid  \Delta^{\infty}$, $(n_2\delta_6, \Delta) = 1$. We obtain by the Chinese Remainder Theorem that the above character sum equals
\begin{displaymath}
\begin{split}
& \sum_{\mu \, (\text{mod } \delta_6 n_2)}   e\Big( \frac{\mu\overline{\delta_5 n_1|\Delta|} (  r_2 \overline{c}|\Delta| + m)}{n_2 \delta_6 }\Big)\sum_{\mu \, (\text{mod } \delta_5 n_1 |\Delta|)} \chi_{\Delta}(\mu)  e\Big( \frac{\mu \overline{\delta_6n_2}(  r_2 \overline{c}|\Delta| + m)}{ \delta_5n_1|\Delta|}\Big).
 \end{split}
\end{displaymath}
The first sum vanishes unless $n_2\delta_6 \mid   r_2|\Delta| + m c$ in which case it equals $n_2\delta_6$. Since $\Delta$ is a negative fundamental discriminant, the second sum vanishes unless $n_1\delta_5 \mid   r_2|\Delta| + m c$ in which case it equals
$$i\sqrt{|\Delta|} n_1 \delta_5\chi_{\Delta}\Big( n_2\delta_6   \frac{  r_2 \bar{c}|\Delta| + m}{n_1\delta_5}\Big). $$
Having this evaluation available, the Poisson summation formula transforms  \eqref{E4} into
\begin{equation}\label{E5}
 \begin{split}
 \frac{K^{\varepsilon}T^2}{M K^{1/2}}&\sum_{r_1\delta_1\delta_3\delta_5\delta_6 r_2 = r} \frac{\delta_1^{3/4} \delta_3 \mu(\delta_5\delta_6)}{(rr_1r_2 \delta_5\delta_6)^{1/4} } \underset{\substack{(c, n_1n_2r_2\delta_3\delta_5\delta_6) = 1\\ \delta_1 \mid c^{\infty}, n_1\delta_5 \mid  \Delta^{\infty}, (n_2 \delta_6, \Delta) = 1\\ n_1n_2\delta_5\delta_6 \mid   r_2|\Delta| + mc }}{\sum_{\Delta,   c, n_1, n_2} \sum_{  m\in \Bbb{Z}}} \chi_{\Delta}( \delta_1\delta_3)    \frac{\chi_{\Delta}\Big( n_2\delta_6  \frac{  r_2 \bar{c}|\Delta| + m}{n_1\delta_5}\Big)}{ |\Delta|^{1/2}}  \\
 &\int_0^{\infty} W_2\Big( \frac{|\Delta|}{X}, \frac{n_1n_2 \delta_3\sqrt{N \delta_1 \delta_5\delta_6}}{\sqrt{Mr_2}}, \frac{x\delta_1\delta_3}{M}, \frac{c\delta_1\delta_3\delta_5\delta_6}{C}\Big) e\Big(- \frac{mx}{n_1n_2\delta_5\delta_6 |\Delta|}\Big)dx,
 \end{split}
\end{equation}
up to a factor $i$.
The above integral is just the Fourier transform of $W_2$ with respect to the third variable which we write as 
$$ \frac{M}{\delta_1\delta_3}W_3\Big( \frac{|\Delta|}{X}, \frac{n_1n_2 \delta_3\sqrt{N \delta_1\delta_5\delta_6}}{\sqrt{Mr_2}}, \frac{mM }{n_1n_2\delta_5\delta_6 |\Delta|\delta_1\delta_3}, \frac{c\delta_1\delta_3\delta_5\delta_6}{C}\Big) . $$
Defining $W_4(x, y, z, w) = W_3(x, y, z/(xy), w)$ we recast the previous display as
$$ \frac{M}{\delta_1\delta_3}W_4\Big( \frac{|\Delta|}{X}, \frac{n_1n_2\delta_3 \sqrt{N \delta_1\delta_5\delta_6}}{\sqrt{Mr_2}}, \frac{m\sqrt{NM   } }{ (r_2\delta_1 \delta_5\delta_6 )^{1/2}X}, \frac{c\delta_1\delta_3\delta_5\delta_6}{C}\Big) . $$
The function $W_4$ is compactly supported   in the first, second and fourth variable and rapidly decaying in the third.  
We first estimate the $m=0$ contribution to be 
\begin{displaymath}
\begin{split}
 \ll \frac{K^{\varepsilon}T^2}{M K^{1/2}}\sum_{r_1\delta_1\delta_3\delta_5\delta_6 r_2 = r} \frac{\delta_1^{3/4} \delta_3  }{(rr_1r_2 \delta_5\delta_6)^{1/4} }  \frac{M}{\sqrt{X}\delta_1\delta_3} X \frac{C}{\delta_1\delta_3\delta_5\delta_6}  \ll K^{\varepsilon} \frac{X^{1/2}CT^2}{K^{1/2} } \ll K^{\varepsilon} \frac{(XNM)^{1/2}}{K^{1/2} }  \ll XK^{\varepsilon}
\end{split}
\end{displaymath}
by \eqref{C}, \eqref{NM}, \eqref{X}, \eqref{T} and a divisor estimate for $n_1n_2$. This is clearly admissible for \eqref{J-bound}, so from now on we assume $m\not= 0$. 
 By the same argument, we obtain for the $m \not= 0$ contribution  the bound 
\begin{equation}\label{E6}
\begin{split}
 \frac{K^{\varepsilon}T^2}{M K^{1/2}}\sum_{r_1\delta_1\delta_3\delta_5\delta_6 r_2 = r} \frac{\delta_1^{3/4} \delta_3  }{(rr_1r_2 \delta_5\delta_6)^{1/4} }  \frac{M}{\sqrt{X}\delta_1\delta_3} X \frac{C}{\delta_1\delta_3\delta_5\delta_6} \frac{(r_2\delta_1 \delta_5\delta_6 )^{1/2}X}{\sqrt{NM   } } \ll K^{\varepsilon} \frac{X^{3/2}CT^2}{(MNK)^{1/2} }.
\end{split}
\end{equation}
By \eqref{C} and \eqref{X}, this is only a factor $K^{3\eta + \varepsilon}$ away from our target \eqref{J-bound}, so a very small additional saving suffices to win. For easier readability we consider only the case $m > 0$, the other case being entirely analogous. 

\subsection{The endgame}\label{133} Up until now we have not touched the long $\Delta$-sum, which we will now use to obtain some additional saving. Before we do this, we must exclude the case that $C$ is very small. To this end we combine \eqref{E6} and \eqref{E4a} to obtain
$$\mathscr{J}^{(1)}_r(X, N, M, C, T)\ll K^{\varepsilon} \frac{T^2C}{(NK)^{1/2}}\min\Big(\frac{X^{3/2}}{M^{1/2}} + X + (XM)^{1/2}\Big) \ll K^{\varepsilon} \frac{T^2C}{(NK)^{1/2}}X.$$
Similarly we can combine \eqref{E6} and \eqref{E2} to obtain
\begin{displaymath}
\begin{split}
\mathscr{J}^{(1)}_r(X, N, M, C, T)& \ll K^{\varepsilon}  \frac{T^2C}{K^{1/2}}\min\Big( \frac{X^{3/2}}{(MN)^{1/2}}, (MN)^{1/4} X^{1/2} + \Big(\frac{N}{M}\Big)^{1/4} X\Big) \\
&\leq  K^{\varepsilon}  \frac{T^2CX}{K^{1/2}} \Big( \frac{1}{(MN)^{1/8}}+  \Big(\frac{N}{M}\Big)^{1/4} \Big)  \ll K^{\varepsilon}  \frac{T^2CX}{K^{1/2}} \Big( \frac{r^{1/8}}{C^{1/4}T^{1/2}}+  \frac{r^{1/4}N^{1/2}}{C^{1/2}T}\Big) 
\end{split}
\end{displaymath}
using \eqref{C}. Combining the previous two bounds we finally obtain 
$$\mathscr{J}^{(1)}_r(X, N, M, C, T) \ll K^{\varepsilon}  \frac{T^2CX}{K^{1/2}} \cdot \frac{r^{1/8}}{C^{1/4}T^{1/2}} \ll  XK^{\frac{1}{4} + \frac{1}{2}\eta + \varepsilon} C^{\frac{3}{4}}$$
by \eqref{T} and \eqref{X} which meets our target \eqref{J-bound} unless
$$C \geq K^{1/3 - 3\eta}$$
which we assume from now on. We recall that $c$ is automatically coprime to $n_1n_2 r_2\delta_3\delta_5\delta_6$. For fixed $k \in \Bbb{N}$ let $\mathscr{S}(k) = \{k_1 x^2 : k_1 \mid k, x \in \Bbb{N}\}$ be the set of square classes of all divisors of $k$. From   \eqref{E5} we remove all $c\in \mathscr{S}(2r_1\delta_1m)$.  Since $C$ is large, the $O(K^{\varepsilon})$ square classes are only a thin subset of all $c$ and  by the same computation as in \eqref{E6}, they contribute no more than
\begin{equation}\label{square-c}
\begin{split}
& 
\ll K^{\varepsilon} \frac{X^{3/2}C^{1/2}T^2}{(MNK)^{1/2} } \ll r^{1/2}XK^{1/2 + 2\eta + \varepsilon} C^{-1/2} \ll XK^{1/3 + 6\eta}
\end{split}
\end{equation}
to$ \mathscr{J}^{(1)}_r(X, N, M, C, T) $ which for $\eta < 1/40$ is admissible. 

With this we return to \eqref{E5} and explain   the idea how to obtain additional savings. Ignoring (for the purpose of these heuristic remarks) the secondary variables $\delta_1, \delta_3, \delta_5, \delta_6, r_2, n_1$, we have to sum
$$\sum_{n_2 \mid |\Delta| + mc} \chi_{\Delta}(n_2m)$$
Writing $|\Delta| + m c = n_2 s$ for $s \in \Bbb{N}$, and assuming also for simplicity that $n_2, m$ are odd, we obtain a sum over
$$\Big( \frac{-sn_2 + mc}{n_2m}\Big) = \Big( \frac{-sn_2 + mc}{n_2 }\Big) \Big( \frac{-sn_2 + mc}{m }\Big) =  \Big( \frac{  mc}{n_2 }\Big) \Big( \frac{-sn_2 }{m }\Big) = \Big( \frac{  m}{n_2 }\Big)\Big( \frac{n_2 }{m }\Big) \cdot \Big( \frac{  c}{n_2 }\Big) \Big( \frac{-s }{m }\Big)  $$
where $-n_2s +mc$ is essentially restricted to squarefree numbers. By quadratic reciprocity, the first two factors are essentially constant. Since $c$ is not a square, the map $n_2 \mapsto (\frac{c}{n_2})$ is a non-trivial character, and since typically the length of $n_2$ is much longer than the length of $c$,  the $n_2$-sum  has some saving from the P\'olya-Vinogradov inequality (here we also need to deal with the squarefree condition). We now make this precise. For clarity, we repeat \eqref{E5} with the small amendments we have made so far:
\begin{equation}\label{E7}
 \begin{split}
 \frac{K^{\varepsilon}T^2}{ (XK)^{1/2}}&\sum_{r_1\delta_1\delta_3\delta_5\delta_6 r_2 = r} \frac{ \ \mu(\delta_5\delta_6)}{(rr_1r_2 \delta_1\delta_5\delta_6)^{1/4} } \underset{\substack{(c, n_1n_2r_2\delta_3\delta_5\delta_6) = 1\\ \delta_1 \mid c^{\infty}, n_1\delta_5 \mid  \Delta^{\infty}, (n_2 \delta_6, \Delta) = 1\\ n_1n_2\delta_5\delta_6 \mid   r_2|\Delta| + mc }}{\sum_{\Delta, n_1, n_2, m} \sum_{ c\not\in \mathscr{S}(2r_1\delta_1m)} } \chi_{\Delta}\Big( \delta_1\delta_3 n_2\delta_6   \frac{  r_2 \bar{c}|\Delta| + m}{n_1\delta_5}\Big)  \\
&  W_5\Big( \frac{|\Delta|}{X}, \frac{n_1n_2 \delta_3\sqrt{N \delta_1\delta_5\delta_6}}{\sqrt{Mr_2}}, \frac{m\sqrt{NM   } }{ (r_2\delta_1 \delta_5\delta_6 )^{1/2}X}, \frac{c\delta_1\delta_3\delta_5\delta_6}{C}\Big) 
 \end{split}
\end{equation}
where $W_5(x, y, z, w) = x^{-1/2} W_4(x, y, z, w)$. 
 We define $s$ through the equation 
  \begin{equation}\label{s}
  n_1n_2 \delta_5\delta_6 s =   r_2|\Delta| + mc.
  \end{equation}
 Note that, up to a negligible error,  
\begin{equation}\label{unbal}
m c \ll K^{\varepsilon}\frac{CX r_2 ^{1/2}}{\sqrt{NM\delta_1\delta_5\delta_6}\delta_3  } \ll K^{\varepsilon} \frac{r_2 X}{T^2}  
\end{equation}
by  \eqref{C}, so that by \eqref{T} we conclude that $mc $ is substantially smaller than $|  r_2\Delta|$. In particular, $n_1n_2 \delta_5\delta_6 s \asymp r_2 X$, so that
\begin{equation}\label{Ss}
s \asymp \frac{ X\delta_3 \sqrt{N \delta_1r_2}}{\sqrt{M\delta_5\delta_6}  }.
\end{equation}
We first argue that we can truncate the $n_1$-sum in \eqref{E5} at $n_1 \leq K^{4\eta}$.  Indeed, since $(n_1, c) = 1$, but $n_1 \mid \Delta^{\infty}$, the   squarefree kernel $\text{rad}(n_1)$ must divide $m$. Summing trivially over $n_1, n_2, s, c, m$ in \eqref{E5} as in \eqref{E6},  the portion $n_1 \geq Y$ contributes at most
\begin{equation}\label{similar}
\begin{split}
&K^{\varepsilon} \frac{T^2}{(XK)^{1/2}} \sum_{r_1\delta_1\delta_3\delta_5\delta_6 r_2 = r} \frac{  1 }{(rr_1r_2 \delta_1\delta_5\delta_6)^{1/4} }  \\
&  \times \sum_{n_1 \geq Y} \frac{\sqrt{Mr_2}}{\delta_3\sqrt{N\delta_1\delta_5\delta_6}n_1}  \frac{ X\delta_3 \sqrt{N \delta_1r_2}}{\sqrt{M\delta_5\delta_6}  } \frac{C}{\delta_1\delta_3\delta_5\delta_6} \frac{X(r_2\delta_1\delta_5\delta_6)^{1/2}}{ \sqrt{NM  } \text{rad}(n_1)} \ll K^{\varepsilon} \frac{X^{3/2} CT^2}{(MNK)^{1/2} Y^{1-\varepsilon}}
\end{split}
\end{equation}
 by applying Rankin's trick and using that $\sum_n \text{rad}(n)^{-1} n^{-\sigma}$ is absolutely convergent for $\sigma > 0$. If $Y \geq K^{4\eta}$, this is $\ll X K^{1/2-2\eta + \varepsilon}$ by \eqref{C} and \eqref{X}, hence admissible for \eqref{J-bound}. Having truncated the $n_1$-sum, we decompose $\Delta = \Delta_1\Delta_2$ into two fundamental discriminants of suitable signs where  $(\Delta_2, 2n_1\delta_5) = 1$ and $\Delta_1 \mid 8n_1 \delta_5 $, in particular $|\Delta_1| \leq 8K^{8\eta}$ and $(\Delta_1, cn_2\delta_6) = 1$. With this notation and recalling \eqref{s} we can write
 $$\chi_{\Delta}\Big( \delta_1\delta_3  n_2\delta_6   \frac{  r_2 \bar{c}|\Delta| + m}{n_1\delta_5}\Big) = \chi_{\Delta_1}(\delta_1\delta_3cs) \chi_{\Delta_2}(\delta_1\delta_3\delta_5\delta_6 n_1n_2m).$$
 Next we make $n_2m$ coprime to $2r_2\Delta_1$ by factoring $n_2 = n_2'n_2''$ and $m=m'm''$ with $(n_2'm', 2r_2\Delta_1) = 1$, $n_2''m'' \mid (2r_2\Delta_1)^{\infty} \mid (2r_2n_1\delta_5)^{\infty}$ so that the previous display equals
 \begin{equation}\label{char}
 \begin{split}
& \chi_{\Delta_1}(\delta_1\delta_3cs) \chi_{\Delta_2}(\delta_1\delta_3\delta_5\delta_6 n_1 n_2''m'') \chi_{ (n_1n_2\delta_5\delta_6 s - mc)(  r_2|\Delta_1|)} (n_2'm')        \\
 = & \chi_{\Delta_1}(\delta_1\delta_3cs) \chi_{\Delta_2}(\delta_1\delta_3\delta_5\delta_6 n_1 n_2''m'') \Big(\frac{-  r_2|\Delta_1|mc}{n_2'}\Big) \Big(\frac{  r_2|\Delta_1|n_1n_2\delta_5\delta_6 s}{m'}\Big)\\
 = & \chi_{\Delta_1}(\delta_1\delta_3cs) \chi_{\Delta_2}(\delta_1\delta_3\delta_5\delta_6 n_1 n_2''m'') \Big(\frac{  r_2|\Delta_1|n_1n''_2\delta_5\delta_6 s}{m'}\Big) \Big(\frac{-  r_2|\Delta_1|m'' }{n_2'}\Big)\Big(\frac{m'}{n_2'}\Big) \Big(\frac{n_2'}{m'}\Big)   \Big(\frac{c}{n_2'}\Big). 
  \end{split}
 \end{equation}
By a computation similar to \eqref{similar}, this time using that $$\sum_{\substack{a \mid b^{\infty}\\ a \leq X}} 1 \ll (bX)^{\varepsilon}$$ for $X \geq 1$, $b \in \Bbb{N}$ (which follows in the same way by Rankin's trick), we can assume $n_2'', m'' \leq K^{4\eta}$, the remaining portion to \eqref{E5} being $\ll XK^{1/2 - 2\eta + \varepsilon}$.  We are left with short (i.e.\ $\ll K^{4\eta}$) variables
\begin{equation}\label{short}
r_1, \delta_1, \delta_3, \delta_5, \delta_6, r_2, n_1, n_2'', m'', \Delta_1
\end{equation}
and potentially long variables 
$$\Delta_2, c, n_2', m', s$$
subject to $c \not\in \mathscr{S}(2r_1\delta_1m'm'')$ as well as 
\begin{displaymath}
\begin{split}
  & (c, n_1n_2'n_2'' r_2 \delta_3\delta_5\delta_6) = 1, \quad \delta_1 \mid c^{\infty} , \quad n_1\delta_5 \mid (\Delta_1\Delta_2)^{\infty}, \quad (n_2'n_2'' \delta_6, \Delta_1\Delta_2) = 1,\quad \Delta_1 \mid 8n_1\delta_5,\\&(\Delta_2, 2n_1\delta_5) = 1,  \quad (n_2'm', 2r_2\Delta_1) = 1, \quad n_2''m'' \mid(2r_2\Delta_1)^{\infty}, \quad n_1n_2'n_2'' \delta_5\delta_6s =   r_2 \Delta_1\Delta_2 + mc.
\end{split}
\end{displaymath} 
 We can eliminate $\Delta_2$ from the last equation, so that a congruence 
 $$  n_1n_2'n_2'' \delta_5\delta_6s =  mc \, (\text{mod } r_2\Delta_1)$$
 remains. Then the conditions $(\Delta_2, 2n_1\delta_5) = (n_2'n_2''\delta_6, \Delta_2) = 1$ are re-phrased as
 $$(n_1n_2'n_2'' \delta_5\delta_6s - mc , 2r_2n_1\delta_5n_2' n_2'' \delta_6) = r_2$$
 which is equivalent to
 $$(n_1n_2'n_2'' \delta_5\delta_6, mc) = r_2', \quad (2r_2, n_1n_2'n_2'' \delta_5\delta_6s - mc) = r_2'', \quad r_2'r_2'' = r_2.$$
 The condition  $n_1\delta_5 \mid (\Delta_1\Delta_2)^{\infty}$ reads
 $$\text{rad}(n_1\delta_5) \mid \frac{n_1n_2'n_2'' \delta_5\delta_6s - mc}{r_2}.$$
 All of these conditions on $n_2'$ can be detected by congruences modulo ``short'' variables in \eqref{short} (and some powers of 2) as well as $(n_2', m) = 1$. Finally we need to remember that $\Delta_1\Delta_2$ is a fundamental discriminant. To this end we split into residue classes $\Delta_1\Delta_2 \equiv 1, 5, 8, 9, 12, 13$ (mod 16) and insert a factor $\mu^2((n_1n_2'n_2'' \delta_5\delta_6s - mc)/(\alpha r_2))$ with $\alpha \in \{1, 4\}$. We use the convolution formula
 $$\mu^2\Big(\frac{n_1n_2'n_2'' \delta_5\delta_6s - mc}{\alpha r_2}\Big)= \sum_{y^2 \mid \frac{n_1n_2'n_2'' \delta_5\delta_6s - mc}{\alpha r_2}} \mu(y),$$
and insert all of this back into \eqref{E7}. We claim that we can restrict $y \leq K^{A\eta}$ for some constant $A$. Indeed, summing over all short variables, as well as $ c, m'$, we get a congruence for $n_2's$ modulo $y^2$, so that the portion $y > Y$ contributes at most
$$K^{O(\eta)}  \sum_{y \geq Y}\frac{T^2}{(XK)^{1/2}} C \frac{X}{\sqrt{NM}} \frac{X}{y^2} \ll \frac{XK^{1/2+ O(\eta)}}{Y}$$
which is acceptable if $Y = K^{A\eta}$ for $A$ sufficiently large. Note that \eqref{unbal} implies that  $n_1n_2'n_2'' \delta_5\delta_6s - mc $ never vanishes, so there is no ``$1+$'' in the congruence count.  In addition, for fixed $y$ we also remove all $c \in \mathscr{S}(y^2)$ at the cost of an error $XK^{1/3+ O(\eta)}$ as in \eqref{square-c}.

We are finally ready to return to \eqref{char} and split the sum over $n_2'$ into residue classes modulo $\nu$ modulo $H := 32 r_1\delta_1\delta_3\delta_5\delta_6r_2n_1n_2''m''\Delta_1 y^2 = K^{O(\eta)}$. By assumption, $c$ is not in a square class of any divisor of $H$. 
Thus we consider 
$$\sum_{\substack{n'_2 \equiv \nu \, (\text{mod } H)\\ (n'_2, m') = 1}}\Big(\frac{c}{n_2'}\Big)W\Big(\frac{n'_2}{R}\Big)$$
for $\nu$ odd, $c \ll C$  and $R  \ll \sqrt{M/N} K^{O(\eta)}$. We can detect the congruence condition by characters, none of which conspires with $n_2' \mapsto (\frac{c}{n_2'})$ to become the trivial character. By the P\'olya-Vinogradov inequality, we can bound the previous display by  $C^{1/2} K^{O(\eta)}$, and by trivial estimates over $c, m', s$ and the present estimate for the sum over $n_2'$ we obtain the final bound
$$K^{O(\eta)} \frac{T^2}{(XK)^{1/2}} C \frac{X}{\sqrt{NM}}\frac{X\sqrt{N}}{\sqrt{M}} C^{1/2} \ll K^{O(\eta)}  \frac{N^{3/4} X^{3/2}}{K^{1/2}TM^{1/4}} \ll K^{O(\eta)}\frac{X^{3/2}}{K^{1/2}T^{1/4}} \ll XK^{3/8+ O(\eta)}$$
by \eqref{C}, \eqref{NM}, \eqref{X} and \eqref{T}. For sufficiently small $\eta$ this is in agreement with \eqref{J-bound} and completes the analysis of the first  off-diagonal term in \eqref{diagdiag}. 

\subsection{The second off-diagonal term}\label{144}  The analysis of the second off-diagonal term in \eqref{kuz-even-form} is very similar, so we can be brief. Here we need to consider
\begin{displaymath}
\begin{split}
\mathcal{I}^{\text{off-diag,2}}(\Delta, t,k,  r):= \frac{16}{\pi} \sum_{d \mid r} &\sum_{{\tt n},{\tt m}} \frac{ \chi_{\Delta}({\tt m}) }{(d{\tt n}{\tt m})^{1/2}}  \sum_c \frac{S({\tt m}, {\tt n}r/d, c)}{c} \\
&\int_{-\infty}^{\infty}   K_{2i\tau}(4\pi \sqrt{({\tt m}{\tt n}r/d)}/c) \sinh(\pi \tau) W^+_{\tau}(d{\tt n}) W^-_{\tau}\Big(\frac{{\tt m}}{|\Delta|}\Big)h(\tau) \tau  \frac{d\tau}{4\pi}
\end{split}
\end{displaymath}
with $h$ as in \eqref{htau}. 
This needs to be inserted into \eqref{weobt} with $r = f^2d\nu/(d_1 d_2 d_3^2) \leq K^{4\eta}$. Under the same size restrictions \eqref{X} as before we want to show that
$$\mathcal{J}^{(2)}(X, t,k,  r) := \sum_{\Delta}w\Big( \frac{|\Delta|}{X}\Big) \mathcal{I}^{\text{off-diag,2}}(\Delta, t,k,  r) \ll XK^{1/2 - \eta}.$$
As before we first use holomorphicity ot ensure absolute convergence of the $c$-sum and obtain a very coarse truncation. Then we apply smooth partitions of unity restricting to $\tau \asymp T$ satisfying \eqref{T}, $n \asymp N$, $m \asymp M$ satisfying \eqref{NM} and $c\asymp C$. This time we use Lemma \ref{bessel-kuz}b to conclude that
\begin{equation}\label{C1}
C \leq K^{\varepsilon} \frac{\sqrt{NM}}{T},
\end{equation}
and by an analogue of \eqref{E1}  we need to bound the quantity
\begin{equation*}
\begin{split}
\mathscr{J}^{(2)}_r(X, N, M, C, T) = &\frac{K^{\varepsilon}T}{(NMK)^{1/2}C}\sum_{r_1r_2=r} \frac{1}{(rr_1)^{1/4}} \sum_{\Delta, {\tt n} , {\tt m}, c}  \chi_{\Delta}({\tt m}) S({\tt m}, {\tt n} r_2, c) W\Big( \frac{|\Delta|}{X}, \frac{{\tt n}}{N}, \frac{{\tt m}}{M}, \frac{c}{C}\Big)
\end{split}
\end{equation*}
where $W$ satisfies the same properties. The trivial bound using the large sieve (Lemma \ref{HBlemma}) is now
\begin{equation}\label{F1}
  \mathscr{J}^{(2)}_r(X, N, M, C, T) \ll \frac{K^{\varepsilon}T}{(NMK)^{1/2}C} XNMC^{3/2}\Big( \frac{1}{X^{1/2}} + \frac{1}{M^{1/2}}\Big)=    \frac{K^{\varepsilon}T}{ K^{1/2}} (XNC)^{1/2} (  X^{1/2} +  M^{1/2} ).
  \end{equation}
Next we apply Poisson summation in ${\tt n}$ in residue classes modulo $c$ which is simpler than before because there is no exponential $e(\pm 2\sqrt{{\tt n}{\tt m}r_2}/c)$. This transforms $ \mathscr{J}^{(2)}_r(X, N, M, C, T)$ into 
\begin{displaymath}
\begin{split}    \frac{K^{\varepsilon}T}{(NMK)^{1/2}C}\sum_{r_1r_2=r} \frac{1}{(rr_1)^{1/4}} \sum_{\Delta,  {\tt m}, c}  \chi_{\Delta}({\tt m}) N \sum_{n \in \Bbb{Z}}   \underset{\substack{\gamma \, (\text{mod }c)\\ r_2\bar{\gamma} \equiv -n \, (\text{mod } c)}}{\left.\sum\right.^{\ast}} e\Big( \frac{{\tt m} \gamma  }{c}\Big)   W_1\Big( \frac{|\Delta|}{X}, \frac{nN}{C}, \frac{{\tt m}}{M}, \frac{c}{C}\Big)
\end{split}
\end{displaymath}
for a weight function $W_1$ that is compactly supported in the first, third and fourth variable and rapidly decaying in the second. This term contains the same character sum as in \eqref{char1}. By the same manipulation we obtain
\begin{equation}\label{F2}
\begin{split}    \mathscr{J}^{(2)}_r(X, N, M, C, T) = & \frac{K^{\varepsilon}TN^{1/2}}{(MK)^{1/2}C}\sum_{r_1\delta_1\delta_3\delta_4 r_2 = r} \frac{(\delta_1 \delta_3)^{3/4}\mu(\delta_4)}{(rr_1r_2 \delta_4)^{1/4} } \\
&\underset{\substack{(c, nr_2\delta_3\delta_4) = 1\\ \delta_1 \mid c^{\infty}}}{\sum_{\Delta,  {\tt m}, c} \sum_{n\in \Bbb{Z}}} \chi_{\Delta}({\tt m}\delta_1\delta_3)   e\Big( -\frac{{\tt m}r_2 \overline{n\delta_2}}{c}\Big)W_1\Big( \frac{|\Delta|}{X}, \frac{n\delta_1\delta_2 N}{C}, \frac{{\tt m}\delta_1}{M}, \frac{c\delta_1\delta_2}{C}\Big).
\end{split}
\end{equation}
This is arithmetically analogous to \eqref{E4} except that the roles of $\delta_2n$ and $c$ are reversed in the exponential. This makes good sense since in generic ranges we have $c \asymp K^{1/2}$, $ n \asymp K$ in \eqref{E4}, but $c \asymp K$, $n \asymp K^{1/2}$ in \eqref{F2}. The large sieve now gives the bound
\begin{equation}\label{F3}
  \mathscr{J}^{(2)}_r(X, N, M, C, T) \ll \frac{K^{\varepsilon}TN^{1/2}}{(MK)^{1/2}C} X\frac{C}{N}MC\Big(\frac{1}{\sqrt{X}} + \frac{1}{\sqrt{M}}\Big) =\frac{K^{\varepsilon}TX^{1/2} C }{(NK)^{1/2}}  ( \sqrt{X} +\sqrt{M} ). 
\end{equation}
Next we apply Poisson summation in ${\tt m}$ in residue classes modulo $|\Delta| c$. In the character sum    
$$\sum_{\mu\, (\text{mod }c|\Delta|)} \chi_{\Delta}(\mu) e\Big(\frac{\mu(m- r_2\overline{n\delta_2}|\Delta|)}{c|\Delta|}\Big)$$
we decompose $c = c_1 c_2$ where $c_1 \mid \Delta^{\infty}$, $(c_2, \Delta) = 1$. The character sum vanishes unless $c_1c_2 \mid mn\delta_2 - r_2 |\Delta|$ in which case it equals
$$i \sqrt{|\Delta|} c_1 \chi_{\Delta}\Big(c_2 \frac{m- r_2 \overline{n\delta_2}|\Delta|}{c_1}\Big).$$
Estimating trivially at this point yields
\begin{equation}\label{F4}  
  \mathscr{J}^{(2)}_r(X, N, M, C, T)  \ll K^{\varepsilon}\frac{T(MN)^{1/2}C^2 X^2 }{C(XK)^{1/2}MN} = \frac{X^{3/2}CT}{(KNM)^{1/2}} \ll X K^{1/2 +2 \eta + \varepsilon}
\end{equation}
 by \eqref{C1} and \eqref{X}, matching the bound in \eqref{E6}.  With the roles of $c$ and $n$ reversed, we now want to make sure that $n \ll  C(N\delta_1\delta_2)^{-1}$ is large enough.  By the trivial bound \eqref{F4} we can assume that $n \geq C(N\delta_1\delta_2)^{-1} K^{-4\eta}$. Now combining \eqref{F4} and \eqref{F1} we obtain
 $$  \mathscr{J}^{(2)}_r(X, N, M, C, T)  \ll K^{\varepsilon} \frac{(CX)^{1/2} T }{K^{1/2} }\min\Big(\frac{XC^{1/2}}{(NM)^{1/2}} , (NX)^{1/2} + (NM)^{1/2}\Big) \leq K^{\varepsilon} \frac{TXC^{1/2}}{K^{1/2}}(N^{1/2} + C^{1/4}).$$
Combining \eqref{F4} and \eqref{F3} we obtain
$$  \mathscr{J}^{(2)}_r(X, N, M, C, T) \ll K^{\varepsilon} \frac{CTX^{1/2}}{(MNK)^{1/2}}\min(\sqrt{MX} + M, X)\ll K^{\varepsilon} \frac{CTX }{(NK)^{1/2}} . $$
Combining the previous two bounds, we obtain
$$   \mathscr{J}^{(2)}_r(X, N, M, C, T) \ll K^{\varepsilon}\frac{TXC^{3/4}}{K^{1/2}} =  K^{\varepsilon}\frac{TXN^{3/4}}{K^{1/2}} \Big( \frac{C}{N}\Big)^{3/4} \ll XK^{1/4+\varepsilon} \Big(\frac{C}{N}\Big)^{3/4} $$
by \eqref{NM} and \eqref{T}. This is acceptable unless 
$$C/N \geq K^{1/3 - 2\eta} \quad \text{and} \quad C \geq K^{2/3 - 2\eta}$$
which we assume from now on, so that in particular $n \gg  C(N\delta_1\delta_2)K^{-4\eta} \gg N^{1/3 -10\eta}$. By the same argument as in the previous subsection we can now extract certain square classes in the $n$-sum and then save from the P\'olya-Vinogradov inequality. In effect, we replace the factor $C$ from a trivial bound of the $c_2$-sum by a factor $ K^{O(\eta)}\sqrt{N/C}$ of the square root of the conductor of $c_2 \mapsto (\frac{n}{c_2})$. This leads to the final bound
$$   \mathscr{J}^{(2)}_r(X, N, M, C, T) \ll K^{O(\eta)} \frac{X^{3/2} T \sqrt{N/C}}{(KNM)^{1/2}} \leq K^{O(\eta)} \frac{X^{3/2} T  }{K^{1/2}C^{1/2}} \ll XK^{1/6 +O(\eta)}$$
by \eqref{X}, \eqref{T} and our assumption  $C \geq K^{2/3 - 2\eta}$. This is in agreement with \eqref{J-bound} and completes the analysis of the the second diagonal off-diagonal term, hence the analysis of the complete diagonal term.

\section{The off-diagonal term}\label{off-off}

\subsection{Initial steps}
We return to \eqref{slightly-simp} and analyze the off-diagonal term in Lemma \ref{lem1} applied to the $h$-sum. Here we are only interested in upper bounds, so dropping all numerical constants  it suffices to estimate 
\begin{displaymath} 
\begin{split}
    \frac{1}{  K^2} &  \sum_{k \in 2\Bbb{N}} W\left(\frac{k}{K}\right)   \sum_{\substack{(n, m) = 1\\ n, m \leq K^{\eta}}} \frac{  \mu(n)\mu^2(m)}{n^{3/2}m^{3}}  \int_{-\infty}^{\infty}  \int^{\ast}_{\Lambda_{\text{\rm ev}}} \omega(t_{\tt u})    \sum_{f_1, f_2, D_1, D_2}  \sum_{\substack{d_1 \mid d_2 \mid n\\ (d_1d_2)^2 \mid n^2D_2}} \left(\frac{d_1}{d_2} \right)^{1/2} \chi_{  D_2}\Big(\frac{d_2}{d_1}\Big) \\
  &\frac{ P(D_1; {\tt u})\overline{P(D_2; {\tt u})}}{f_1f_2 |D_1D_2|^{3/4}} \Big(\frac{|D_2|f_2^2}{|D_1|f_1^2}\Big)^{it} i^k\sum_c  \frac{K^+(|D_1|, |D_2|n^2/(d_1d_2)^2, c)}{c} J_{k-3/2}\Big( \frac{ 4\pi \sqrt{|D_1D_2|}n}{cd_1d_2}\Big) \\
 &   V_{t}(|D_1D_2|(f_1f_2)^2; k, t_{\tt u})  d{\tt u} \, dt,
 \end{split}
\end{displaymath}
 and our target bound is $K^{-\eta}$. We recall that $V_t(x, k, \tau)$ was defined in \eqref{defV3} and besides the decay properties it is important to note that $V_t(x, k, \tau)$ is holomorphic in, say,  $|\Im \tau| \leq 1$. Since we want to apply the trace formula (Theorem \ref{thm5}) to the spectral sum  later, we must not destroy holomorphicity in the third variable.

 As in Section \ref{sec12} we write   $d_2 = d_1\delta$, $n = d_1\delta \nu$. Then $d_1^2 \mid \nu^2 D_2$ and $d_1^2 \mid D_2$ since $n$ is squarefree. Again we  write $d_1 = d$ and $D_2 d^2$ in place of $D_2$ and bound the preceding display as
 \begin{displaymath} 
\begin{split}
    \frac{1}{  K^2} &   \sum_{ d\delta \nu, m \leq K^{\eta}} \frac{  \mu^2(d\delta \nu m) }{d^3 \delta^2 \nu^{3/2} m^3 }\Big| \sum_{k \in 2\Bbb{N}} W\left(\frac{k}{K}\right)  \int_{-\infty}^{\infty}  \int^{\ast}_{\Lambda_{\text{\rm ev}}} \omega(t_{\tt u})    \sum_{f_1, f_2, D_1, D_2}      \frac{\chi_{D_2}(\delta) P(D_1; {\tt u})\overline{P(D_2d^2; {\tt u})}}{f_1f_2 |D_1D_2|^{3/4}}  \\
 &  \Big(\frac{|D_2|(df_2)^2}{|D_1| f_1^2}\Big)^{it}  i^k\sum_c  \frac{K_{3/2}^+(|D_1|, |D_2|\nu^2 , c)}{c} J_{k-3/2}\Big( \frac{ 4\pi \sqrt{|D_1D_2|}\nu}{c }\Big)  V_{t}(|D_1D_2|( d f_1f_2)^2; k, t_{\tt u})  d{\tt u} \, dt\Big|. 
 \end{split}
\end{displaymath}
 We sum over $k$ using Lemma \ref{lem2} and open the Kloosterman sum. As in Section \ref{112}, up to a negligible error we obtain the upper bound
 \begin{displaymath} 
\begin{split}
    &   \sum_{ d\delta \nu, m \leq K^{\eta}}  \sum_{4 \mid c}  \sum_{f_1, f_2} \frac{  \mu^2(d\delta \nu m) }{d^3 \delta^2 \nu^{3/2} m^3 cf_1f_2 }\underset{\substack{\gamma\, (\text{mod } c)\\ (\gamma, c) = 1}}{\max} | \mathcal{I}^{\text{off}}(K)|
     \end{split}
\end{displaymath}
    where $\mathcal{I}^{\text{off}}(K) = \mathcal{I}_{d, \delta, \nu, m, c, f_1, f_2, \gamma}^{\text{off}}(K)$ is given by
    \begin{displaymath} 
\begin{split}
\mathcal{I}^{\text{off}}(K)
 =   &    \frac{1}{K^2}  \int_{-\infty}^{\infty}  \int^{\ast}_{\Lambda_{\text{\rm ev}}} \omega(t_{\tt u})    \sum_{  D_1, D_2}      \frac{ \chi_{D_2}(\delta)P(D_1; {\tt u})\overline{P(D_2d^2; {\tt u})}}{  |D_1D_2|^{3/4}} \Big(\frac{|D_2|(df_2)^2}{|D_1| f_1^2}\Big)^{it}  \\
 &    e\Big( \frac{|D_1|\gamma + |D_2|\nu^2\bar{\gamma}}{c}\Big) e\Big(\pm \frac{ 2 \sqrt{|D_1D_2|}\nu}{c }\Big)  \tilde{V}\Big(|D_1D_2|(df_1f_2)^2,   \frac{\sqrt{|D_1D_2|}\nu}{c}, t, t_{\tt u}\Big)   d{\tt u} \, dt; 
 \end{split}
\end{displaymath}
here $\tilde{V}$  satisfies  \eqref{tildeV} and is holomorphic in $|\Im t_{\tt u}| \leq 1$.  The bounds contained in \eqref{tildeV}  imply in particular $c, f_1, f_2 \leq K^{\eta}$ up to a negligible error. For notation simplicity we consider only the plus-case, the minus case being entirely similar. 
 
 As in Section \ref{sec12} we start with a Voronoi step, but in the present situation  (since we have already excluded the constant function that requires a very careful treatment) we can afford to lose small powers of $K$ on the way. We introduce the notation
 $$A \preccurlyeq B \quad :\Longleftrightarrow \quad A \ll K^{c\eta} B$$
 for some constant $c$, not necessarily the same on every occasion. We always assume that $\eta$ is sufficiently small. 
 
  To begin with, we integrate over $t$  which by the properties of \eqref{tildeV} induces the condition $|D_1| f_1^2 - |D_2|(df_2)^2 \preccurlyeq K^{1/2}$ up to a negligible error. This now implies $K^2 \preccurlyeq D_1, D_2  \preccurlyeq  K^{2  }$, up to a negligible error.   In $\tilde{V}$ we can separate the variables $D_1, D_2$ from $t_{\tt u}$ by Mellin inversion (keeping holomorphicity in $t_{\tt u}$). Since $\tilde{V}$ is of size $K^{-1/2}$ and we also get a factor $K^{1/2}$ from the $t$-integration,  we are left with bounding
   \begin{displaymath} 
\begin{split}
\tilde{\mathcal{I}}^{\text{off}}(K)
 =   &    \frac{1}{K^2}     \int^{\ast}_{\Lambda_{\text{\rm ev}}} \Omega(t_{\tt u})    \sum_{ D_1, D_2}      \frac{ \chi_{D_2}(\delta)P(D_1; {\tt u})\overline{P(D_2d^2; {\tt u})}}{  |D_1D_2|^{3/4}} V_1\Big(\frac{|D_1|}{K^2}\Big)  V_2\Big(\frac{|D_2|}{K^2}\Big)  \\
 &  V_3\Big(K^{1/2} \log\frac{|D_2|(df_2)^2}{|D_1| f_1^2}\Big)   e\Big( \frac{|D_1|\gamma + |D_2|\nu^2\bar{\gamma}}{c}\Big) e\Big(  \frac{ 2 \sqrt{|D_1D_2|}\nu}{c }\Big)     d{\tt u}  
 \end{split}
\end{displaymath}
where $V_1, V_2$ have support in $[K^{-O(\eta)}, K^{O(\eta)}]$ with Sobolev norms bounded by $ \preccurlyeq 1$, $V_3$ is rapidly decaying, and $\Omega(\tau)$ is holomorphic in $|\Im \tau| \leq 1$, satisfies the conditions \eqref{prop-omega} and is non-negligible only in the range $K^{1/2} \preccurlyeq |\tau| \preccurlyeq K^{1/2}$.  
 
We now consider the $D_1$-sum
\begin{equation}\label{D1-sum}
  \sum_{ D_1}      \frac{ P(D_1; {\tt u})\overline{P(D_2d^2; {\tt u})}}{  |D_1D_2|^{3/4}} V_1\Big(\frac{|D_1|}{K^2}\Big)   V_3\Big(K^{1/2} \log\frac{|D_2|(df_2)^2}{|D_1| f_1^2}\Big)     e\Big( \frac{|D_1|\gamma  }{c}\Big) e\Big(  \frac{ 2 \sqrt{|D_1D_2|}\nu}{c }\Big)  
  \end{equation}
and insert \eqref{mixed} with $t = t_{\tt u}/2$ if ${\tt u}$ is cuspidal and \eqref{eisen1} if ${\tt u} = E(., 1/2 + it_{\tt u})$ is Eisenstein. For clarity we recall that
\begin{displaymath}
\begin{split}
 \frac{ P(D_1; {\tt u}) \overline{P(D_2d^2; {\tt u})}}{|D_1D_2 |^{3/4}} = \frac{d^{1/2}}{|D_1|^{1/2}|D_2|^{3/4}} \begin{cases} b(D_1) \sqrt{|D_1|} A(D_2d^2, {\tt u}), & {\tt u} \text{ cuspidal},\\
 L(D_1, 1/2 + it_{\tt u}) |D_1|^{it_{\tt u}/2} A(D_2, {\tt u}), & {\tt u} = E(., 1/2 + it_{\tt u}),\end{cases}
\end{split}
\end{displaymath}
where
$$A(D, {\tt u}) = \begin{cases} \frac{3}{\pi} L({\tt u}, 1/2) \Gamma(\frac{1}{4} + \frac{it_{\tt u}}{2})\Gamma(\frac{1}{4} - \frac{it_{\tt u}}{2})b(D) \sqrt{|D|}, & {\tt u} \text{ cuspidal},\\
\frac{\zeta(1/2 + it_{\tt u})\zeta(1/2 - it_{\tt u}) L(D, 1/2 - it_{\tt u}) |D|^{-it_{\tt u}/2}}{2|\zeta(1 + 2 it_{\tt u})|^2}, & {\tt u} = E(., 1/2 + it_{\tt u}).\end{cases}$$
Even though the normalization in the cuspidal and the Eisenstein case is quite different, the Voronoi formulae in Lemma \ref{Vor} (with $r = t_{\tt u}/2)$ and \ref{Vor-eis} (with $t = t_{\tt u}/2$) can deal in the same fashion with the sums
\begin{equation}\label{d1}
\sum_{D_1 < 0} \left\{ \begin{array}{l}  b(D_1) \sqrt{|D_1|} \\  L(D_1, 1/2 + it_{\tt u}) |D_1|^{it_{\tt u}/2}\end{array} \right\} e\Big( \frac{|D_1|\gamma}{c}\Big) \phi(D_1)
\end{equation}
with
$$\phi(x) =   \frac{1}{|x|^{1/2}}V_1\Big(\frac{|x|}{K^2}\Big)   V_3\Big(K^{1/2} \log\frac{|D_2|(df_2)^2}{|x| f_1^2}\Big) e\Big(  \frac{ 2 \sqrt{|xD_2|}\nu}{c }\Big)   $$
for $x < 0$ and $\phi(x) = 0$ for $x > 0$. It is easy to see  that with this choice of $\phi$ and for $t  \ll K^{1/2 + \varepsilon}$ the two polar terms in Lemma \ref{Vor-eis} are negligible, due to the strong oscillatory behaviour of the exponential $e(\pm 2\sqrt{|xD_2|}\nu/c)$ for $K^2 \preccurlyeq x, D_2 \preccurlyeq   K^{2 }$, $c \preccurlyeq   1$. By the same argument as in Section \ref{104}, see in particular \eqref{BesselK}, the  terms with $D > 0$ on the right hand side of the Voronoi summation formula are negligible. Hence, up to a negligible error, \eqref{d1} becomes
\begin{equation*}
\begin{split}
\Big( \frac{-c}{\gamma}\Big)&\epsilon_{\gamma} e(1/8) \frac{2\pi}{c}\sum_{D_1 < 0} \left\{ \begin{array}{l}  b(D_1) \sqrt{|D_1|} \\  L(D_1, 1/2 + it_{\tt u}) |D_1|^{it_{\tt u}/2}\end{array} \right\} e\Big(- \frac{|D_1|\bar{\gamma}}{c}\Big)  \int_0^{\infty}V_1\Big(\frac{|x|}{K^2}\Big)   \\
&  \sum_{\pm} (\mp)  \frac{\cos(\pi/4 \pm \pi i t_u/2)}{\sin(\pi i t_u)} J_{\pm i t_u}\Big( \frac{4\pi \sqrt{|D_1x|}}{c}\Big)  V_3\Big(K^{1/2} \log\frac{|D_2|(df_2)^2}{|x| f_1^2}\Big)    e\Big(  \frac{2\sqrt{|xD_2|}\nu}{c}\Big) \frac{dx}{ x^{1/2   }}. 
   \end{split}
\end{equation*}
Using \eqref{bessel-approx} we approximate the Bessel $J$-function for $t_{\tt u} \ll K^{1/2 + \varepsilon}$ by an exponential and replace up to a negligible error the preceding display by
\begin{equation*}
\begin{split}
\Big( \frac{-c}{\gamma}\Big)&\epsilon_{\gamma} e(1/8) \frac{2\pi}{c} \sum_{D_1 < 0} \left\{ \begin{array}{l}  b(D_1) \sqrt{|D_1|} \\  L(D_1, 1/2 + it_{\tt u}) |D_1|^{it_{\tt u}/2}\end{array} \right\} e\Big( -\frac{|D_1|\bar{\gamma}}{c}\Big)   \sum_{\pm} \int_0^{\infty}  \frac{c^{1/2}}{|D_1|^{1/4}} V_1\Big(\frac{|x|}{K^2}\Big)  \\
& f^{\pm}\Big( \frac{2\sqrt{|D_1x|}}{c}, t_{\tt u}\Big) V_3\Big(K^{1/2} \log\frac{|D_2|(df_2)^2}{|x| f_1^2}\Big)    e\Big(  \frac{2\sqrt{|xD_2|}\nu \pm 2 \sqrt{|D_1x|}}{c}\Big) \frac{dx}{ x^{3/4   }}
   \end{split}
\end{equation*}
with 
$$f^{\pm}(x, \tau) =  \frac{ 1}{2\pi }\sum_{k=0}^{n-1} \frac{i^k(\pm 1)^k}{(4\pi x)^{k}} \frac{\Gamma(i\tau+ k + 1/2)}{k! \Gamma(i\tau - k + 1/2)}   
 $$
for $n$ sufficiently large, but fixed.  In particular  $f^{\pm}$ satisfies   \eqref{besseldecay}, and we see from the asymptotic expansion \eqref{bessel-approx} that $f^{\pm}$ is holomorphic in the second variable in, say, $|\Im \tau| < 1$. For notational simplicity we consider only the leading term $k=0$, the lower order terms being completely analogous (but easier).

We restore the periods $P(D; {\tt u})$, so that the leading term of \eqref{D1-sum} has the shape 
\begin{equation*}
\begin{split}
  \sum_{\pm} \Big( \frac{-c}{\gamma}\Big)&\epsilon_{\gamma} e(1/8) \frac{2\pi }{c^{1/2}}\sum_{D_1 < 0}  \frac{P(D_1; {\tt u}) \overline{P(D_2d^2; {\tt u})}}{|D_2|^{3/4}|D_1|^{1/2}}  e\Big(- \frac{|D_1|\bar{\gamma}}{c}\Big) \frac{1}{2\pi}\int_0^{\infty}  V_1\Big(\frac{|x|}{K^2}\Big)    \\
&  V_3\Big(K^{1/2} \log\frac{|D_2|(df_2)^2}{|x| f_1^2}\Big)    e\Big(  \frac{2\sqrt{|xD_2|}\nu \pm 2 \sqrt{|D_1x|}}{c}\Big) \frac{dx}{ x^{3/4   }}
\end{split}
 \end{equation*}
 with lower order terms of similar form. Recall that $K^2 \preccurlyeq D_2 \preccurlyeq   K^{2 }$. 
Integrating by parts, we see as in \eqref{add1} -- \eqref{add2} that only the minus-term in the exponential is relevant (the plus-term is negligible) and the $x$-integral restricts to $\sqrt{|D_2|}\nu - \sqrt{|D_1|} \preccurlyeq K^{-1/2}$. We therefore introduce a new variable $h \in \Bbb{Z}$ by 
$$|D_1| = |D_2|\nu^2 + h$$
with $h \preccurlyeq K^{1/2 }$. Changing variables in the $x$-integral,   we obtain
\begin{equation*}
\begin{split}
&    \Big( \frac{-c}{\gamma}\Big)\epsilon_{\gamma} e(1/8) \frac{1 }{c^{1/2}}\sum_{h \preccurlyeq K^{1/2}}  \frac{P(-|D_2|\nu^2 - h, {\tt u}) \overline{P(D_2d^2; {\tt u})} }{|D_2|^{1/2}(|D_2|\nu^2 + h)^{1/2}}  e\Big( -\frac{(|D_2|\nu^2 + h)\bar{\gamma}}{c}\Big)  \\
  &\int_0^{\infty} V_1\Big(\frac{|xD_2|}{K^2}\Big)    V_3\Big(K^{1/2} \log\frac{ (df_2)^2}{|x| f_1^2}\Big)    e\Big(  \frac{2 \sqrt{x D_2}(\sqrt{|D_2|}\nu-  \sqrt{(|D_2|\nu^2 + h)})}{c}\Big) \frac{dx}{ x^{3/4   }}
\end{split}
 \end{equation*}
where $h$ is restricted to numbers such that $-(|D_2|\nu^2 + h)$ is a negative discriminant. The weight function $V_3$ forces $x \preccurlyeq 1$ and more precisely 
\begin{equation}\label{x}
x - (df_2/f_1)^2  \preccurlyeq K^{-1/2}.
\end{equation}  By a Taylor expansion we can write
$$ e\Big(  \frac{2 \sqrt{x|D_2|}(\sqrt{|D_2|}\nu-  \sqrt{(|D_2|\nu^2 + h)})}{c}\Big) = e\Big( -\frac{\sqrt{x}h}{c\nu}\Big) F(D_2)$$
where  $$F(D) = F_{x,  h, \nu, c}(D) = 1 + \frac{i\pi \sqrt{x} h^2}{2 |D|  \nu^3 c} - \frac{i\pi \sqrt{x}h^3}{4|D|^2 \nu^5 c} +\ldots $$ 
Again we only keep the leading term (the lower order terms being similar, but easier). We substitute all of this back  into $\tilde{\mathcal{I}}^{\text{off}}(K)$ and pull the $x$-integration outside which is subject to \eqref{x}. In this way we see that it suffices to bound  
\begin{equation}\label{I1}
\begin{split}
 \mathcal{I}_1^{\text{off}}(K)= \frac{1}{K^{5/2}}     \int^{\ast}_{\Lambda_{\text{\rm ev}}} \Omega(t_{\tt u})   & \sum_{  D_2}        \sum_{h \preccurlyeq K^{1/2}}  \frac{P(-|D_2|\nu^2 - h, {\tt u}) \overline{P(D_2d^2; {\tt u})}}{|D_2|^{1/2} (|D_2|\nu^2 + h)^{1/2}}\chi_{D_2}(\delta) \\
 &   V_x\Big(\frac{|D_2|}{K^2}\Big) e\Big( -\frac{  h\bar{\gamma}}{c} -\frac{\sqrt{x}h}{c\nu}\Big)d{\tt u}
 \end{split}
 \end{equation}
where $V_x(z) = V_2(z) V_1(xz)$, uniformly in
$$x, \nu, c, \delta, d \preccurlyeq 1, \quad (\gamma, c) = 1.$$
A trivial bound using Cauchy-Schwarz and Proposition \ref{Lfunc} gives $ \mathcal{I}_1^{\text{off}}(K) \preccurlyeq 1$, as in Section \ref{weakversion}. In order to make progress and get additional savings, we must now treat the ${\tt u}$-integral non-trivially. This is where the trace formula has its appearance. 
 
\subsection{Application of the trace formula} We now apply Theorem \ref{thm5} to the spectral expression
\begin{displaymath}
\begin{split}
& \int^{\ast}_{\Lambda_{\text{\rm ev}}} \Omega(t_{\tt u})      \frac{P(-|D_2|\nu^2 - h, {\tt u}) \overline{P(D_2d^2; {\tt u})}}{|D_2|^{1/4} (|D_2|\nu^2 + h)^{1/4}} d{\tt u} \\
 &=  \int_{\Lambda_{\text{\rm ev}}} \Omega(t_{\tt u})      \frac{P(-|D_2|\nu^2 - h, {\tt u}) \overline{P(D_2d^2; {\tt u})}}{|D_2|^{1/4} (|D_2|\nu^2 + h)^{1/4}} d{\tt u} - \frac{3}{\pi} \frac{H(D_2d^2)H(-|D_2|\nu^2 - h)}{|D_2|^{1/4} (|D_2|\nu^2 + h)^{1/4}} \Omega(i/2).
 \end{split}
 \end{displaymath}
 We discuss the four terms on the right hand side of the trace formula. 
 
1)  The class number term gets immediately cancelled. 
 
2)  The $t$-integral in the polar term is rapidly decaying and by a Burgess-type   subconvexity bound $L(D, 1/2 + it) \ll |D|^{3/16+\varepsilon}(|1+|t|)^{10}$, say, its contribution to \eqref{I1} is at most\footnote{Of course, instead of subconvexity, we could also used mean value bounds on average over $D_2$ to get an even stronger saving.}  
 $$
 \preccurlyeq \frac{1}{K^{5/2}} \sum_{  D_2 \preccurlyeq K^2}        \sum_{h \preccurlyeq K^{1/2}}  \frac{1}{|D_2|^{1/16} (|D_2|\nu^2 + h)^{1/16  }} \preccurlyeq K^{-1/4}.$$  

3) For the diagonal term we observe that  $\sum_m m^{-1} \int |\Omega(t)W_{t}(rnvm)t| dt \preccurlyeq K$ uniformly in $n, r, v$ by \eqref{bound-wt}. 
%
If the fundamental discriminants underlying $D_2d^2$ and $-|D_2|\nu^2 - h$ coincide, we have  at most $\preccurlyeq K^{1/4}$ choices for $h$ (in most cases much fewer), so the diagonal term contributes at most 
 $$\preccurlyeq \frac{1}{K^{5/2}}  K  \sum_{  D_2\preccurlyeq K^2}         \frac{K^{1/4}}{|D_2|^{1/2}  } \preccurlyeq   K^{-1/4}$$
 to \eqref{I1}. 
 
4) It remains to deal with the fourth term, and to this end we  write $D_2 = \Delta_2 f_2$ with a fundamental discriminant $\Delta_2$. Then the Kloosterman term can be bounded by 
  \begin{equation*} 
\begin{split}
 \mathcal{I}_1^{\text{off}, \text{off}}(K)= &\frac{1}{K^{5/2}}      \sum_{h \preccurlyeq K^{1/2}}  \Big|   \sum_{  D_2  = \Delta_2 f_2^2}        \frac{ \chi_{D_2}(\delta) }{|D_2|^{1/4} (|D_2|\nu^2 + h)^{1/4}}   V_x\Big(\frac{|D_2|}{K^2}\Big) e\Big( -\frac{h\bar{\gamma}}{c}\Big)   \\
 &  \sum_{ d_1r\tau vw = f_2 }\sum_{n,  c, m}  \frac{\mu(d_1)\mu(v)  \chi_{\Delta_2}(d_1vm\tau)}{\sqrt{d_1rn\tau} vm}     \frac{  K_{3/2}^+(|\Delta_2|(vwn)^2, |D_2|\nu^2 +  h, c)}{ c} \\
&\quad\quad\quad \times \int_{-\infty}^{\infty}  \frac{F(4\pi vwn\sqrt{|\Delta_2|(|D_2|\nu^2 + h)}/c, t, 1/2)}{\cosh(\pi t)}   \Omega(t) W_{t}(  rnvm) t \frac{dt}{\pi}\Big|.
 \end{split}
 \end{equation*}
(Here we exchanged the roles of $D_1$ and $D_2$ in Theorem \ref{thm5}.) The general strategy is now as follows: We evaluate the $t$-integral by Lemma \ref{bessel-kuz} and simplify the expression by using suitable Taylor expansions. It is a lucky coincidence that this step yields rational phases in the exponentials. We are then ready to  apply Poisson summation in the long $\Delta_2$-sum which will eventually give enough savings. We now make these ideas precise. 

We recall that $\Omega$ satisfies the conditions stated in \eqref{prop-omega} and is negligible for $|t| \gg K^{1/2+\varepsilon}$. 
The $n, m$-sum is absolutely convergent by \eqref{bound-wt}, and we can truncate it at $rnvm  \preccurlyeq K^{1/2}$ at the cost of a negligible error. We split the $n, m$-sum into dyadic ranges $N \leq n \leq 2N$, $M \leq m \leq 2M$ where 
\begin{equation}\label{N-M}
   NM \preccurlyeq (rv)^{-1}K^{1/2}.
   \end{equation}
By the remark after Theorem \ref{thm5} and the properties of $\Omega$, the integrand of the $t$-integral is holomorphic in, say, $|\Im t| < 2/3$, so by contour shifts, Weil's bound \eqref{weil} and the power series expansion of the Bessel $J$-function we see that the $c$-sum is absolutely convergent and can be truncated at $c \ll K^{10^6}$, say, at the cost of a negligible error. Having truncated the $c$-sum in this very coarse way, we can sacrifice holomorphicity and include a smooth partition of unity into the $t$-integral, where a typical portion is weighted by $w(|t|/T)$ with a smooth compactly function $w$ localizing $|t| \asymp  T$ with $K^{1/2-2\eta} \ll T \ll K^{1/2+\varepsilon}$.   We apply Lemma \ref{bessel-kuz}a) to evaluate asymptotically the $t$-integral which in particular restricts the size of $c$. Splitting also the $c$-sum into dyadic ranges $c \asymp C$, we can assume, up to a negligible error, 
\begin{equation}\label{CC}
C \preccurlyeq \frac{vwn \sqrt{|\Delta_2|(|D_2|\nu^2 + h)}}{T^2} \preccurlyeq  \frac{KN}{d_1r\tau}.
\end{equation}
Having recorded these conditions, we can write the $t$-integral (as usual, up to a negligible error) as $$\sum_{\pm} \frac{T^2c^{1/2}}{\sqrt{vwn} |\Delta_2|^{1/4}(|D_2|\nu^2 + h)^{1/4}} e\Big(\pm \frac{2 vwn\sqrt{|\Delta_2|(|D_2|\nu^2 + h)}}{c}\Big) H^{\pm}\Big(  \frac{2 vwn\sqrt{|\Delta_2|(|D_2|\nu^2 + h)}}{c}\Big) $$
 for a flat function $H$, i.e.\ $x^j\frac{d^j}{dx^j} H^{\pm}(x) \ll_j 1$ for $j \in \Bbb{N}_0$.  Substituting back, it remains to bound
  \begin{equation}\label{145} 
\begin{split}
 & \frac{1}{K^{3/2}}      \sum_{h \preccurlyeq K^{1/2}}  \sum_{ d_1,r,\tau ,v,w   }  \sum_{\substack{n \asymp N\\ m \asymp M}} \sum_{c \asymp C} \Big|   \sum_{  \Delta_2 }        \frac{ \chi_{D_2}(\delta) }{|D_2|^{1/2} (|D_2|\nu^2 + h)^{1/2}}   V_x\Big(\frac{|D_2|}{K^2}\Big)       \frac{ \chi_{\Delta_2}(d_1vm\tau)}{nvm}   \\
 &  \frac{  K_{3/2}^+(|\Delta_2|(vwn)^2, |D_2|\nu^2 +  h, c)}{ c^{1/2}} e\Big(\pm \frac{2 vwn\sqrt{|\Delta_2|(|D_2|\nu^2 + h)}}{c}\Big)   H^{\pm}\Big(  \frac{2 vwn\sqrt{|\Delta_2|(|D_2|\nu^2 + h)}}{c}\Big)\Big|.
 \end{split}
 \end{equation}
where for given $r, v, d_1, \tau$ the parameters $N, M, C$ are subject to \eqref{N-M} and \eqref{CC} and $D_2 = \Delta_2 (d_1r\tau vw )^2$. Estimating trivially at this point using the Weil bound \eqref{weil}, we obtain the bound
\begin{equation}\label{tri}
\preccurlyeq \frac{1}{K^{3/2}} K^{1/2} C \preccurlyeq K^{1/2}.
\end{equation}
We see that applying the trace formula was a gambit in the sense that the trivial bound is now roughly a factor $K^{1/2}$ off our target. On the other hand, all automorphic information is now gone, and we may hope to get enough saving from the long  character sums. In particular, we can assume that $C \geq K^{1 - a\eta}$ for some sufficiently large constant $a$, otherwise the trivial bound \eqref{tri} suffices. For such $C$, we can use a Taylor expansion
 $$e\Big(\pm \frac{2 vwn\sqrt{|\Delta_2|(|D_2|\nu^2 + h)}}{c}\Big) = e\Big(\pm \frac{2 vwn|\Delta_2d_1r\tau vw\nu}{c} \pm \frac{vwnh}{d_1r\tau vw\nu c}\Big)\Phi(\Delta_2)$$
 with 
 $$\Phi(\Delta)  - 1  \ll \frac{nh^2}{c|D_2|} \preccurlyeq K^{-3/2}$$
in the current range of variables.  The error term contributes $\preccurlyeq K^{-1}$ to \eqref{145}. Similarly, we also have
$$\frac{H^{\pm}(  2 vwn\sqrt{|\Delta_2|(|D_2|\nu^2 + h)}/c) }{|D_2|^{1/2} (|D_2|\nu^2 + h)^{1/2}}  = \frac{H^{\pm}(  2 vwn |\Delta_2|d_1r\tau vw\nu/c)}{|D_2|\nu} + O\Big(\frac{h}{|D_2|^2}\Big)$$
and again the error term contributes $\preccurlyeq K^{-1}$ to \eqref{145}. Defining $\tilde{V}_x(z) = z^{-1} V_x(z)$, we are left with bounding
  \begin{equation*} 
\begin{split}
  \frac{1}{K^{7/2}}   &  \sum_{h \preccurlyeq K^{1/2}}  \sum_{ (d_1r\tau vw, \delta) = 1   }  \sum_{\substack{n \asymp N\\ m \asymp M}} \sum_{c \asymp C} \Big|   \sum_{  \Delta_2 }            \tilde{V}_x\Big(\frac{|\Delta_2|(d_1r\tau vw)^2}{K^2}\Big)     \frac{ \chi_{\Delta_2}(\delta d_1vm\tau)}{nvmc^{1/2}} \\
 &     K_{3/2}^+(|\Delta_2|(vwn)^2, |\Delta_2|(d_1r\tau vw\nu)^2 +  h, c)  e\Big(\pm \frac{2 (vw)^2n |\Delta_2|d_1r\tau \nu }{c}\Big)   H^{\pm}\Big(  \frac{2 (vw)^2n |\Delta_2|d_1r\tau  \nu}{c}\Big)\Big|.
 \end{split}
 \end{equation*}
Note the very fortunate fact that the algebraic phase $ e(\pm  2 vwn\sqrt{|\Delta_2|(|D_2|\nu^2 + h)}/c) $ in \eqref{145} has become a rational phase.  As usual, the $\Delta_2$-sum runs over negative fundamental discriminants, and we split the sum into residue classes $\Delta_2 \equiv 1,  5, 8, 9, 12, 13$ (mod 16)  and insert a factor $\mu^2(\Delta_2/\alpha)$ with $\alpha \in \{1, 4\}$ to detect squarefreeness. For notational simplicity let us treat the case of odd $\Delta$, the case of even $\Delta$ being similar. Using the well-known convolution formula for $\mu^2$, this leaves us with bounding  
  \begin{equation}\label{almostdone} 
\begin{split}
&  \frac{1}{K^{7/2}}     \sum_{h \preccurlyeq K^{1/2}}  \sum_{ (d_1r\tau vw, \delta) = 1   } \sum_{(d_2, \delta d_1vm\tau) = 1}  \sum_{\substack{n \asymp N\\ m \asymp M}} \sum_{c \asymp C} \frac{1}{vnmc^{1/2}}\\
& \Big|   \sum_{  \Delta_2 }   \psi(\Delta_2)         \tilde{V}_x\Big(\frac{|\Delta_2|(d_2d_1r\tau vw)^2}{K^2}\Big)     \Big( \frac{\Delta_2}{\delta d_1vm\tau}\Big)     K_{3/2}^+(|\Delta_2|(d_2vwn)^2, |\Delta_2|(d_2d_1r\tau vw\nu)^2 +  h, c) \\
& e\Big(\pm \frac{2 (vw)^2nd_2^2 |\Delta_2|d_1r\tau  \nu }{c}\Big)   H^{\pm}\Big(  \frac{2 (vwd_2)^2n |\Delta_2|d_1r\tau  \nu}{c}\Big)\Big|.
 \end{split}
 \end{equation}
 for a character $\psi$ modulo 4. Recall that the Kloosterman sum is nonzero only  if $4 \mid c$. Estimating trivially at this point (using \eqref{weil}), we can assume that 
 $$d_2d_1r\tau v w \preccurlyeq C/K \preccurlyeq K^{1/2}$$
 by \eqref{CC} and \eqref{N-M}, 
the remaining portion being $O(K^{-\eta})$ if the $K^{O(\eta)}$ factor in the previous $ \preccurlyeq$ sign is sufficiently large.  We now open the Kloosterman sum  and apply Poisson summation in $\Delta_2$ in residue classes modulo $ c \delta d_1vm\tau$. If ${\tt D}$ denotes the dual variable, this yields the character sum
\begin{equation}\label{charsum}
\begin{split}
&\sum_{\Delta_2 \, (\text{mod } c \delta d_1vm\tau)} \psi(\Delta_2) \Big( \frac{\Delta_2}{\delta d_1vm\tau}\Big)  e\Big(  \pm \frac{2 (vw)^2nd_2^2 |\Delta_2|d_1r\tau  \nu }{c}\Big)  \\
&  \sum_{\substack{\gamma\, (\text{mod }c)\\ (\gamma, c) = 1}} \epsilon^{2\kappa}_{\gamma} \left(\frac{c}{\gamma}\right) e\left(\frac{|\Delta_2|(d_2vwn)^2\gamma + (|\Delta_2|(d_2d_1r\tau vw\nu)^2 +  h)\bar{\gamma}}{c}\right)e\Big( \frac{\Delta_2 {\tt D}}{c \delta d_1vm\tau}\Big). 
\end{split}
\end{equation}
We write $c = c_1c_2$ where $(c_1, 2\delta d_1vm\tau) = 1$ and $c_2 \mid (2\delta d_1vm\tau)^{\infty}$. Then both sums split off a sum modulo $c_1$ given by
\begin{displaymath}
\begin{split}
&\sum_{\Delta_2 \, (\text{mod } c_1)}    e\Big(  \pm \frac{2 (vw)^2nd_2^2 |\Delta_2|d_1r\tau  \nu \bar{c}_2}{c_1}\Big)  \\
&  \sum_{\substack{\gamma\, (\text{mod }c_1)\\ (\gamma, c_1) = 1}}  \left(\frac{\gamma}{c_1}\right) e\left(\frac{(|\Delta_2|(d_2vwn)^2\gamma + (|\Delta_2|(d_2d_1r\tau vw\nu)^2 +  h)\bar{\gamma})\bar{c}_2}{c_1}\right)e\Big( \frac{\Delta_2 {\tt D} \overline{c_2 \delta d_1vm\tau}}{c_1  }\Big),
\end{split}
\end{displaymath}
cf.\ \cite[Lemma 2]{Iw0} for the treatment of the theta-multiplier. Summing over $\Delta_2$ bounds this double sum modulo $c_1$ by
\begin{displaymath}
\begin{split}
& \leq c_1\#\{ \gamma \in (\Bbb{Z}/c_1\Bbb{Z}_1)^{\ast} \mid (d_2vwn)^2\gamma + (d_2d_1r\tau vw\nu)^2  \bar{\gamma} \pm  2 (vw)^2nd_2^2 d_1r\tau  \nu + {\tt D} \overline{  \delta d_1vm\tau} = 0 \}\\
& \ll_{\varepsilon} c_1^{1+\varepsilon} (c_1, (d_2vwn, d_2d_1r\tau vw\nu)^2). 
\end{split}
\end{displaymath}
We estimate the remaining part of the character sum \eqref{charsum} trivially by $c_2^2\delta d_1vm\tau$. By the properties of the (essentially non-oscillating) weight functions $\tilde{V}_x$ and $H^{\pm}$, the dual variables ${\tt D}$ can be truncated at
$${\tt D} \preccurlyeq  \frac{ c \delta d_1vm\tau}{K^2/(d_2d_1r\tau vw)^2},$$
and so the $\Delta_2$-sum in \eqref{almostdone} can be bounded by 
\begin{displaymath}
\begin{split}
& \preccurlyeq \Big(\frac{K^2/(d_2d_1r\tau vw)^2}{ c \delta d_1vm\tau} + 1\Big) c^{1+\varepsilon} (c_1, (d_2vwn, d_2d_1r\tau vw\nu)^2) c_2\delta d_1vm\tau\\
& = \Big( \frac{K^2c^{\varepsilon}}{  (d_2d_1r\tau vw)^2  }     + c^{1+\varepsilon}   \delta d_1vm\tau\Big) \Big(c,  (d_2w)^2( n,   r   \nu)^2 (2\delta d_1vm\tau)^{\infty}\Big) 
\end{split}
\end{displaymath}
using the notation explained in Section \ref{15}. The first term in the first parenthesis accounts for the zero frequency in the Poisson summation formula. We substitute this back into \eqref{almostdone} getting the (generous) upper bound
  \begin{equation}\label{generous} 
  \begin{split}
\preccurlyeq &  \frac{1}{K^{3}}      \sum_{d_2   d_1r\tau v w\preccurlyeq K^{1/2}   }  \sum_{\substack{n \asymp N\\ m \asymp M}} \sum_{c \asymp C} \frac{ (c,  (2d_2   d_1r\tau v wnm)^{\infty})}{vnmc^{1/2}}\Big( \frac{K^2c^{\varepsilon}}{  (d_2d_1r\tau vw)^2  }     + c^{1+\varepsilon}    d_1vm\tau\Big).  
 \end{split}
 \end{equation}
Here we dropped the variable $\delta \preccurlyeq 1$. The rest is book-keeping. By Rankin's trick we have
$$\sum_{c \asymp C} (c, x^{\infty}) \ll C\sum_{\substack{ d \ll C\\ d \mid x^{\infty}}} 1 \ll C\sum_{ d \mid x^{\infty}} \Big( \frac{C}{d} \Big)^{\varepsilon} \ll C(Cx)^{\varepsilon}$$
 for every $\varepsilon > 0$ and $x \in \Bbb{N}$.  Using \eqref{CC} and \eqref{N-M},  the    bound \eqref{generous} becomes
  \begin{equation*} 
  \begin{split}
&\preccurlyeq   \frac{1}{K^{3}}       \sum_{d_2   d_1r\tau v w\preccurlyeq K^{1/2}   }   \frac{C^{1/2}}{v }\Big( \frac{K^2 }{  (d_2d_1r\tau vw)^2  }     + C M  d_1v\tau\Big) \\
&\preccurlyeq  \frac{1}{K^{3}}       \sum_{d_2   d_1r\tau v w\preccurlyeq K^{1/2}   }   \frac{(KN)^{1/2}}{(d_1r\tau)^{1/2}v }\Big( \frac{K^2 }{  (d_2d_1r\tau vw)^2  }     +\frac{ KN M   v }{r}\Big)\\
& \preccurlyeq  \frac{1}{K^{3}}       \sum_{d_2   d_1r\tau v w\preccurlyeq K^{1/2}   }   \frac{K^{3/4}}{(d_1r^2v\tau)^{1/2} }\Big( \frac{K^2 }{ v (d_2d_1r\tau vw)^2  }     +\frac{ K^{3/2}   }{r^2v}\Big) \preccurlyeq K^{-1/4}.
 \end{split}
 \end{equation*}
This is the desired power saving and completes the proof of Theorem \ref{thm1}.

\appendix
\section{ A period formula on average}\label{appb}

In order to verify the constant $1/4$ in the period formula \eqref{katok-Sarnak}, we  take a large parameter $T$, a very small $\varepsilon > 0$ and consider the two averages
$$A_1(T) = \sum_{u} \frac{L(u, 1/2) L(u \times \chi_{\Delta}, 1/2)}{L(\text{sym}^2 u, 1)} h_{T, \varepsilon}(t_u)$$
and
$$A_2(T) =\sum_{u} \frac{|P(\Delta, u)|^2}{\sqrt{|\Delta|}} h_{T, \varepsilon}(t_u)$$
for a (fixed) negative fundamental discriminant $\Delta$ with class number 1 (for simplicity) and
$$ h_{T, \varepsilon}(t) = \frac{t^2 + 1/4}{T^2} \exp\Big(-\Big(\frac{t - T}{T^{1-\varepsilon}}\Big)^2 - \Big(\frac{t + T}{T^{1-\varepsilon}}\Big)^2\Big).$$
The computation is relatively standard, so we can be brief. 
Using the same approximate functional equation as in \eqref{approx-basic} we have
$$L(u, 1/2) = 2 \sum_n \frac{\lambda_u(n)}{n^{1/2}} W^+_{t_u}(n), \quad L(  u \times \chi_{\Delta}, 1/2) = 2 \sum_m \frac{\lambda_u(m) \chi_{\Delta}(m)}{m^{1/2}} W^-_{t_u}\Big(\frac{m}{|\Delta|}\Big)$$
for even $u$ with $W_t$ as in \eqref{v-t}. 
For odd $u$, each summand in $A_1(T)$ vanishes. 
This gives
$$A_1(T ) = 4 \sum_{nm} \frac{\chi_{\Delta}(m)}{\sqrt{nm}} \sum_{u \text{ even}} \frac{\lambda_u(n)  \lambda_u(m)}{L(\text{sym}^2 u, 1)} V_{t_u}(n) W_{t_u}(m) h_{T, \varepsilon}(t_u).$$
To make this spectrally complete, we artificially add the corresponding Eisenstein contribution
$$ 4 \sum_{nm} \frac{\chi_{\Delta}(m)}{\sqrt{nm}} \int_{-\infty}^{\infty}  \frac{\rho_{it}(n)  \rho_{-it}(m)}{|\zeta(1 + 2it)|^2} W^+_{t}(n) W^-_t\Big(\frac{m}{|\Delta|}\Big) h_{T, \varepsilon}(t) \frac{dt}{2\pi} .$$
Using \eqref{approx-basic1}, this can be written in terms of moments of the Riemann zeta function and Dirichlet $L$-functions. By standard mean value bounds the contribution is $O(T^{1+\varepsilon})$ (recall that $\Delta$ is fixed). We apply the Kuznetsov formula for the even spectrum given in Lemma \ref{kuz-even}. In this way we get a main term
\begin{displaymath}
\begin{split}
&4 \sum_{nm} \frac{\chi_{\Delta}(n)}{n} \int_{-\infty}^{\infty}  W^+_{t}(n) W^-_{t}(n) h_{T, \varepsilon}(t) t \tanh(\pi t) \frac{dt}{4\pi^2} \\
&= \int_{-\infty}^{\infty} h_{T, \varepsilon}(t)
\int_{(2)} \int_{(2)} \prod_{\pm}  \frac{\Gamma(1/2 + s_1 \pm it_u)}{\Gamma(1/2 \pm it_u) \pi^{s_1} s_1} \frac{\Gamma(3/2 + s_2 \pm it_u)}{\Gamma(3/2 \pm it_u) \pi^{s_2} s_2}\\
&\quad\quad\quad |\Delta|^{s_2}  L( \chi_{\Delta}, 1 + s_1 + s_2) 
    e^{s_2^2 +s_1^2} \frac{ds_1 \, ds_2}{(2\pi i)^2}
  t \tanh(\pi t) \frac{dt}{\pi^2}.
  \end{split}
  \end{displaymath}
  We shift the $s_1, s_2$-contour to real part $-1/4$, obtaining
  \begin{equation}\label{A1}
  L( \chi_{\Delta}, 1) \int_{-\infty}^{\infty} h_{T, \varepsilon}(t)t \tanh(\pi t) \frac{dt}{\pi^2} + O(T^{7/4 + \varepsilon})
  \end{equation}
with main term $\gg T^{2-\varepsilon}$. It remains to deal with the off-diagonal terms in \eqref{kuz-even-form} and we briefly sketch why both of them are negligible. By \eqref{bound-wt} we can restrict $n, m \ll T^{1+\varepsilon}$ at the cost of a negligible error. The first off-diagonal term contributes a term of the shape
$$ \sum_{nm} \frac{\chi_{\Delta}(m)}{\sqrt{nm}} \sum_c \frac{S(n, m, c)}{c}\int_{-\infty}^{\infty}  J_{2it}\Big(4\pi  \frac{\sqrt{nm}}{c}\Big)    W^+_{t}(n) W^-_{t}(m) h_{T, \varepsilon}(t_u)  t \, \frac{dt}{\cosh(\pi t)}.$$
By Lemma \ref{bessel-kuz}, the $t$-integral is negligible unless $c \leq \sqrt{nm} T^{-2+\varepsilon}$ which does not happen for $n, m \ll T^{1+\varepsilon}$.   The second off-diagonal term contributes a term of the shape
$$ \sum_{nm} \frac{\chi_{\Delta}(m)}{\sqrt{nm}} \sum_c \frac{S(n, m, c)}{c}\int_{-\infty}^{\infty} K_{2it}\Big(4\pi  \frac{\sqrt{nm}}{c}\Big) V_{t}(n) W_{t}(m) h_{T, \varepsilon}(t_u) \sinh(\pi t) t \, dt.$$
Again by  Lemma \ref{bessel-kuz},  up to a small error the $t$-integral is negligible unless $c \leq \sqrt{nm} T^{-1+\varepsilon}$ in which case it is essentially  non-oscillating in $n, m$, so we can restrict to $n, m = T^{1+o(1)}$, $c \ll T^{\varepsilon}$. Poisson summation in the $m$-variable now shows that the entire expression is negligible, since $\chi_{\Delta}$ is a   non-trivial character. 

We continue with the analysis of $A_2(T)$. Since we assume class number 1, we have
$$A_2(T) =  \frac{1}{\sqrt{\Delta} \epsilon_{\Delta}^2} \sum_u |u(z_{\Delta})|^2 h_{T, \varepsilon}(t_u)$$
where $z_{\Delta}$ is the unique Heegner point (modulo $\Gamma$) and $\epsilon_{\Delta} \in \{1, 2, 3\}$ is half the number of roots of unity in $\Bbb{Q}(\sqrt{\Delta})$. We artificially add the constant function and the Eisenstein series at a cost of $O(T)$ and apply the pre-trace formula for the entire spectrum--for odd $u$ we have $u(z_{\Delta}) = 0$ automatically. This gives
$$A_2(T) = \frac{1}{\sqrt{\Delta} \epsilon_{\Delta}^2} \sum_{\gamma \in \overline{\Gamma}} k(\gamma z_{\Delta}, z_{\Delta}) + O(T)$$
where
$$k(z, w) =\frac{1}{4\pi} \int_{-\infty}^{\infty} F(1/2 + it, 1/2 - it, 1, -v) h_{T, \varepsilon}(t) \tanh(\pi t) t\, dt, \quad v = v(z, w) = \frac{|z-w|^2}{4\Im z \Im w}. $$
The stabilizer of $z_{\Delta}$ contributes
\begin{equation}\label{A2}
\frac{1}{\sqrt{\Delta} \epsilon_{\Delta}}\cdot  \frac{1}{4\pi} \int_{-\infty}^{\infty} h_{T, \varepsilon}(t) \tanh(\pi t) t\, dt. 
\end{equation}
It is easy to see that $u(\gamma z, z) \leq \delta$ implies $\| \gamma \| \ll \sqrt{\delta + 1}$, cf.\ e.g. \cite[(A.7) with $n=1$]{IS}. From \cite[(1.64)]{Iw3} we see that $k(z, w)$ is negligible as soon as $u \gg T^{\varepsilon - 2}$, so that the contribution of all matrices not in the stabilizer is negligible. Combining \eqref{A1}, \eqref{A2} and the class number formula in the case $h_{\Delta} = 1$, we obtain
$$A_2(T) \sim \frac{1}{4} A_1(T), \quad T \rightarrow \infty$$
in accordance with \eqref{katok-Sarnak}.

\section{A Dirichlet series with Hurwitz class numbers}\label{appa}

The aim of this section is an analysis of the $L$-function
$$L_+(s, a/c) = \sum_{D < 0} \frac{H(D)e(a|D|/c)}{|D|^{1/4 + s}}$$
for $4 \mid c$, $(a, c) = 1$. As before, $H(D)$ denotes the Hurwitz class number, and the series converges absolutely in $\Re s > 5/4$. 
The results are probably known to specialists, but do not seem to be in the literature and may be of independent interest.  We recall the notation \eqref{epsd}.

\begin{lemma}\label{class-num}
Let $c > 0$, $4 \mid c$, $(a, c) = 1$.   The Dirichlet series $L_+(s, a/c)$ has meromorphic continuation to all $\Bbb{C}$. It has two simple poles  at $s = 5/4$, $s = 3/4$  (and no other poles) with residues
 \begin{displaymath}
 \begin{split}
& \underset{s = 5/4}{\text{{\rm res}}} L_+(s, a/c) =- \frac{\sqrt{2} \pi}{3c^{3/2}} \left( \frac{-c}{a}\right) \bar{\epsilon}_a e(3/8), \quad \underset{s = 3/4}{\text{{\rm res}}} L_+(s, a/c)  =    \frac{1}{\sqrt{8}c^{1/2}} \left( \frac{-c}{a}\right) \bar{\epsilon}_a  e(3/8).
 \end{split}
 \end{displaymath}
\end{lemma}
  
  \begin{proof} We recall the definition of $\mathcal{H}(z)$ in \eqref{zag} and compute 
  \begin{equation}\label{defI}
  \mathcal{I}(s, a/c) := \int_0^{\infty} \Big(\mathcal{H}(a/c + iy) -  (c_1 y^{3/4} +c_2y^{1/4})\Big)y^{s-1/2} \frac{dy}{y} 
  \end{equation}
  with $c_1 = - (4\pi)^{3/4}/12$, $c_2 = 1/(\sqrt{8} \pi^{1/4})$ in two ways. Let
  $$L_-(s, a/c) = \frac{1}{4 \sqrt{\pi}} \sum_{n=1}^{\infty} \frac{e(-an^2/c)}{ n^{2s-1/2}}.$$
This $L$-function is obviously holomorphic in $\Re s > 3/4$. It has a simple pole at $s=3/4$ with residue
\begin{equation}\label{resl-}
\underset{s=3/4}{\text{res}}L_-(s, a/c)  =\frac{1}{4 \sqrt{\pi}}  \frac{1}{2c}  \sum_{n \,(\text{mod } c)} e\Big(-\frac{an^2}{c}\Big) = \frac{1}{4 \sqrt{\pi}} (1+i) \bar{\epsilon}_{-a} \left( \frac{c}{-a}\right) \frac{1}{2\sqrt{c}},
\end{equation}
and it   has a functional equation 
\begin{equation}\label{functheta}
  L_-(s, a/c) = \Big(\frac{c}{2\pi}\Big)^{1-2s} \left( \frac{-c}{-a}\right)  \epsilon_{-a} e(1/8) \frac{\Gamma(  3/4-s)}{\Gamma(s - 1/4)} L_-(1-s, -\bar{a}/c).
\end{equation}
This follows from the corresponding properties of Hurwitz zeta function and standard computations with Gau{\ss} sums (or the transformation behaviour of one-dimensional  theta series). In particular we obtain  the  analytic continuation of $L_-(s, a/c)$ to all of $\Bbb{C}$ with only a simple pole  at $s=3/4$. 
 Moreover, $L_-(s, a/c) = 0$ for $s = 1/4 - n$, $n = 1, 2, 3 \ldots$.

Returning to \eqref{defI},  we have
\begin{equation}\label{I}
\mathcal{I}(s, a/c) = L_+(s, a/c) G_{3/4}(s) + L_-(s, a/c) G_{-3/4}(s)
\end{equation}
where
$$G_{c}(s) = \int_0^{\infty} W_{c, 1/4}(4\pi y) y^{s-1/2} \frac{dy}{y}.$$  
Since $W_{3/4, 1/4}(x) = e^{-x/2} x^{3/4}$, we can explicitly compute
\begin{equation}\label{g+}
G_{3/4}(s) =  2^{s+1/4}(4\pi)^{1/2 - s}  \Gamma(s + 1/4).
\end{equation}
For  the analysis of $G_{-3/4}$, we combine \cite[7.621.3, 9.131.1, 9.111]{GR} to obtain
\begin{displaymath}
\begin{split}
G_{-3/4}(s) &= (4\pi)^{1/2 - s} \frac{\Gamma(s -1/4)\Gamma(s + 1/4) }{\Gamma(s + 5/4)}F(s + 1/4, s-1/4, s + 5/4; 1/2) \\
&=    (4\pi)^{1/2 - s} \Gamma(s -1/4)\int_0^1 t^{s-3/4}(1 - t/2)^{1/4 - s} dt.
\end{split}
\end{displaymath}
In particular, by repeated partial integration in the $t$-integral we see that $G_{-3/4}(s)$ is meromorphic with simple poles at most at $s= (2n+1)/4$, for integers $n \leq 0$. 
From \cite[9.121.24]{GR} we get
\begin{equation}\label{g-}
G_{-3/4}(3/4) = 2 (\sqrt{2} - 1)\pi^{1/4}.
\end{equation}

On the other hand, we may complete the pair $(a, c)$ to a matrix $(\begin{smallmatrix} a & b\\ c & d\end{smallmatrix})\in\Gamma_0(4)$. Then
$$\mathcal{H}(z) = \mathcal{H}\left( \frac{dz - b}{-cz + a}\right) \left( \frac{-c}{a}\right) \bar{\epsilon}_a \left( \frac{-cz + a}{|-cz + a|}\right)^{-3/2},$$
in particular
$$\mathcal{H}\left( \frac{a}{c} + iy\right) = \mathcal{H}\left( - \frac{d}{c} + \frac{i}{c^2 y}\right)  \left( \frac{-c}{a}\right) \bar{\epsilon}_a e(3/8).$$
Splitting the integral in \eqref{defI} into $\int_0^{1/c}$ and $\int_{1/c}^{\infty}$ and applying the functional equation in the former, we obtain
\begin{equation}\label{funcI}
\begin{split}
\mathcal{I}(s, a/c)= & c^{1-2s} \left( \frac{-c}{a}\right) \bar{\epsilon}_a e(3/8)  \int_{1/c}^{\infty}   \Big(\mathcal{H}\left( - \frac{d}{c} +  iy \right)  - (c_1 y^{3/4} + c_2 y^{1/4})\Big)y^{-s+1/2} \frac{dy}{y} \\
&- \frac{c_1 c^{-s-1/4}}{s +1/4} -\frac{c_2 c^{-s+1/4}}{s - 1/4} + c^{1-2s} \left( \frac{-c}{a}\right) \bar{\epsilon}_a e(3/8) \Big( \frac{c_1c^{s-5/4}}{s-5/4} + \frac{c_2c^{s-3/4}}{s-3/4}\Big)\\
  & +   \int_{1/c} ^{\infty} \Big(\mathcal{H}(a/c + iy) -  (c_1 y^{3/4} + c_2 y^{1/4})\Big)y^{s-1/2} \frac{dy}{y} .
\end{split}
\end{equation}
This establishes the analytic continuation of $\mathcal{I}(s, a/c)$ to all of $\Bbb{C}$ except for poles at $s =  5/4, 3/4, 1/4, -1/4$.  From the preceding analysis we conclude the meromorphic continuation of $L^+(s, a/c)$ as a function of finite order with possible poles at most at $5/4, 3/4, 1/4$. Since $L_{-}(s, a/c) G_{-3/4}(s)$ is holomorphic at $s = 5/4$, we obtain the formula for the residue at $s = 5/4$ from \eqref{I}, \eqref{g+} and \eqref{funcI}. 
Using in addition  \eqref{resl-}, \eqref{g-}, we obtain
\begin{displaymath}
\begin{split}
\underset{s = 3/4}{\text{res}} L_+(s, a/c)&  =  \frac{\pi^{1/4}}{\sqrt{2}c^{1/2}} \left( \left( \frac{-c}{a}\right) \bar{\epsilon}_a e(3/8) \frac{1}{\sqrt{8}\pi^{1/4}} - (\sqrt{2} - 1)\pi^{1/4} \frac{1}{4 \sqrt{\pi}} (1+i) \bar{\epsilon}_{-a} \left( \frac{c}{-a}\right)\right) \\
\end{split}
\end{displaymath}
which confirms the residue formula at $s = 3/4$,  
since $\bar{\epsilon}_{-a}(\frac{c}{-a})  = (-i) (\frac{-c}{a}) \bar{\epsilon}_a$. Similarly, using also \eqref{functheta} and the simple formula $F(1/2, 0, 3/2; 1/2) = 1$, we get
\begin{displaymath}
\begin{split}
\underset{s = 1/4}{\text{res}} L_+(s, a/c)&  = \frac{1}{2\pi^{3/4}} \Big( - \frac{1}{\sqrt{8}\pi^{1/4}} - \sqrt{8} \pi^{1/4} L_-(1/4, a/c)\Big) \\
&=  \frac{1}{2\pi^{3/4}} \Big( - \frac{1}{\sqrt{8}\pi^{1/4}} + \sqrt{8} \pi^{1/4}\Big( \frac{c}{2\pi}\Big)^{1/2} \left( \frac{-c}{-a}\right) \epsilon_{-a} e(1/8)\sqrt{\pi}  \frac{1}{4 \sqrt{\pi}} (1+i) \bar{\epsilon}_{d} \left( \frac{c}{d}\right) \frac{1}{2\sqrt{c}} \Big) 
\end{split}
\end{displaymath}
with $d \equiv -\bar{a}$ (mod $c$), which vanishes. \end{proof}

\emph{Remark:} One can show that away from the two poles, the function $L^+(s, a/c)$ satisfies the growth condition $L^+(s, a/c) \ll_{\Re s} ((1 + |s|) c)^{\max(0, 5/4 - \Re s, 1 - 2\Re s) + \varepsilon}.$\\

\section{A volume computation}\label{appc}

In this appendix we justify \eqref{vol} for a Saito--Kurokawa lift $F \in S_k^{(2)}$ associated with a Hecke eigenform $f \in S_{2k-2}$. Following \cite[Section 2]{Bl}, we write its Fourier expansion at $Z = iY$ as
$$F(iY) = \sum_{T\in \mathcal{P}(\Bbb{Z}) } \alpha(T) (\det 2T)^{\frac{k}{2} - \frac{3}{4}} e^{-2\pi \text{tr}(TY)}, $$
normalized such that $\alpha(T)^2 = L(f \times \chi_{-\det 2T}, 1/2)$ if $-\det 2T$ is a fundamental discriminant. In this case, $\| F \| = (2\pi)^{-k}\Gamma(k) k^{-1/4 + o(1)}$ by \cite[(2.8)]{Bl}. We conclude that
$$\mathcal{F}(Y) := \frac{(\det Y)^{k/2}F(iY) }{\| F \|_2}=k^{1/4 + o(1)} \sum_{T\in \mathcal{P}(\Bbb{Z}) } \frac{\alpha(T)}{\det(2T)^{3/4}} \frac{(4\pi)^{k} (\det TY)^{\frac{k}{2}} e^{-2\pi \text{tr}(TY)}}{\Gamma(k)}. $$
The function $$X \mapsto \frac{(4\pi)^k (\det X)^{k/2} e^{-2\pi \text{tr}X}}{\Gamma(k)}$$ is invariant under conjugation, and for a diagonal matrix $X = \text{diag}(x_1, x_2)$ it is negligible unless $x_1, x_2 = k/4\pi + O(\sqrt{k}\log k)$, cf. \cite[Section 4]{Bl}. For large $k$, we conclude that $\mathcal{F}(Y)$ is negligible   unless there exists $T \in \mathcal{P}(\Bbb{Z})$ such that the two eigenvalues of $TY$ are $k/4\pi + O(\sqrt{k}\log k)$. In particular, the essential support of $\mathcal{F}$ can be restricted to matrices $Y$ whose maximal eigenvalue $\lambda_{\max}(Y)$ satisfies $\lambda_{\max}(Y) \ll k$.  Since $|\mathcal{F}|$ is invariant under $Y \mapsto Y^{-1}$, its minimal eigenvalue $\lambda_{\min}(Y)$ must satisfy $\lambda_{\min}(Y) \gg 1/k$. The implied constants will not be relevant. As in \eqref{product}, we write 
$$Y  = Y(r, x, y) = \sqrt{r}\left(\begin{matrix} y^{-1} & -xy^{-1}\\-x y^{-1} & y^{-1} (x^2 + y^2)\end{matrix}\right)$$
with $x + iy \in \Gamma \backslash \Bbb{H}$. Let $\mathcal{Y}$ be the set of such matrices $Y$ with $1/k \ll \lambda_{\min}(Y) \leq \lambda_{\max}(Y) \ll k$. 
Without loss of generality we may assume that $Y$ is Minkowski-reduced, equivalently $x+iy$ is in the standard fundamental domain $|x|\leq 1/2$, $x^2 + y^2 \geq 1$. The two eigenvalues of $Y$ are given by
$$\sqrt{r} \frac{1 + x^2 + y^2 \pm \sqrt{(1+x)^2 + 2(x^2 - 1)y^2 + y^4}}{2y}.$$
Thus we have 
\begin{displaymath}
\begin{split}
\text{vol}(\mathcal{Y})& = \underset{Y(r, x, y) \in \mathcal{Y}}{\int_{-1/2}^{1/2}   \int_{\sqrt{1 - x^2}}^{\infty} \int_{\Bbb{R}} } \frac{dr}{r} \frac{dy \, dx}{y^2}= \int_{-1/2}^{1/2}  \int_{\sqrt{1 - x^2}}^{\infty}  \big(4\log (k/y) + O(1) \big) \frac{dy \, dx}{y^2} \\
&= \text{vol}({\rm SL}_2(\Bbb{Z}) \backslash \Bbb{H}) \cdot 4 \log k + O(1),
\end{split}
\end{displaymath}
as desired. 


\end{document}